\pdfoutput=1 

\documentclass[a4paper,twoside,leqno]{article}

\usepackage{a4wide}

\usepackage[p,osf]{scholax}
\usepackage{amsmath,amsthm}
\usepackage[scaled=1.075,ncf,vvarbb]{newtxmath}

\usepackage{microtype}

\usepackage[only,llbracket,rrbracket]{stmaryrd}

\usepackage{graphicx}

\usepackage{enumitem}
\setlist[enumerate,1]{label=(\alph*)}
\setlist[enumerate,2]{label=(\roman*)}

\usepackage[T1]{fontenc}
\usepackage[utf8x]{inputenc}
\usepackage[usenames,dvipsnames]{xcolor}

\usepackage{url}
\usepackage[hidelinks,colorlinks=true,linkcolor=blue,citecolor=darkgray,linktocpage=false]{hyperref}

\usepackage{mathtools}
\mathtoolsset{showonlyrefs}

\usepackage{dsfont}

\newtheoremstyle{citing}
  {3pt}
  {3pt}
  {\itshape}
  {}
  {\bfseries}
  {.}
  {.5em}
  {\thmnote{#3}}
\theoremstyle{citing}

\theoremstyle{definition}

\theoremstyle{remark}

\swapnumbers
\theoremstyle{plain}

\newtheorem{theorem}{Theorem}[section]

\newtheorem{lemma}[theorem]{Lemma}
\newtheorem{proposition}[theorem]{Proposition}
\newtheorem{corollary}[theorem]{Corollary}

\theoremstyle{remark}
\newtheorem{remark}[theorem]{Remark}
\newtheorem*{remark*}{Remark}
\newtheorem{example}[theorem]{Example}

\theoremstyle{definition}
\newtheorem{definition}[theorem]{Definition}
\newtheorem*{definition*}{Definition}
\newtheorem{miniremark}[theorem]{}
\newtheorem{question}[theorem]{Question}
\newtheorem*{question*}{Question}

\theoremstyle{definition}
\swapnumbers

\usepackage{chngcntr}

\counterwithin*{equation}{theorem}
\counterwithin*{enumi}{theorem}

\newcommand{\End}[1]{ \mathrm{End}({#1}) }
\newcommand{\GF}[1]{ \mathcal{G}_{#1} }

\newcommand{\ccspace}[1]{\mathscr{K}(#1)}

\newcommand{\project}[1]{{#1}_\natural}

\newcommand{\eqproject}[1]{(#1)_\natural}

\newcommand{\perpproject}[1]{{#1}^\perp_\natural}

\newcommand{\grass}[2]{\mathbf{G}(#1,#2)}

\newcommand{\pgrass}[2]{\mathbf{G}_{\natural}(#1,#2)}

\newcommand{\ograss}[2]{\mathbf{G}_0(#1,#2)}

\newcommand{\AllP}{\mathbf{P}}

\newcommand{\Proj}[2]{\AllP(#1,#2)}

\newcommand{\Var}[1]{\mathbf{V}_{#1}}     

\newcommand{\adim}{n}
\newcommand{\vdim}{k}
\newcommand{\codim}{n-k}

\newcommand{\cball}[2]{\mathbf{B}(#1,#2)}

\newcommand{\sphere}[1]{\mathbb{S}^{#1}}

\ifdefined\textint
    \renewcommand{\textint}[2]{{\textstyle\int_{#1}^{#2}}}
\else    
    \newcommand{\textint}[2]{{\textstyle\int_{#1}^{#2}}}
\fi

\ifdefined\textfint
    \renewcommand{\textfint}[2]{{\textstyle\fint_{#1}^{#2}}}
\else    
    \newcommand{\textfint}[2]{{\textstyle\fint_{#1}^{#2}}}
\fi

\ifdefined\textsum
    \renewcommand{\textsum}[2]{{\textstyle\sum_{#1}^{#2}}}
\else  
    \newcommand{\textsum}[2]{{\textstyle\sum_{#1}^{#2}}}
\fi

\ifdefined\textint
    
\else    
    
\fi

\newcommand{\natp}{\mathscr{P}}

\newcommand{\R}{\mathbf{R}}

\newcommand{\LM}{\mathscr{L}}

\newcommand{\Lpnorm}[3]{{#1}_{({#2})}({#3})}

\newcommand{\HM}{\mathscr{H}}

\newcommand{\density}{\boldsymbol{\Theta}}

\newcommand{\restrict}{ \mathop{ \rule[0pt]{.5pt}{6pt} \rule[0pt]{4pt}{0.5pt} }\nolimits }

\newcommand{\ud}{\ensuremath{\,\mathrm{d}}}

\newcommand{\uD}{\ensuremath{\mathrm{D}}}

\DeclareMathOperator{\extreme}{extreme}
\DeclareMathOperator{\exposed}{exposed}

\newcommand{\id}[1]{\mathds{1}_{#1}}

\newcommand{\lIm}{[}
\newcommand{\rIm}{]}

\newcommand{\scale}[1]{\boldsymbol{\mu}_{#1}}
\newcommand{\trans}[1]{\boldsymbol{\tau}_{#1}}

\newcommand{\tbwedge}{{\textstyle\boldsymbol{\bigwedge}}}

\newcommand{\aforms}[1]{\tbwedge^{\!{#1}}}

\newcommand{\extpower}[1]{\tbwedge_{{#1}}}

\newcommand{\adjoint}[1]{\aforms{1} {#1}}

\newcommand{\transpose}[1]{ #1^{*} }

\newcommand{\polar}[1]{{#1}^{\circ}}

\newcommand{\Clos}[1]{\mathop{\mathrm{Clos}}#1}

\newcommand{\Dirac}[1]{ \boldsymbol{\delta}_{#1} }

\newcommand{\VF}{\mathscr{X}}

\DeclareMathOperator{\aff}{af{}f}

\DeclareMathOperator{\trace}{trace}

\DeclareMathOperator{\Hom}{Hom}

\DeclareMathOperator{\Bdry}{Bdry}

\DeclareMathOperator{\lin}{span}

\newcommand{\cnt}[1]{\mathscr{C}^{#1}}

\newcommand{\orthgroup}[1]{\mathbf{O}({#1})}

\DeclareMathOperator{\Int}{Int}

\DeclareMathOperator{\rInt}{relint}

\DeclareMathOperator{\dmn}{dmn}

\DeclareMathOperator{\grad}{grad}

\DeclareMathOperator{\subgrad}{subgr}

\DeclareMathOperator{\spt}{spt}

\DeclareMathOperator{\Sp}{space}

\DeclareMathOperator{\conv}{conv}

\DeclareMathOperator{\ray}{ray}

\DeclareMathOperator{\cone}{cone}

\DeclareMathOperator{\Tan}{Tan}

\DeclareMathOperator{\Nor}{Nor}

\DeclareMathOperator{\Lip}{Lip}

\newcommand{\without}{\!\sim\!}

\newcommand{\QEC}{\textrm{QEC}}
\newcommand{\EC}{\textrm{EC}}
\newcommand{\SAC}{\textrm{SAC}}
\newcommand{\SEC}{\textrm{SEC}}
\newcommand{\CI}{\textrm{CI}}
\newcommand{\USAC}{\textrm{USAC}}
\newcommand{\AC}{\textrm{AC}}

\newcommand{\imsp}{\mathcal{I}}

\DeclareMathOperator{\im}{im}

\date{\today}

\title{Geometry of the Atomic Condition}

\author{
  Jacek Jakimiuk
  \and
  Sławomir Kolasiński
  \and
  Maciej Leśniak
}

\begin{document}

\maketitle


\begin{abstract}
    The \emph{atomic condition}~\AC{} of De~Philippis, De~Rosa, and~Ghiraldin lies at the heart of
    anisotropic theory for varifolds. Recent results show it plays a fundamental role in the
    validity of the Michael--Simon inequality and an Allard-type regularity theorem. Yet, there are
    very few examples, and checking the condition seems hard. This work provides insight into the
    geometric structure of~\AC{} and its connection with convexity. First we reformulate \AC{} in
    terms of convex geometry and study properties of the related convex sets. In the second part, we
    define a~quantitative version~\QEC{} and prove it governs the validity of a~Caccioppoli-type
    inequality and is stable under perturbations of class~$\cnt{1,1}$. Finally, for all $\vdim \ge
    2$ and $\adim - \vdim \ge 3$ we construct strictly polyconvex integrands, associated to
    inner-product norms on~$\extpower{\vdim} \R^{\adim}$, which already violate the rank condition
    (AC1) of the atomic condition. Hence, strict polyconvexity, which is necessary for~\AC{} by
    De~Rosa, Lei, and Young, is not sufficient for~\AC{} in this range.
\end{abstract}

\tableofcontents

\section{Introduction}
\label{sec:intro}

Let $F : \grass{\adim}{\vdim} \to \R \cap \{ t : t > 0 \}$ be a~continuous integrand on the
Grassmannian of unoriented $\vdim$-dimensional planes in~$\R^{\adim}$ and consider the functional
\begin{displaymath}
    \Phi_F(S) = \textint{S}{} F \bigl( \Tan^{\vdim}(\HM^{\vdim} \restrict S, x) \bigr) \ud
    \HM^{\vdim}(x)
\end{displaymath}
defined whenever $S \subseteq \R^{\adim}$ is $\HM^{\vdim}$~measurable and
$(\HM^{\vdim},\vdim)$~rectifiable. A~Plateau-type problem asks for the existence and regularity of
minimisers of~$\Phi_F$ in a~class of competitors spanning a~given boundary; the classical frameworks
go back to Reifenberg~\cite{Reifenberg1960}, Federer and Fleming~\cite{Federer1960}, and
Almgren~\cite{Almgren1968}. Already for the existence of rectifiable minimisers one must impose an
ellipticity condition on~$F$. Almgren's ellipticity~\cite[1.2]{Almgren1968}
(cf.~\cite[Definition~3.16]{Fang2018} and~\cite[Definition~4.5]{DeRosa2020}) is the prototype, but it
is notoriously difficult to verify and depends sensitively on the class of competitors;
see~\cite[Theorem~5]{Burago2004} and~\cite[Theorem~1.2]{DeRosa2025}.

In the varifold formulation one considers, for an integrand~$F$ of class~$\cnt{1}$, the functional
\begin{displaymath}
    \Phi_F(V) = \textint{}{} F(T) \ud V(x,T) \quad \text{for $V \in \Var{\vdim}(U)$}
\end{displaymath}
defined on $\vdim$-dimensional varifolds in an open set $U \subseteq \R^{\adim}$, i.e., Radon
measures over $U \times \grass{\adim}{\vdim}$; see~\cite[3.1]{Allard1972}. Given a~smooth compactly
supported vectorfield $g : U \to \R^{\adim}$ with flow $h_t : U \to U$, the \emph{anisotropic first
  variation} of~$V$ is
\begin{equation}
    \label{eq:intro_fst_var} \delta_F V(g) = \left. \tfrac{\ud}{\ud t} \right|_{t=0} \Phi_F(h_{t*}V)
    = \textint{}{} B_F(T) \bullet \uD g(x) \ud V(x,T) \,,
\end{equation}
where $B_F(T) \in \End{\R^{\adim}}$ can be computed explicitly; identifying $\grass{\adim}{\vdim}$
with the submanifold $\pgrass{\adim}{\vdim} \subseteq \End{\R^{\adim}}$ of orthogonal projections of
rank~$\vdim$, one may extend~$F$ so that $B_F(T) = \grad F(T)$ for $T \in \pgrass{\adim}{\vdim}$;
see~\ref{rem:BF_gradient}.

De~Philippis, De~Rosa, and Ghiraldin~\cite[Definition~1.1 and Theorem~1.2]{DePhilippis2018}
identified a~condition on~$F$ -- the \emph{atomic condition}~\AC{} (see~\ref{def:AC}) -- which is
both necessary and sufficient for the validity of the anisotropic Allard rectifiability theorem:
every $V \in \Var{\vdim}(U)$ such that $\| \delta_F V \|$ is a~Radon measure and
$\density_{*}^{\vdim}(\|V\|,x) > 0$ for $\|V\|$~almost all~$x$ must be rectifiable. In particular,
$F$-stationary varifolds of \AC{} integrands with positive lower $\vdim$-density $\|V\|$-almost
everywhere are rectifiable. Subsequently, De~Rosa and the second
author~\cite[Theorem~8.8]{DeRosa2020} proved that \AC{} implies Almgren ellipticity. De~Rosa and
Tione~\cite[Definition~3.3]{DeRosa2022} introduced a~scalar variant \SAC{} together with its uniform
version \USAC{} (see~\ref{def:SAC} and~\ref{rem:USAC_in_QEC}). They also derived
$\cnt{1,\alpha}$-regularity of Lipschitz graphs with $p$-integrable anisotropic mean curvature, $p >
\vdim$, away from a~closed null set~\cite[Theorem~A]{DeRosa2022}. De~Rosa, Lei, and
Young~\cite[Theorem~1.16]{DeRosa2025} showed that every \AC{} integrand of class~$\cnt{1}$ is
\emph{strictly polyconvex}, i.e., arises from a~strictly convex norm~$\varphi$ on $\extpower{\vdim}
\R^{\adim}$ restricted to the cone of simple $\vdim$-vectors (see~\ref{def:polyconvex}).
In~codimension one, \cite[Theorem~1.3]{DePhilippis2018} characterises \AC{} completely: it is
equivalent to strict convexity of the one-homogeneous extension of~$F$.

Very recently, Gennaioli and Rindler~\cite[Theorems~1.6 and~1.9, Lemma~3.2]{Gennaioli2026} resolved
Bouchitt\'e's Vanishing Mass Conjecture and developed quantitative concentration and heat-flow
estimates for measures satisfying first-order PDE constraints. Firester and Tsiamis extended the
perturbative surface result of De~Philippis and Pigati~\cite[Theorem~1.3]{DePhilippis2026v2} to
arbitrary codimension in~\cite[Theorem~1.1]{FiresterTsiamis2026surfaces}. Combining their projection
method with the Gennaioli--Rindler estimates, they also obtained the anisotropic Michael--Simon
inequality in every dimension and codimension under a~spectral condition, covering integrands close
to area and all convex even hypersurface anisotropies; see~\cite[Theorems~1.1, 1.2,
and~4.1]{Firester2026} and~\ref{rem:concave_barriers}.  De~Philippis, Gennaioli, Pigati, and
Rindler~\cite[Theorems~1.1 and~1.6]{DePhilippisGennaioliPigatiRindler2026} now prove this inequality
in every dimension and codimension under the rank condition \ref{def:AC}\ref{i:AC:dim}, called
(AC1), for rectifiable varifolds with density at least~$1$, finite mass, and finite anisotropic
first variation. Their proof uses the heat-flow estimates of~\cite[Lemma~3.2]{Gennaioli2026} to
obtain scale-invariant mass bounds on sets of finite Hausdorff measure, without assuming closeness
to the area integrand.

These results make~\AC{} and its refinements the natural hypotheses for an~anisotropic counterpart
of Allard's theory. The condition~\AC{} is equivalent to the rectifiability
theorem~\cite[Theorem~1.2]{DePhilippis2018}. Its rank part (AC1) yields the Michael--Simon
inequality~\cite[Theorem~1.6]{DePhilippisGennaioliPigatiRindler2026}, which takes the place of the
monotonicity formula in the proof of density estimates~\cite[p.~2]{DeRosa2026v1}. Under~\AC{} and
for $\vdim \ge 2$, this inequality is equivalent to the compactness of rectifiable varifolds with
bounded anisotropic first variation~\cite[Proposition~1.9]{DePhilippis2026v2}. For integrands of
class~$\cnt{2}$, the quadratic exposed condition~\QEC{} of~\ref{def:QEC} together with the
Michael--Simon inequality yields an~anisotropic Allard regularity theorem in every dimension and
codimension~\cite[Theorem~1.1]{DeRosa2026v1}. The scope of this theory is therefore the class~\AC{}
and its subclasses. Yet \AC{} ``looks difficult to check on specific instances of anisotropic
integrands, even numerically''~\cite[section~2.2, p.~6]{Rosa2026}.

The area integrand~\cite[p.~1125]{DePhilippis2018}, its small class~$2$
neighbourhood~\cite[Theorem~B]{DeRosa2022}, and the area integrands of other scalar products
on~$\R^{\adim}$ (called ``ellipsoidal'') are examples of~\AC{} integrands in intermediate
dimensions; see~\ref{ex:J_fields}\ref{i:Jf:pairing}\ref{i:Jf:formula}, \ref{prop:SAC_in_EC},
and~\ref{prop:EC_in_AC}. De~Rosa, Firester, and Tsiamis construct explicit examples which are not
area integrands of inner products on~$\R^{\adim}$. They prove that the integrands associated to the
$\ell^q$~norms in Pl{\"u}cker coordinates satisfy \USAC{}, hence~\AC{}, on every
$\pgrass{\adim}{\vdim}$ whenever $2 - \tfrac{7}{10 \min \{ \vdim, \adim - \vdim \}} \le q \le 2$,
and construct axisymmetric \USAC{} integrands at a~definite $\cnt{0}$ distance from every
ellipsoidal one; see~\cite[Theorem~1.3 and Corollary~4.7]{DeRosa2026v1}.

In this paper we explain what~\AC{} says about an~integrand. The definition
in~\cite[Definition~1.1]{DePhilippis2018} quantifies over all probability measures on the
Grassmannian. We~replace it by a~statement about the convex position of a~single compact subset
of~$\End{\R^{\adim}}$, the \emph{Gauss image~$\GF{F}$} of~$F$:
\begin{displaymath}
        \conv \GF{F} \cap \bigl\{ A : \dim \im A \le \vdim \bigr\}
        = \GF{F} = \extreme{(\conv \GF{F})} \,,
\end{displaymath}
i.e., the convex hull of~$\GF{F}$ meets the variety of endomorphisms of rank at most~$\vdim$ exactly
in~$\GF{F}$, and every point of~$\GF{F}$ is an~extreme point of this convex hull;
see~\ref{thm:AC_extreme}. The remaining results of the paper build on this reformulation. The
exposed conditions \EC{}, \SEC{}, and~\QEC{} of~\ref{def:EC}, \ref{def:SEC}, and~\ref{def:QEC}
strengthen extremality to exposedness, to exposedness by hyperplanes containing every endomorphism
whose image lies in the corresponding plane, and to a~quantitative form of the latter. Strict
polyconvexity is necessary for~\AC{} by~\cite[Theorem~1.16]{DeRosa2025}, but it does not control the
position of this convex hull relative to the variety. The integrand of~\ref{thm:AC1_fails} is
strictly polyconvex, yet the convex hull of its Gauss image contains an~endomorphism of rank one;
by~\ref{thm:AC1_fails_all_codim}, strictly polyconvex integrands violating (AC1) exist whenever $2
\le \vdim \le \adim - 3$.

Extend~$F$ to a~continuous function on~$\End{\R^{\adim}}$ of class~$\cnt{1}$ away from the origin
such that $F(\lambda A) = |\lambda|^{\vdim} F(A)$ for $\lambda \in \R$ and $A \in \End{\R^{\adim}}$.
In the sequel we call such a~function a~\emph{$\vdim$-integrand}; see~\ref{def:integrand}
and~\ref{rem:integrand_defined_by_values_on_Gnk}. Regularity of integrands and of norms is always
understood away from the origin. The map~$B_F$ of~\eqref{eq:intro_fst_var} depends only
on~$F|\pgrass{\adim}{\vdim}$; see~\ref{rem:BF_gradient}. Following~\ref{def:aux-objects}, we set
\begin{gather*}
    P_F(T) = \transpose{B_F(T)} / F(T)
    \quad \text{and} \quad
    Q_F(T) = \id{\R^{\adim}} - P_F(T)
    \quad \text{for $T \in \pgrass{\adim}{\vdim}$} \,,
    \\
    \GF{F} = \bigl\{ \transpose{P}_F(T) : T \in \pgrass{\adim}{\vdim} \bigr\} \,,
    \quad \text{and} \quad
    \transpose{\GF{F}} = \bigl\{ \transpose{P} : P \in \GF{F} \bigr\} \,.
\end{gather*}
By~\ref{rem:se-cse}, the endomorphism~$P_F(T)$ is a~projection onto~$\im T$. It~is orthogonal if
and only if $\uD F(T)$ vanishes on the tangent space of~$\pgrass{\adim}{\vdim}$ at~$T$. Hence
$\transpose{P}_F(T)$ is a~projection of rank~$\vdim$ with kernel~$\ker T$, and $Q_F(T)$ is
a~projection with kernel~$\im T$.

In this notation the atomic condition reads as follows; see~\ref{def:AC} and~\ref{rem:AC_orig}.
For a~probability measure~$\mu$ over~$\pgrass{\adim}{\vdim}$ set
\begin{displaymath}
    A(\mu) = \textint{}{} \transpose{P}_F \ud \mu \in \End{\R^{\adim}} \,.
\end{displaymath}
The integrand~$F$ satisfies~\AC{} if and only if every such~$\mu$ satisfies two conditions. The
first is the rank condition $\dim \ker A(\mu) \le \adim - \vdim$, called (AC1)
in~\cite[Definition~3.1]{DeRosa2022}. The second requires $\mu = \Dirac{T}$ for some $T \in
\pgrass{\adim}{\vdim}$ whenever $\dim \ker A(\mu) = \adim - \vdim$. The endomorphisms~$A(\mu)$ fill
the convex hull~$\conv \GF{F}$ by the proof of~\ref{thm:AC_extreme}, and $\dim \ker A(\mu) \ge \adim
- \vdim$ if and only if $\dim \im A(\mu) \le \vdim$. Hence \AC{} is a~statement about the position
of~$\conv \GF{F}$ relative to the determinantal variety $\End{\R^{\adim}} \cap \{ A : \dim \im A \le
\vdim \}$.

\subsection*{Main results}
\addcontentsline{toc}{subsection}{Main results}

\begin{theorem}[\protect{see~\ref{thm:AC_extreme}}]
    \label{introthm:extreme}
    For every $\vdim$-integrand~$F$,
    \begin{displaymath}
        F \in \AC{}
        \quad \iff \quad
        \conv \GF{F} \cap \bigl\{ A : \dim \im A \le \vdim \bigr\}
        = \GF{F} = \extreme{(\conv \GF{F})} \,.
    \end{displaymath}
\end{theorem}
\noindent{}Thus $F \in \AC{}$ if and only if two conditions hold. First, the endomorphisms of rank
at most~$\vdim$ lying in~$\conv \GF{F}$ are exactly the points of~$\GF{F}$. Second, every point
of~$\GF{F}$ is an~extreme point of~$\conv \GF{F}$.

Replacing extreme points by exposed points yields the \emph{exposed condition}. We~write
$F \in \EC{}$ if
\begin{displaymath}
    \conv \GF{F} \cap \bigl\{ A : \dim \im A \le \vdim \bigr\}
    = \GF{F} = \exposed{(\conv \GF{F})} \,;
\end{displaymath}
see~\ref{def:EC}. Every exposed point is extreme; hence $\EC{} \subseteq \AC{}$
by~\ref{prop:EC_in_AC}.

De~Rosa and Tione~\cite[Definition~3.3]{DeRosa2022} introduced the \emph{scalar atomic
  condition}: write $F \in \SAC{}$ if
\begin{displaymath}
    \transpose{P}_F(T) \bullet Q_F(S) > 0
    \quad \text{for $S,T \in \pgrass{\adim}{\vdim}$ with $S \ne T$} \,;
\end{displaymath}
see~\ref{def:SAC}. The condition \SAC{} means that, for each $T \in \pgrass{\adim}{\vdim}$, the
hyperplane $\{ A : A \bullet \transpose{Q}_F(T) = 0 \}$ meets $\conv \transpose{\GF{F}}$ only
at~$P_F(T)$; see~\ref{mr:SAC_not_natural}. This hyperplane contains every endomorphism with image
in~$\im T$, because $Q_F(T) \circ T = 0$. The \emph{strong exposed condition} keeps this property of
the exposing hyperplane but leaves its normal free. We~write $F \in \SEC{}$ if there exists $N_F :
\pgrass{\adim}{\vdim} \to \End{\R^{\adim}}$ such that
\begin{displaymath}
    T \circ N_F(T) = 0
    \quad \text{and} \quad
    N_F(T) \bullet P_F(S) > 0
    \quad \text{for $S,T \in \pgrass{\adim}{\vdim}$ with $S \ne T$} \,;
\end{displaymath}
see~\ref{def:SEC}. The identity $T \circ N_F(T) = 0$ means $\im N_F(T) \subseteq \ker T$. It
ensures that the hyperplane $\{ A : A \bullet N_F(T) = 0 \}$ contains every endomorphism with image
in~$\im T$. The choice $N_F(T) = \transpose{Q}_F(T)$ gives $\SAC{} \subseteq \SEC{}$. The
condition~\EC{} asks only for some exposing hyperplane at each point of~$\GF{F}$, and
$\SEC{} \subseteq \EC{}$ by~\ref{lem:EC_normals}. We~regard the freedom in the choice of~$N_F$ as
more natural than the prescribed normal~$\transpose{Q}_F(T)$;
cf.~\ref{rem:specific_hyperplane_for_SAC}. Altogether,
\begin{displaymath}
    \SAC{} \subseteq \SEC{} \subseteq \EC{} \subseteq \AC{}
\end{displaymath}
by~\ref{prop:SAC_in_EC} and~\ref{prop:EC_in_AC}. In codimension one, $\SEC{} = \EC{} = \AC{}$
by~\ref{thm:AC=EC-codim1}. Whether \SEC{} or~\EC{} differs from~\AC{} in general is open;
see~\ref{q:EC=AC}.

The condition~\AC{} itself yields normals satisfying the weak inequalities. If $F \in \AC{}$, then
for each $T \in \pgrass{\adim}{\vdim}$ there is a~unit vector $N_F(T) \in \End{\R^{\adim}}$ with
$T \circ N_F(T) = 0$ such that
\begin{displaymath}
    P_F(S) \bullet N_F(T) \ge 0
    \quad \text{for $S \in \pgrass{\adim}{\vdim}$}
\end{displaymath}
and the hyperplane $\{ A : A \bullet N_F(T) = 0 \}$ misses $\rInt \cone \transpose{\GF{F}}$;
see~\ref{prop:existence_of_normals}. Such normals do not characterise the fields of projections
arising from integrands. Indeed, \ref{ex:J_fields} attaches to every $J \in \End{\R^{\adim}}$ with
$Jx \bullet x > 0$ for $x \in \R^{\adim} \without \{0\}$ a~field of projections~$P^J$ with $\im
P^J(T) = \im T$ and $\ker P^J(T) = \transpose{J} \lIm \ker T \rIm$ for $T \in
\pgrass{\adim}{\vdim}$. Set $N(T) = T^{\perp} \circ J$, where $T^{\perp} = \id{\R^{\adim}} - T$;
see~\ref{def:perp}. Then $T \circ N(T) = 0$, and $P^J(S) \bullet N(T) \ge 0$ with equality if and
only if $S = T$; see~\ref{ex:J_fields}\ref{i:Jf:pairing}. However, the field~$P^J$ equals~$P_F$ for
some integrand~$F$ \emph{if and only if} $J$ is symmetric. In~that case $F$ is a~positive multiple
of the area integrand of the scalar product $[(x,y) \mapsto J^{-1}x \bullet y]$;
see~\ref{ex:J_fields}\ref{i:Jf:dichotomy}\ref{i:Jf:formula}.

For polyconvex integrands, \SEC{} becomes a~condition of accretivity. An~integrand~$F$ is
\emph{polyconvex} if there exists a~norm $\varphi : \extpower{\vdim} \R^{\adim} \to \R$ such that
\begin{displaymath}
    F(T) = \varphi(\xi)
    \quad \text{whenever $T \in \pgrass{\adim}{\vdim}$, $\xi \in \ograss{\adim}{\vdim}$, and
      $\Sp \xi = \im T$} \,;
\end{displaymath}
here $\ograss{\adim}{\vdim}$ is the set of simple $\vdim$-vectors of length one and $\Sp \xi$ is
the subspace of~$\R^{\adim}$ associated to~$\xi$; see~\ref{def:oriented_grassmannian},
\ref{def:space}, and~\ref{def:polyconvex}. If $\varphi$ is strictly convex, $F$ is called
\emph{strictly polyconvex}. Let
$\varkappa : \End{\R^{\adim}} \to \End{\extpower{\vdim} \R^{\adim}}$ be the derivative at the
identity of the map $[A \mapsto \extpower{\vdim} A]$; thus
\begin{displaymath}
    \varkappa(L)(u_1 \wedge \cdots \wedge u_{\vdim})
    = \textsum{i=1}{\vdim} u_1 \wedge \cdots \wedge L u_i \wedge \cdots \wedge u_{\vdim}
    \quad \text{for $L \in \End{\R^{\adim}}$ and $u_1, \ldots, u_{\vdim} \in \R^{\adim}$} \,;
\end{displaymath}
see~\ref{def:varkappa}. If $\varphi$ is of class~$\cnt{1}$, then~\ref{lem:BF_via_varkappa} gives
\begin{gather*}
    B_F(T) \bullet L = \grad \varphi(\xi) \bullet \varkappa(L)\xi
    \\
    \text{whenever $L \in \End{\R^{\adim}}$, $T \in \pgrass{\adim}{\vdim}$,
      $\xi \in \ograss{\adim}{\vdim}$, and $\Sp \xi = \im T$} \,.
\end{gather*}
This formula is the basic computational tool of section~\ref{sec:AC1_failure}. Since
$B_F = F \, \transpose{P}_F$, it~also shows that $F \in \SEC{}$ if and only if there exists
$N : \pgrass{\adim}{\vdim} \to \End{\R^{\adim}}$ such that $T \circ N(T) = 0$ and
\begin{displaymath}
    \grad \varphi(\xi) \bullet \varkappa(\transpose{N(T)}) \xi > 0
    \quad \text{whenever $T \in \pgrass{\adim}{\vdim}$, $\xi \in \ograss{\adim}{\vdim}$, and
      $\Sp \xi \ne \im T$} \,;
\end{displaymath}
see~\ref{lem:SEC_accretive}. In the terminology of~\ref{def:accretive}, the operator
$\varkappa(\transpose{N(T)})$ is strictly \emph{$\varphi$-accretive} on the cone of simple
$\vdim$-vectors off~$\extpower{\vdim} \im T$. Accretivity belongs to the theory of semigroups of
operators; see~\cite[Definition~1.4.1]{Pazy1983}, \cite[Chap.~VI, 1.1]{Cioranescu1990},
and~\ref{rem:accretivity_origin}. The concave barriers with which Firester and Tsiamis prove the
anisotropic Michael--Simon inequality in~\cite[section~2.1]{Firester2026} are accretivity
functionals in exactly this sense; see~\ref{rem:concave_barriers}.

The \emph{quadratic exposed condition} is a~quantitative form of~\SEC{}. Given
$0 < \Gamma < \infty$ and $N : \pgrass{\adim}{\vdim} \to \End{\R^{\adim}} \cap \{ A : \|A\| = 1 \}$,
we write $F \in \QEC(\Gamma,N)$ if
\begin{displaymath}
    T \circ N(T) = 0
    \quad \text{and} \quad
    P_F(S) \bullet N(T) \ge \Gamma \|S - T\|^2
    \quad \text{for $S,T \in \pgrass{\adim}{\vdim}$} \,,
\end{displaymath}
and $F \in \QEC{}$ if such $\Gamma$ and~$N$ exist; see~\ref{def:QEC}. Then
$\QEC{} \subseteq \SEC{}$ by~\ref{lem:QEC_implies_EC}. The \emph{uniform scalar atomic
  condition}~\USAC{} of~\cite[Definition~3.3]{DeRosa2022} requires
$\transpose{P}_F(T) \bullet Q_F(S) \ge \Gamma \|S - T\|^2$ for $S,T \in \pgrass{\adim}{\vdim}$.
Normalising $\transpose{Q}_F(T)$ shows $\USAC{} \subseteq \QEC{}$; see~\ref{rem:USAC_in_QEC}.

The condition~\QEC{} is \emph{equivalent} to a~Caccioppoli-type inequality. For an~open set $U
\subseteq \R^{\adim}$ and $V \in \Var{\vdim}(U)$ let $V_F \in \Var{\vdim}(U)$ be given by
\begin{displaymath}
    V_F(\alpha) = \textint{}{} \alpha(x,T) F(T) \ud V(x,T)
    \quad \text{for $\alpha \in \ccspace{U \times \pgrass{\adim}{\vdim}}$} \,;
\end{displaymath}
see~\ref{def:VF}. We~write $F \in \CI{}$ if there exist an~open set $U \subseteq \R^{\adim}$,
a~smooth function $\zeta : U \to \R \cap \{ t : 0 \le t \le 1 \}$ with compact support and
$\zeta \ne 0$, and $0 < \Gamma < \infty$ such that
\begin{multline*}
    \textint{}{} \zeta(x)^2 \|S-T\|^2 \ud V_F(x,S)
    \le \Gamma \max \Bigl\{
    \textint{}{} \zeta(x)^2 |T^{\perp}(x-a)| \ud \|\delta_FV\|(x) \,,
    \\
    \textint{}{} |\grad \zeta(x)|^2 |T^{\perp}(x-a)|^2 \ud \|V_F\|(x)
    \Bigr\}
\end{multline*}
whenever $V \in \Var{\vdim}(U)$, $\|\delta_FV\|$ is a~Radon measure, $T \in \pgrass{\adim}{\vdim}$,
and $a \in \R^{\adim}$; see~\ref{def:CI}. The~inequality bounds the tilt-excess by the
height-excess and the anisotropic first variation. It~is modelled on~\cite[8.13]{Allard1972}
and~\cite[3.3]{Allard1986}.
\begin{theorem}[\protect{see~\ref{cor:CI_iff_QEC}}]
    \label{introthm:CI_QEC}
    Assume $F$ is a~$\vdim$-integrand of class~$\cnt{2}$. Then
    \begin{displaymath}
        F \in \CI{} \quad \iff \quad F \in \QEC{} \,.
    \end{displaymath}
\end{theorem}
\noindent{}The inclusion $\QEC{} \subseteq \CI{}$ is~\ref{prop:tilt-height}. Its proof
combines~\ref{lem:QEC_compatibility} with the argument of~\cite[8.13]{Allard1972}. The converse
inclusion is obtained by testing~\CI{} on diffuse varifolds and by a~minimax argument;
see~\ref{thm:CI_implies_LQ} and~\ref{cor:LQ_iff_QEC}.

The definition of~\QEC{} imposes no regularity on~$N$. However, if $F \in \QEC(\Gamma,N)$ is of
class~$\cnt{2}$, then $F \in \QEC(\Gamma/4,\tilde{N})$ for some Lipschitzian field~$\tilde{N}$
by~\ref{thm:continuous_normals}. Lipschitzian normals make~\QEC{} stable under small perturbations
of class~$\cnt{1,1}$. We~measure such perturbations as follows; see~\ref{def:C11_norm}. Let
$\End{\R^{\adim}}$ carry the operator norm. For $h : \pgrass{\adim}{\vdim} \to \R$ of
class~$\cnt{1,1}$ set
\begin{gather*}
    d_h(T) = \uD h(T) \circ \project{\Tan(\pgrass{\adim}{\vdim},T)}
    \in \Hom(\End{\R^{\adim}},\R)
    \quad \text{for $T \in \pgrass{\adim}{\vdim}$} \,,
    \\
    \boldsymbol\nu_2(h)
    = \sup \bigl\{ |h(T)| : T \in \pgrass{\adim}{\vdim} \bigr\}
    + \sup \bigl\{ \|d_h(T)\| : T \in \pgrass{\adim}{\vdim} \bigr\} + \Lip d_h \,,
\end{gather*}
where $\uD h(T)$ is the differential of~$h$ at~$T$, defined on the tangent space
$\Tan(\pgrass{\adim}{\vdim},T)$.
\begin{theorem}[\protect{see~\ref{thm:QEC_stability}}]
    \label{introthm:stability}
    Assume $F \in \QEC(\Gamma,N_F)$ is a~$\vdim$-integrand of class~$\cnt{1,1}$ and $N_F$ is
    Lipschitzian. There exists $\varepsilon > 0$ such that every $\vdim$-integrand~$G$ of
    class~$\cnt{1,1}$  with
    \begin{displaymath}
        \boldsymbol\nu_2 \bigl( (\log \circ F - \log \circ G) | \pgrass{\adim}{\vdim} \bigr)
        \le \varepsilon
    \end{displaymath}
    satisfies $G \in \QEC(\Gamma/2,\tilde{N}_G)$ for some Lipschitzian
    field~$\tilde{N}_G$.
\end{theorem}
\noindent{}Stability and the Caccioppoli-type inequality enter the $\varepsilon$-regularity theorem
of De~Rosa, Firester, and Tsiamis~\cite[Theorem~1.1]{DeRosa2026v1}.
By~\cite[Proposition~2.1]{DeRosa2026v1}, \QEC{} with a~Lipschitzian field of normals implies the
quantitative quasiconvexity of the nonparametric integrand, the analogue
of~\cite[Proposition~3.13]{DeRosa2022}, and hence a~local uniform Legendre--Hadamard condition.
The proof uses~\ref{thm:QEC_stability} and the inclusion $\QEC{} \subseteq \AC{}$. Their
Caccioppoli inequality~\cite[Proposition~2.2]{DeRosa2026v1} is deduced from~\ref{prop:tilt-height};
see~\ref{rem:regularity_missing_step}.

\begin{theorem}[\protect{see~\ref{thm:AC1_fails} and~\ref{thm:AC1_fails_all_codim}}]
    \label{introthm:example}
    Suppose $\vdim \ge 2$ and $\adim - \vdim \ge 3$. There exists a~strictly polyconvex
    $\vdim$-integrand~$F$ associated to an~inner-product norm on~$\extpower{\vdim} \R^{\adim}$ such
    that $F$ violates~the condition (AC1) of~\cite[Definition~3.1]{DeRosa2022}. In~particular,
    $F \notin \AC$.
\end{theorem}
\noindent{}For $\adim = 5$ and $\vdim = 2$, theorem \ref{thm:AC1_fails} exhibits a~norm~$\varphi$
on~$\extpower{2} \R^{5}$ induced by a~scalar product. The norm~$\varphi$ is strictly convex and real
analytic. Let $e_1, \ldots, e_5$ be the standard basis of~$\R^5$ and $e_6 = e_1$. Let $\mu$ be the
uniform probability measure on the five orthogonal projections onto $\lin \{ e_i, e_{i+1} \}$ for $i
\in \{1,\ldots,5\}$; see~\ref{mr:cyclic_data}. For every $2$-integrand~$F$ associated to~$\varphi$
we have $\dim \ker A(\mu) = 4$, whereas (AC1) demands $\dim \ker A(\mu) \le \adim - \vdim = 3$. The
example~is not isolated.  By~\ref{cor:AC1_family} the same failure occurs for every norm~$\psi$
on~$\extpower{2} \R^{5}$ of class~$\cnt{1}$ such that $\grad \psi(\xi)$ is a~positive multiple
of~$\grad \varphi(\xi)$ whenever $\xi \in \ograss{5}{2}$ and $\project{(\Sp \xi)} \in \spt \mu$.
These norms include the inner-product norms $[x \mapsto (x \bullet Mx)^{1/2}]$ with $M$ in
a~non-empty relatively open subset of a~$15$-dimensional affine space of self-adjoint endomorphisms
of~$\extpower{2} \R^{5}$. By~\ref{thm:AC1_fails_all_codim} a~suspension carries the example to every
$\vdim \ge 2$ and $\adim - \vdim \ge 3$. The case $\adim - \vdim = 2$ remains open;
see~\ref{rem:AC1_fails_discussion}.

\subsection*{Organisation of the paper}
\addcontentsline{toc}{subsection}{Organisation of the paper}

Sections~\ref{sec:basic_notation} and~\ref{sec:integrands} fix notation. We~follow
Federer~\cite[pp.~669--671]{Federer1969} and Rockafellar~\cite{Rockafellar1970}. These sections
also collect the basic properties of integrands and of the objects $B_F$, $P_F$, $Q_F$,
and~$\GF{F}$ attached to them.

Section~\ref{sec:ac_by_convex_geometry} proves~\ref{introthm:extreme} (i.e.~\ref{thm:AC_extreme}),
introduces \EC{} and~\SEC{}, and proves the inclusions between \SAC{}, \SEC{}, \EC{},
and~\AC{}. It~also contains a~rigidity property of \AC{} integrands. If $F \in \AC{}$, then the maps
$[S \mapsto \project{( \im \transpose{P}_F(S) )}]$ and $[S \mapsto \ker P_F(S)]$ are homeomorphisms
of~$\pgrass{\adim}{\vdim}$ onto $\pgrass{\adim}{\vdim}$ and~$\grass{\adim}{\codim}$ respectively;
see~\ref{prop:AC_imPF_homeo} and~\ref{cor:kerPF_homeo}. For integrands of class~$\cnt{2}$ the
normals of~\ref{prop:existence_of_normals} satisfy $N_F(T) \circ \transpose{P}_F(T) = 0$;
see~\ref{rem:normals_compatible}. The section closes with an example concerning fields of
projections on the Grassmannian which satisfy an AC-type condition but do not come from an
integrand; see~\ref{ex:J_fields}.

Section~\ref{sec:polyconvex_integrands} treats polyconvex integrands. It~contains the formula
of~\ref{lem:BF_via_varkappa} and the characterisation of~\SEC{} by accretivity
in~\ref{lem:SEC_accretive}. It~also discusses the \emph{total convexity} of Carath{\'e}odory and
Busemann~\cite[p.~36]{Busemann1960} and~\cite[p.~141]{Busemann1961}
(see~\ref{def:totally_convex_norm}),
and its connection with~\SAC{} in~\ref{rem:totally_convex_norms_and_AC}.

Section~\ref{sec:uniform_AC} introduces~\QEC{} and~\CI{}. It~proves $\USAC{} \subseteq \QEC{}
\subseteq \SEC{}$, the existence of Lipschitzian normals in~\ref{thm:continuous_normals},
and~\ref{introthm:CI_QEC}: the equivalence of~\CI{} and~\QEC{}; cf.~\ref{cor:CI_iff_QEC}.

Section~\ref{sec:QEC_stability} proves stability~\ref{introthm:stability}
(i.e.~\ref{thm:QEC_stability}) of~\QEC{}. It~records in~\ref{rem:regularity_missing_step} how the
results of the present paper enter~\cite[Propositions~2.1 and~2.2]{DeRosa2026v1}.

Section~\ref{sec:AC1_failure} contains the counterexamples to the sufficiency of strict
polyconvexity and, in~particular, the proof of~\ref{introthm:example}. They have a~consequence for
the known positive results. The validity of~\ref{def:AC}\ref{i:AC:dim} for the $\ell^p$~norms
on~$\pgrass{4}{2}$ established in~\cite[Theorem~6.1]{DeRosa2022}, and of~\USAC{} for the
$\ell^q$~norms with $q$ close to~$2$ on every $\pgrass{\adim}{\vdim}$ established
in~\cite[Theorem~1.3]{DeRosa2026v1}, cannot be attributed to strict convexity of the norm alone;
see~\ref{rem:AC1_fails_discussion}.

The questions left open are~\ref{q:EC=AC}, \ref{q:total_convexity_SAC}, and the case
$\adim - \vdim = 2$ of~\ref{thm:AC1_fails_all_codim}; see~\ref{rem:AC1_fails_discussion}.

\section{Basic notions and notation}
\label{sec:basic_notation}

\begin{miniremark}
    In this section we fix the notation and recall the background material used throughout the
    paper. We~set up the model of the Grassmannian by orthogonal projections
    (see~\ref{def:projections} and~\ref{def:space}), recall the notions of convex geometry needed to
    speak about extreme and exposed points of a~convex set (see~\ref{def:exposed}
    and~\ref{thm:Straszewicz}), and relate these notions to the extreme and exposed rays of the
    cone generated by a~convex set (see~\ref{lem:base_of_cone}).
\end{miniremark}

\begin{miniremark}
    \label{mr:notation_citations}
    In~principle we shall follow the notation of Federer; see~\cite[pp.~669--671]{Federer1969}. For
    notions of convex analysis we adopt the notation
    of~Rockafellar~\cite{Rockafellar1970}. Regarding varifolds, our main reference is Allard's classical
    paper~\cite{Allard1972}.

    \begin{itemize}
        
    \item Given two real vectorspaces $X$ and $Y$ we write $\Hom(X,Y)$ for the space of linear maps
        of type $X \to Y$. The space $\Hom(X,\R)$ is denoted by $\adjoint{X}$ and
        $\End{X} = \Hom(X,X)$.

    \item If $A$ is a subset of a topological space, then $\Clos{A}$ is the closure
        of~$A$, $\Int{A}$ is the interior of~$A$, and $\Bdry{A}$ denotes the topological
        boundary of~$A$.

    \item By a~\emph{Euclidean space} we mean a finite-dimensional real vectorspace
        equipped with an inner product. Whenever $X$ is a~Euclidean space, the scalar
        product of $x,y \in X$ is denoted by $x \bullet y$. In~this section, unless stated
        otherwise, $X$ and $Y$ shall denote Euclidean spaces.
    
    \item If $A \subseteq X$ and $x \in X$, then $\Tan(A,x)$ and $\Nor(A,x)$ are the
        tangent and normal cones as defined in~\cite[3.1.21]{Federer1969}.

    \item We use three alternative notations for the evaluation of a~linear map~$L$ on
        a~vector~$u$, namely $\langle u ,\, L \rangle$, $L u$, or~$L(u)$.

    \item For the Dirac measure centred at $x$ we write $\Dirac{x}$; cf.~\cite[2.5.19]{Federer1969}.

    \end{itemize}
\end{miniremark}

\begin{definition}
    \label{def:scale_translate}
    Whenever $X$ is a vectorspace, $a \in X$, and $r \in \R$, the maps
    \begin{gather}
        \scale{r} : X \to X
        \quad \text{and} \quad
        \trans{a} : X \to X
        \\
        \text{are given by} \quad
        \scale{r}(x) = rx
        \quad \text{and} \quad \trans{a}(x) = a+x
        \quad \text{for $x \in X$} \,.
    \end{gather}
\end{definition}

\begin{miniremark}
    In case $X$ is Euclidean, $A \subseteq X$, and $f : A \to \R$ is differentiable at $x \in A$,
    we~distinguish between the \emph{gradient} $\grad f(x) \in X$ and the \emph{derivative}
    $\uD f(x) \in \adjoint{X}$ of~$f$ at~$x$. The two notions are related by the formula
    \begin{displaymath}
        \grad f(x) \bullet u = \langle u ,\, \uD f(x) \rangle
        \quad \text{for $u \in X$} \,.
    \end{displaymath}
    The set of differentiability points of~$f$ is the domain of~$\uD f$, i.e.,
    $\dmn \uD f$.
\end{miniremark}

\subsection*{Model of the Grassmannian}
\addcontentsline{toc}{subsection}{Model of the Grassmannian}

\begin{miniremark}
    In this subsection we assume $X$ is a Euclidean space.
\end{miniremark}

\begin{definition}
    Whenever $0 \le \vdim \le \adim = \dim X$, we define the \emph{projective Grassmannian} by
    \begin{gather}
        \pgrass{X}{\vdim} =  \End{X} \cap \bigl\{
        P : P \circ P = P ,\, \transpose{P} = P ,\, \trace P = \vdim
        \bigr\} 
        \\
        \text{and} \quad
        \pgrass{\adim}{\vdim} = \pgrass{\R^\adim}{\vdim}
        \quad \text{in case $X = \R^\adim$} \,.
    \end{gather}
\end{definition}

\begin{definition}
    \label{def:perp}
    If $T \in \End{X}$, we set
    \begin{displaymath}
        T^{\perp} = \id{X} - T \in \End{X}\,.
    \end{displaymath}
\end{definition}

\begin{remark}
    \label{rem:perp}
    In case $T \in \pgrass{X}{\vdim}$, we have $T^{\perp} \in \pgrass{X}{\dim X - \vdim}$.
    More generally, if $P \in \End{X}$ is a projection (meaning $P \circ P = P$), then
    $P^{\perp}$ is also a projection.
\end{remark}

\begin{definition}
    The Grassmannian of $\vdim$-dimensional subspaces of $X$ is denoted $\grass{X}{\vdim}$.
    Whenever $T \in \grass{X}{\vdim}$, i.e., $T$ is a linear subspace of~$X$, we define
    $\project{T} \in \pgrass{X}{\vdim}$ to be the orthogonal projection onto~$T$, i.e.,
    $\project{T} \in \End{X}$ is characterised by the conditions
    $\project{T} \circ \project{T} = \project{T}$, $\transpose{{\project{T}}} = \project{T}$, and
    $\im \project{T} = T$. The Grassmannian $\grass Xk$ is endowed with the metric inherited from
    $\bigl[ S \mapsto \project{S} \bigr] : \grass{X}{\vdim} \to \pgrass{X}{\vdim}$, i.e., the
    distance between $S,T \in \grass{X}{\vdim}$ equals $\| \project{S} - \project{T} \|$.
\end{definition}

\begin{remark}
    Note that $\pgrass{X}{\vdim}$ is a real analytic compact submanifold of~$\End{X}$ which is
    diffeomorphic to the abstract manifold $\grass{X}{\vdim}$ via the map
    $\bigl[ \grass{X}{\vdim} \ni T \mapsto \project{T} \bigr]$.
\end{remark}

\begin{remark}
    Note that if $T \in \pgrass{\adim}{\vdim}$, then
    \begin{displaymath}
        \project{(\im T)} = T 
        \quad \text{and} \quad
        \project{(\ker T)} = T^{\perp} \,.
    \end{displaymath}
\end{remark}

\begin{definition}
    \label{def:projections}
    For $j \in \{1,2,\ldots,\adim\} $ we define the \emph{space of projections of $\R^{\adim}$ onto
      $j$-di\-men\-sion\-al subspaces} as
    \begin{displaymath}
        \Proj{\adim}{j} = \End{\R^{\adim}} \cap \bigl\{ P : P \circ P = P ,\, \trace P = j \bigr\}
        \,.
    \end{displaymath}
\end{definition}

\begin{definition}[\protect{cf.~\cite[1.6.2 and 3.2.28(2)]{Federer1969}}]
    \label{def:oriented_grassmannian}
    The \emph{Grassmannian of oriented $\vdim$-planes} in~$\R^{\adim}$ is defined as
    \begin{displaymath}
        \ograss{\adim}{\vdim} = \extpower{\vdim} \R^{\adim} \cap \bigl\{ \xi : |\xi| = 1 \text{ and $\xi$ is simple} \bigr\} \,.
    \end{displaymath}
\end{definition}

\begin{definition}
    \label{def:space}
    Let $\xi \in \extpower{\vdim} \R^{\adim}$. The space associated to $\xi$ is denoted
    \begin{displaymath}
        \Sp \xi = \R^{\adim} \cap \bigl\{ v : v \wedge \xi = 0 \bigr\} \,.
    \end{displaymath}
\end{definition}

\begin{definition}
    \label{def:image_space}
    Whenever $P \in \End{\R^{\adim}}$ we define the linear spaces
    \begin{displaymath}
        \imsp(P) = \End{\R^{\adim}} \cap \bigl\{ A : \im A \subseteq \im P \bigr\}
        \quad \text{and} \quad
        \transpose{\imsp(P)} = \bigl\{ \transpose{A} : A \in \imsp(P) \ \bigr\} \,.
    \end{displaymath}
\end{definition}

\begin{remark}
    \label{rem:basic_props_of_IT}
    Observe that if $T \in \pgrass{\adim}{\vdim}$, then
    \begin{gather}
        \imsp(T) = \End{\R^{\adim}} \cap \bigl\{ A : A = T \circ A \bigr\} \,,
        \quad
        \imsp(T)^{\perp} = \imsp(T^{\perp}) \,,
        \quad
        \bigl( \transpose{\imsp(T)} \bigr)^{\perp} = \transpose{\imsp(T^{\perp})} \,.
    \end{gather}
\end{remark}

\subsection*{Convex geometry}
\addcontentsline{toc}{subsection}{Convex geometry}

\begin{definition}
    \label{def:cone}
    Let $A \subseteq X$. We say that $A$ is a \emph{cone} if for each $a \in A$
    and $0 < t < \infty$ we have $ta \in A$.
\end{definition}

\begin{definition}
    A~\emph{hyperplane} is any set of the form $X \cap \{ x : \alpha(x) = c \}$
    for some $\alpha \in \adjoint{X} \without \{0\}$ and $c \in \R$.
\end{definition}

\begin{definition}[\protect{\cite[p.~15]{Rockafellar1970}}]
    Let $A \subseteq X$. The \emph{cone generated by $A$} is the set
    \begin{displaymath}
        \ray A = \bigl\{ ta : a \in A ,\, 0 \le t < \infty \bigr\} \,.
    \end{displaymath}
\end{definition}

\begin{definition}[\protect{\cite[p.~12]{Rockafellar1970}}]
    Let $A \subseteq X$. We define the \emph{convex hull} of~$A$ to be the set
    \begin{displaymath}
        \conv A = {\textstyle \bigcap} \bigl\{
        C : A \subseteq C ,\, C \text{ is convex}
        \bigr\} \,.
    \end{displaymath}
\end{definition}

\begin{definition}[\protect{\cite[p.~14]{Rockafellar1970}}]
    Let $A \subseteq X$. The \emph{convex cone generated by $A$} is the set
    \begin{displaymath}
        \cone A = \conv ( \ray  A ) \,.
    \end{displaymath}
\end{definition}

\begin{definition}
    Let $A \subseteq X$. The \emph{affine space generated by $A$} is the
    smallest affine subspace of~$X$ containing $A$ and is denoted $\aff A$.
\end{definition}

\begin{definition}[\protect{\cite[p.~44]{Rockafellar1970}}]
    Let $A \subseteq X$. We define the \emph{relative interior of~$A$} to be the
    largest open subset of~$\aff A$ contained in~$A$ and we denote it by $\rInt A$.
\end{definition}

\begin{definition}[\protect{\cite[section~18]{Rockafellar1970}}]
    Let $C \subseteq X$ be a convex set and $F \subseteq C$.
    \begin{enumerate}
    \item We say that $F$ is a~\emph{face} of~$C$ if $F$ is convex and any closed line segment $J
        \subseteq C$ such that $F \cap \rInt J \ne \varnothing$ is contained in~$F$.

    \item In case $C$ is a cone, $x \in C \without \{0\}$, and $R = \ray \{ x \}$ is a face of $C$,
        we say that $R$ is an \emph{extreme ray} of $C$.

    \item A face $F$ of $C$ is called \emph{exposed} if $F = C$ or there exists
        a~hyperplane~$H \subseteq X$ such that $F = C \cap H$. Thus the whole set is allowed
        as an~improper exposed face.

    \item In case $C$ is a cone, $x \in C \without \{0\}$, and $R = \ray \{ x \}$ is an exposed face
        of $C$, we say that $R$ is an \emph{exposed ray} of $C$.

    \item If $F$ is a single point, then we call it an~\emph{extreme point} of~$C$, and if $F$ is
        additionally exposed, we call it an~\emph{exposed point} of~$C$. The sets of extreme and
        exposed points of $C$ are denoted
        \begin{displaymath}
            \extreme C
            \quad \text{and} \quad
            \exposed C
            \quad \text{respectively} \,.
        \end{displaymath}
    \end{enumerate}
\end{definition}

\begin{theorem}[\protect{The Straszewicz Theorem~\cite[Theorem~18.6]{Rockafellar1970}}]
    \label{thm:Straszewicz}
    Assume $C \subseteq X$ is a closed convex set. Then $\exposed C$ is dense in
    $\extreme C$.
\end{theorem}

\begin{definition}
    \label{def:exposed}
    Let $A \subseteq X$. We say that the set $A$ is \emph{extreme}
    [\emph{exposed}] if
    \begin{displaymath}
        A = \extreme \bigl( \conv A \bigr)
        \quad \bigl[\ 
        A = \exposed \bigl( \conv A \bigr)
        \ \bigr] \,.
    \end{displaymath}
\end{definition}

\begin{remark}
    Note that if $A$ is extreme and $B \subseteq A$, then $B$ is extreme as
    well. Similarly, if $A$ is exposed and $B \subseteq A$, then $B$ is also exposed.
\end{remark}

\begin{example}
    Let $S = \R^{\adim} \cap \{ x : |x| = 1 \}$ be the Euclidean sphere
    and~$E \subseteq S$ be arbitrary. For $x \in S$ the hyperplane
    $\{ y : x \bullet y = 1 \}$ meets the closed unit ball only at~$x$; hence $S$ and~$E$
    are exposed.
\end{example}

\begin{lemma}
    \label{lem:base_of_cone}
    Assume $\alpha \in \adjoint{X}$, $A \subseteq X \cap \{ x : \alpha(x) = 1 \}$, $K = \conv A$,
    and $x \in K$. Then
    \begin{enumerate}
    \item
        \label{i:boc:extreme}
        $x \in \extreme K$ if and only if $\ray \{ x \}$ is an extreme ray of $\cone A$;
    \item
        \label{i:boc:exposed}
        $x \in \exposed K$ if and only if $\ray \{ x \}$ is an exposed ray of $\cone A$;
    \item
        \label{i:boc:rint}
        $\rInt \cone A = \bigl\{ tz : 0 < t < \infty ,\, z \in \rInt K \bigr\}$; in~particular,
        $\rInt K \subseteq \rInt \cone A$.
    \end{enumerate}
\end{lemma}

\begin{proof}
    Since $\cone A = \{ tz : 0 \le t < \infty ,\, z \in K \}$ and $\alpha(tz) = t$, the~point $0$ is
    the only member of $\cone A$ annihilated by~$\alpha$, every other $y \in \cone A$ satisfies
    $y = \alpha(y) z$ with $z = y/\alpha(y) \in K$, and $\ray \{ x \} \cap K = \{ x \}$.

    Assume $x \in \extreme K$ and let $J = [a,b] \subseteq \cone A$ be such that
    $\ray \{ x \} \cap \rInt J \ne \varnothing$; say $y = (1-\lambda)a + \lambda b \in \ray \{x\}$
    with $0 < \lambda < 1$. Applying~$\alpha$ gives
    $\alpha(y) = (1-\lambda)\alpha(a) + \lambda\alpha(b)$. If $\alpha(y) = 0$, then
    $\alpha(a) = \alpha(b) = 0$ and $J = \{0\} \subseteq \ray \{x\}$. Otherwise division
    by~$\alpha(y)$ presents $x = y/\alpha(y)$ as a~convex combination of those of the points
    $a/\alpha(a)$, $b/\alpha(b) \in K$ which are defined; as $x \in \extreme K$ each of them
    equals~$x$, so $a,b \in \ray \{x\}$. Hence $\ray \{x\}$ is a~face of~$\cone A$. Conversely, if
    $\ray \{x\}$ is a~face and $x = (1-\lambda)u + \lambda v$ with $u,v \in K$ and
    $0 < \lambda < 1$, then $[u,v] \subseteq \ray \{x\}$, so $u,v \in \ray \{x\} \cap K = \{x\}$.
    This proves~\ref{i:boc:extreme}.

    If $\{x\} = K \cap H$ for a~hyperplane $H = \{ z : \omega(z) = c \}$ with $\omega \le c$ on~$K$,
    then $\omega - c\alpha$ vanishes on $\ray \{x\}$ and is negative on $\cone A \without \ray\{x\}$,
    so $\ray \{x\} = \cone A \cap \ker(\omega - c\alpha)$. If $\omega-c\alpha = 0$,
    then $\cone A = \ray\{x\}$, which is exposed by the convention on improper faces.
    Conversely, this improper case gives $K = \{x\}$. If $\ray\{x\} = \cone A \cap H$
    for a~hyperplane~$H$, then, since $K \subseteq \cone A$, we get
    $K \cap H = \ray \{x\} \cap K = \{x\}$. Together with~\ref{i:boc:extreme}
    this proves~\ref{i:boc:exposed}.

    Finally, $\aff \cone A$ is a~linear subspace containing~$K$; hence, it equals $V = \lin \aff K$.
    If $v \in V$ and $\alpha(v) = 1$, then $v$ is an~affine combination of points of $\aff K$, so
    $V \cap \{ z : \alpha(z) = 1 \} = \aff K$ and the map $[ (t,z) \mapsto tz ]$ is a~homeomorphism
    of $\{ t : 0 < t < \infty \} \times \aff K$ onto $V \cap \{ z : \alpha(z) > 0 \}$. It~carries
    the open set $\{ t : 0 < t < \infty \} \times \rInt K$ onto a~subset of $\cone A$ which is open
    in~$V$, whence ``$\supseteq$'' in~\ref{i:boc:rint}. Since $\alpha$ takes negative values
    on~$V$, no neighbourhood of~$0$ in~$V$ lies in $\cone A$; thus, $0 \notin \rInt \cone A$ and the
    preimage of $\rInt \cone A$ is an~open subset of $\{ t : 0 < t < \infty \} \times \aff K$
    contained in $\{ t : 0 < t < \infty \} \times K$, hence in
    $\{ t : 0 < t < \infty \} \times \rInt K$. This gives ``$\subseteq$''.
\end{proof}

\begin{definition}[\protect{\cite[section~23]{Rockafellar1970}}]
    \label{def:subdifferential}
    Let $X$ be a Euclidean space and $f : X \to \R$ be convex. The \emph{subdifferential} of~$f$ at
    a~point $x \in X$ is the set $\subgrad f(x)$ of all vectors $v \in X$ satisfying
    \begin{displaymath}
        f(z) \ge f(x) + v \bullet (z-x)
        \quad \text{for all $z \in X$} \,.
    \end{displaymath}
\end{definition}

\begin{definition}
    Let $f : X \to \R$ and $0 \le l < \infty$. When $l = 0$, the identities below
    are required only for $t \ne 0$. We say that $f$ is
    \emph{$l$-homogeneous} if
    \begin{displaymath}
        f(t x) = |t|^l f(x) \quad \text{for $t \in \R$ and $x \in X$} \,.
    \end{displaymath}
    We say that $f$ is \emph{positively $l$-homogeneous} if 
    \begin{displaymath}
        f(t x) = t^l f(x) \quad \text{for $0 \le t < \infty$ and $x \in X$} \,.
    \end{displaymath}
\end{definition}

\begin{definition}[\protect{\cite[section~14]{Rockafellar1970}}]
    \label{def:polar_function}
    Let $f : X \to \overline{\R}$ be a positively $1$-homogeneous function. We~define the
    \emph{polar} function $\polar{f} : X \to \overline{\R}$ by the formula
    \begin{displaymath}
        \polar{f}(\xi) = \sup \bigl\{ x \bullet \xi : x \in X ,\, f(x) \le 1 \bigr\}
        \quad \text{for $\xi \in X$} \,.
    \end{displaymath}
\end{definition}

\begin{definition}[\protect{\cite[section~14, p.~125]{Rockafellar1970}}]
    \label{def:polar_body}
    Let $C \subseteq X$ be a closed convex set with $0 \in C$. We define
    the~\emph{polar} body of~$C$ by
    \begin{displaymath}
        \polar{C} = X \cap \bigl\{ \xi : x \bullet \xi \le 1 \text{ for } x \in C \bigr\} \,.
    \end{displaymath}
\end{definition}

\section{Integrands and related objects}
\label{sec:integrands}

\begin{miniremark}
    Here we introduce the objects attached to an~integrand~$F$ which are used throughout the paper:
    the~map~$B_F$ governing the first variation of the functional~$\Phi_F$
    (cf.~\ref{rem:BF_variational}), the~associated maps $P_F$ and~$Q_F$ with values in projections,
    and the~compact set $\GF{F} = \transpose{P}_F \lIm \pgrass{\adim}{\vdim} \rIm$ whose convex
    geometry turns out to encode the atomic condition; see~\ref{def:aux-objects}. We~also
    observe that $\GF{F}$ lies on an~affine hyperplane missing the origin and deduce that the
    properties of being extreme or exposed are not affected by passing to adjoints or to~$B_F$;
    see~\ref{rem:GF_on_hyperplane} and~\ref{cor:extreme-equivalence}.
\end{miniremark}

\begin{miniremark}
    In the rest of this paper, unless stated otherwise, we shall assume $\vdim$ and $\adim$ are
    positive integers such that $1 \le \vdim \le \adim-1$.
\end{miniremark}

\begin{definition}
    \label{def:integrand}
    By a~\emph{$\vdim$-integrand} we mean a~continuous function $F : \End{\R^{\adim}} \to \R$ which
    is of class~$\cnt{1}$ on $\End{\R^{\adim}} \without \{0\}$ and satisfies
    \begin{gather}
        \inf F \lIm \pgrass{\adim}{\vdim} \rIm > 0 
        \quad \text{and} \quad
        F(\lambda A) = |\lambda|^\vdim F(A)
        \quad \text{for $\lambda \in \R$ and $A \in \End{\R^{\adim}}$} \,.
    \end{gather}
\end{definition}

\begin{remark}
    \label{rem:integrand_defined_by_values_on_Gnk}
    Only the values of~$F$ on~$\pgrass{\adim}{\vdim}$ enter the subsequent constructions.
    Let $F_0 : \pgrass{\adim}{\vdim} \to \R \cap \{t:t>0\}$ be of class~$\cnt{l}$,
    $l \in \natp$. A~tubular retraction and a~smooth cut-off extend
    $[T/\sqrt{\vdim} \mapsto \vdim^{-\vdim/2}F_0(T)]$ to a~function~$g$ of class~$\cnt{l}$
    on the Euclidean unit sphere in~$\End{\R^{\adim}}$. The two compact sets
    $\pgrass{\adim}{\vdim}/\sqrt{\vdim}$ and $-\pgrass{\adim}{\vdim}/\sqrt{\vdim}$ are disjoint,
    so the cut-offs may have disjoint supports and $g$ may be chosen even. Then
    \begin{displaymath}
        F(A) = |A|^{\vdim}g(A/|A|) \quad \text{for $A \ne 0$} \,,
        \qquad F(0) = 0
    \end{displaymath}
    defines a~$\vdim$-integrand extending~$F_0$, of class~$\cnt{l}$ away from the origin.
    If $\vdim > l$, its derivatives of orders at most~$l$ extend by zero at the origin.
\end{remark}

\begin{remark}
    \label{rem:integrand-k-homogeneous}
    Let $F$ be a $\vdim$-integrand and $T \in \pgrass{\adim}{\vdim}$. The~conditions imposed on~$F$
    in~\ref{def:integrand} imply that
    \begin{displaymath}
        \grad F(T) \bullet T = \vdim F(T) \,.
    \end{displaymath}
\end{remark}

\begin{definition}
    \label{def:aux-objects}
    Let $F$ be a~$\vdim$-integrand. Recalling~\ref{def:perp}, we define
    \begin{gather}
        \pi(T)L = T^{\perp} \circ L \circ T + \transpose{\bigl( T^{\perp} \circ L \circ T \bigr)}
        \quad \text{for $L,T \in \End{\R^{\adim}}$} \,,
        \\
        B_F(T) \bullet L = F(T) T \bullet L + \langle \pi(T)L ,\, \uD F(T) \rangle
        \quad \text{for $L \in \End{\R^{\adim}}$ and $T \in \dmn \uD F$} \,,
        \\
        P_F(T) = \transpose{B_F(T)} / F(T) 
        \quad \text{and} \quad
        Q_F(T) = P_F(T)^{\perp}
        \quad \text{for $T \in \dmn B_F \without F^{-1} \{0\}$} \,,
        \\
        \GF{F} = \transpose{P}_F \lIm \pgrass{\adim}{\vdim} \rIm \subseteq \End{\R^{\adim}} \,,
        \quad
        \transpose{\GF{F}} = \bigl\{ \transpose{P} : P \in \GF{F} \bigr\} \,,
        \\
        \sigma(T)L = \tfrac 1k (L \bullet T) T + \pi(T)L
        \quad \text{for $L,T \in \End{\R^{\adim}}$}
        \,.
    \end{gather}
\end{definition}

\begin{remark}
    Since $\uD F$ is continuous in a neighbourhood of $\pgrass{\adim}{\vdim}$, we see that $\GF{F}$
    is a compact set.
\end{remark}

\begin{remark}
    \label{rem:BF_variational}
    Referring to~\cite[Appendix~A]{DePhilippis2018}, we recall the following fact. Given
    $L : \R \to \End{\R^{\adim}}$ of class~$\cnt{1}$ such that $L(0) = \id{\R^{\adim}}$, we have
    \begin{displaymath}
        B_F(T) \bullet L'(0)
        = \left. \tfrac{\ud}{\ud t} \right|_{t=0}
        F( \project{ \im(L(t) \circ T)} )\, \| \extpower{\vdim}(L(t) \circ T) \| \,.
    \end{displaymath}
\end{remark}

\begin{remark}
    \label{rem:sigma_props}
    Let $T \in \pgrass{\adim}{\vdim}$ and $L \in \End{\R^{\adim}}$. There holds
    \begin{gather}
        \label{i:im_pi}
        \im \pi(T) = \Tan(\pgrass{\adim}{\vdim},T)
        \,,
        \\
        \label{i:im_sigma}
        \im \sigma(T) = \Tan(\ray \pgrass{\adim}{\vdim},T)
        \,,
        \\
        \label{i:pi_projection}
        \pi(T) \circ \pi(T) = \pi(T)
        \,,
        \\
        \label{i:sigma_projection}
        \sigma(T) \circ \sigma(T) = \sigma(T)
        \,,
        \\
        \label{i:pi_transpose}
        \transpose{\pi(T)}L = T^{\perp} \circ (L + \transpose{L}) \circ T
        \,,
        \\
        \label{i:sigma_transpose}
        \transpose{(\transpose{\sigma(T)}L)}
        = \tfrac 1k (L \bullet T) T + T \circ (L + \transpose{L}) \circ T^{\perp}
        \,,
        \\
        \label{i:projections}
        \text{if} \quad
        P = \transpose{(\transpose{\sigma(T)}L)}
        \quad \text{and} \quad
        L \bullet T = \vdim \,,
        \quad \text{then} \quad
        P \circ P = P
        \quad \text{and} \quad
        \im P = \im T
        \,,
        \\
        \label{i:pi_self_adjoint}
        \pi(T) X \bullet Y = X \bullet \pi(T) Y
        \quad \text{for $X,Y \in \End{\R^{\adim}} \cap \{ A : A = \transpose{A} \}$}
        \,.
    \end{gather}

    Indeed, referring, e.g., to~\cite[Appendix~A]{DePhilippis2018}, we see that
    \begin{equation}
        \label{eq:tan_grass}
        \Tan(\pgrass{\adim}{\vdim},T)
        = \End{\R^{\adim}} \cap \bigl\{
        X :
        \transpose{X} = X ,\,
        T \circ X \circ T = 0 ,\,
        T^{\perp} \circ X \circ T^{\perp} = 0
        \bigr\} \,.
    \end{equation}
    To prove~\eqref{i:im_pi}, first let $P \in \im \pi(T)$. Then for some
    $M \in \End{\R^{\adim}}$
    \begin{gather}
        P = T^{\perp} \circ M \circ T + \transpose{\bigl( T^{\perp} 
          \circ M \circ T \bigr)} \,;
        \\
        \text{hence,} \quad
        \transpose{P} = P \,,
        \quad
        T \circ P \circ T = 0 \,,
        \quad
        T^{\perp} \circ P \circ T^{\perp} = 0 \,;
    \end{gather}
    thus, $P \in \Tan(\pgrass{\adim}{\vdim},T)$. Now assume
    $P \in \Tan(\pgrass{\adim}{\vdim},T)$. Then, in view of~\eqref{eq:tan_grass},
    \begin{multline}
        P = T^{\perp} \circ P \circ T + T^{\perp} \circ P \circ T^{\perp}
        + T \circ P \circ T + T \circ P \circ T^{\perp}
        \\
        = T^{\perp} \circ P \circ T + T \circ P \circ T^{\perp}
        = T^{\perp} \circ P \circ T + \transpose{(T^{\perp} \circ P \circ T)}
        = \pi(T)P \in \im \pi(T) \,.
    \end{multline}

    To prove~\eqref{i:im_sigma}, note that
    \begin{displaymath}
        \lin \{T\} \perp \Tan(\pgrass{\adim}{\vdim},T) \,,
        \quad
        \pi(T)T = 0 \,,
        \quad \text{and} \quad
        \sigma(T)T = T \,;
    \end{displaymath}
    hence,
    \begin{displaymath}
        \Tan(\ray \pgrass{\adim}{\vdim},T)
        = \lin\{T\} + \Tan(\pgrass{\adim}{\vdim},T)
        = \im \sigma(T) \,.
    \end{displaymath}

    To prove~\eqref{i:pi_projection}, compute
    \begin{multline}
        \bigl( \pi(T) \circ \pi(T) \bigr) L
        = \pi(T)\bigl(
        T^{\perp} \circ L \circ T + \transpose{
          \bigl( T^{\perp} \circ L \circ T \bigr)}
        \bigr) \\
        =T^{\perp} \circ \bigl(
        T^{\perp} \circ L \circ T + \transpose{
          \bigl( T^{\perp} \circ L \circ T \bigr)}
        \bigr) \circ T + \transpose{
          \bigl( T^{\perp} \circ \bigl(
          T^{\perp} \circ L \circ T + \transpose{
            \bigl( T^{\perp} \circ L \circ T \bigr)}
          \bigr) \circ T \bigr)} \\
        =T^{\perp} \circ L \circ T + \transpose{
          \bigl( T^{\perp} \circ L \circ T \bigr)} = \pi(T)L \,.
    \end{multline}

    To prove~\eqref{i:sigma_projection}, observe that
    $T \bullet T = \trace \bigl( T \circ \transpose{T} \bigr) = \vdim$,
    $\pi(T)L \bullet T = 0$, and $\pi(T)T = 0$; thus,
    \begin{multline}
        \bigl( \sigma(T) \circ \sigma(T) \bigr) L =
        \sigma(T) \bigl(\sigma(T)L\bigr) =
        \sigma(T)\bigl( \tfrac 1k (L \bullet T) T + \pi(T)L \bigr) \\
        =\tfrac 1k \bigl(\bigl(
        \tfrac 1k (L \bullet T) T + \pi(T)L
        \bigr) \bullet T\bigr) T + \pi(T)\bigl( 
        \tfrac 1k (L \bullet T) T + \pi(T)L
        \bigr) \\
        =\tfrac{1}{\vdim^2}(L \bullet T)(T \bullet T)T +
        \tfrac{1}{\vdim}(\pi(T)L \bullet T)T +
        \tfrac{1}{\vdim}(L \bullet T)\pi(T)T +
        \pi(T)\bigl( \pi(T)L \bigr) \\
        =\tfrac{1}{\vdim}(L \bullet T)T + \pi(T)L = \sigma(T)L \,.
    \end{multline}

    For the proof of~\eqref{i:pi_transpose}, let $S
    \in \End{\R^{\adim}}$. By~the~definition of the adjoint operator we have
    \begin{multline}
        \pi(T)L \bullet S
        = \bigl(T^{\perp} \circ L \circ T\bigr) \bullet S + 
        \transpose{\bigl( T^{\perp} \circ L \circ T \bigr)} \bullet S
        \\
        = \trace\bigl(\transpose{S} \circ T^{\perp} \circ L \circ T\bigr)+
        \trace\bigl(T^{\perp} \circ L \circ T \circ S\bigr)
        \\
        = \trace\bigl(L \circ T \circ \transpose{S} \circ T^{\perp}\bigr)+
        \trace\bigl(L \circ T \circ S \circ T^{\perp}\bigr)
        \\
        = \trace\bigl(L \circ \bigl(
        T \circ (S + \transpose{S}) \circ T^{\perp}
        \bigr)\bigr)
        = L \bullet \bigl(
        T^{\perp} \circ (S + \transpose{S}) \circ T
        \bigr)
        = L \bullet \transpose{\pi(T)}S \,.
    \end{multline}
    Since $S$ was arbitrary, we obtain $\transpose{\pi(T)}L = 
    T^{\perp} \circ (L + \transpose{L}) \circ T$.

    To prove~\eqref{i:sigma_transpose}, let $S \in \End{\R^{\adim}}$. Again, by the
    definition of the adjoint operator we have
    \begin{multline}
        \sigma(T)L \bullet S
        = \tfrac{1}{\vdim}(L \bullet T)(T \bullet S) +
        \pi(T)L \bullet S
        = \tfrac{1}{\vdim}(L \bullet T)(T \bullet S) +
        L \bullet \transpose{\pi(T)}S
        \\
        = L \bullet \bigl( \tfrac{1}{\vdim}(S \bullet T)T + \transpose{\pi(T)}S \bigr)
        = L \bullet \transpose{\sigma(T)}S \,.
    \end{multline}
    Finally,
    \begin{displaymath}
        \transpose{(\transpose{\sigma(T)}L)}
        = \transpose{\bigl(
          \tfrac{1}{\vdim}(L \bullet T)T + T^{\perp} \circ (L + \transpose{L}) \circ T
          \bigr)}
        = \tfrac{1}{\vdim}(L \bullet T)T + T \circ (L + \transpose{L}) \circ T^{\perp} \,.
    \end{displaymath}

    Assume $L \bullet T = \vdim$, $X = L + \transpose{L}$,
    and $P = \transpose{(\transpose{\sigma(T)}L)}$. We get~\eqref{i:projections}
    from~\eqref{i:sigma_transpose},
    \begin{displaymath}
        P \circ P =
        \bigl( T + T \circ X \circ T^{\perp} \bigr)
        \circ \bigl( T + T \circ X \circ T^{\perp} \bigr)
        = T + T \circ X \circ T^{\perp}
        = P \,,
        \quad \text{and} \quad
        P \circ T = T \,.
    \end{displaymath}

    If $X, Y \in \End{\R^{\adim}}$ are self-adjoint, then
    $\pi(T)X = T^{\perp} \circ X \circ T + T \circ X \circ T^{\perp}$ is self-adjoint and, since $T$
    and~$T^{\perp}$ are self-adjoint,
    \begin{displaymath}
        \pi(T)X \bullet Y
        = \trace \bigl( X \circ T \circ Y \circ T^{\perp} + X \circ T^{\perp} \circ Y \circ T \bigr)
        = X \bullet \pi(T)Y \,.
    \end{displaymath}
    This gives~\eqref{i:pi_self_adjoint}; together with~\eqref{i:pi_projection} it shows that the
    restriction of~$\pi(T)$ to the self-adjoint endomorphisms is an~orthogonal projection.
\end{remark}

\begin{remark}
    \label{rem:BF_gradient}
    Let $T \in \pgrass{\adim}{\vdim}$.
    From~\ref{rem:sigma_props}\eqref{i:sigma_projection}\eqref{i:im_sigma} it follows that
    $\sigma(T)$ is a~projection onto $\Tan(\ray \pgrass{\adim}{\vdim},T)$ and that
    \begin{displaymath}
        B_F(T) = \transpose{\sigma(T)} \grad F(T) \,.
    \end{displaymath}
    In~particular, $B_F(T)$ does not depend on the values of~$F$ outside
    $\ray \pgrass{\adim}{\vdim}$. Moreover, employing the Whitney extension
    theorem~\cite[3.1.14]{Federer1969} (cf. the proof of the existence of a~tubular neighbourhood
    in~\cite[Chapter~4, proof of~5.1, p.~109]{Hirsch1976}, where the appropriate map is obtained
    by also employing the implicit function theorem), there exist an open neighbourhood~$U$
    of~$\ray \pgrass{\adim}{\vdim} \without \{0\}$ and a~smooth map
    $\xi : U \to \ray \pgrass{\adim}{\vdim}$ such that
    \begin{displaymath}
        \xi \circ \xi = \xi \,,
        \quad
        \xi(\lambda T) = \lambda T \,,
        \quad \text{and} \quad
        \uD \xi(\lambda T) = \sigma(T)
        \quad \text{for $T \in \pgrass{\adim}{\vdim}$ and $0 < \lambda < \infty$} \,;
    \end{displaymath}
    hence, if $T \in \pgrass{\adim}{\vdim}$, then
    \begin{displaymath}
        B_F(T) = \grad (F \circ \xi)(T)
        \quad \text{and} \quad
        \transpose{P}_F(T) = \grad (\log \circ F \circ \xi)(T) 
    \end{displaymath}
    so that $B_F|\pgrass{\adim}{\vdim}$ and $\transpose{P}_F|\pgrass{\adim}{\vdim}$
    admit gradient extensions to a~neighbourhood of~$\pgrass{\adim}{\vdim}$.
\end{remark}

\begin{remark}
    \label{rem:se-cse}
    Recalling~\ref{rem:sigma_props}~\eqref{i:projections} and~\ref{rem:integrand-k-homogeneous} we
    observe that for $T \in \pgrass{\adim}{\vdim}$ we have
    \begin{gather}
        \grad (\log \circ F)(T) \bullet T = \vdim \,,
        \quad
        P_F(T) \circ P_F(T) = P_F(T) = T \circ P_F(T) \,,
        \quad
        P_F(T) \circ T = T \,;
        \\
        \transpose{P}_F(T) = P_F(T) \quad \iff \quad \uD F(T)|\Tan(\pgrass{\adim}{\vdim},T) = 0 \,,
        \quad
        \trace P_F(T) = P_F(T) \bullet T = \vdim \,;
    \end{gather}
    thus, $P_F(T)$ is a projection onto~$\im T$; furthermore, it follows that $P_F(T)$ is orthogonal
    if and only if $\uD F(T)|\Tan(\pgrass{\adim}{\vdim},T) = 0$.
\end{remark}

\begin{remark}
    \label{rem:GF_on_hyperplane}
    From~\ref{rem:se-cse} we conclude that $\GF{F} = \transpose{P}_F \lIm \pgrass{\adim}{\vdim}
    \rIm$ lies on a~hyperplane not passing through the origin; namely,
    \begin{displaymath}
        \GF{F} \subseteq \End{\R^{\adim}}
        \cap \bigl\{ A : A \bullet \id{\R^{\adim}} = \vdim \bigr\} \,.
    \end{displaymath}
    Therefore, applying~\ref{lem:base_of_cone}\ref{i:boc:extreme}\ref{i:boc:exposed} with
    $\alpha = [ \End{\R^{\adim}} \ni A \mapsto (A \bullet \id{\R^{\adim}})/\vdim ]$ and $\GF{F}$ in
    place of~$A$, we see that $\GF{F}$ is extreme [exposed] if and only if $\ray \{A\}$ is
    an~extreme [exposed] ray of $\cone \GF{F}$ for all $A \in \GF{F}$. Moreover, $F(T)>0$ gives
    \begin{displaymath}
        \ray \transpose{P}_F \lIm \pgrass{\adim}{\vdim} \rIm = \ray B_F \lIm \pgrass{\adim}{\vdim} \rIm = \ray \GF{F} \,;
    \end{displaymath}
    thus, $\GF{F}$ is extreme [exposed] if and only if $\ray \{ B_F(T) \}$ is an extreme [exposed]
    ray of $\cone B_F\lIm \pgrass{\adim}{\vdim} \rIm$ for each $T \in \pgrass{\adim}{\vdim}$. Since
    taking adjoints is a~linear transformation of~$\End{\R^{\adim}}$ (actually a symmetry) we get
    also that $\GF{F}$ is extreme [exposed] if and only if $P_F \lIm \pgrass{\adim}{\vdim} \rIm$ is
    extreme [exposed].
\end{remark}

\begin{corollary}
    \label{cor:extreme-equivalence}
    The following are equivalent:
    \begin{enumerate}
    \item $P_F \lIm \pgrass{\adim}{\vdim} \rIm$ is extreme [exposed];
    \item $\GF{F}$ is extreme [exposed];
    \item for each $T \in \pgrass{\adim}{\vdim}$ the set $\ray \{ B_F(T) \}$ is an extreme [exposed]
        ray of $\cone B_F\lIm \pgrass{\adim}{\vdim} \rIm$.
    \end{enumerate}
\end{corollary}

\section{\AC{} in terms of convex geometry}
\label{sec:ac_by_convex_geometry}

\begin{miniremark}
    This section contains the reformulation of the atomic condition on which the rest of the paper
    rests: $F \in \AC{}$ if and only if $\GF{F}$ is precisely the set of extreme points
    of~$\conv \GF{F}$ and no~other point of that hull has rank at most~$\vdim$;
    see~\ref{thm:AC_extreme}. Replacing ``extreme'' by ``exposed'' here leads to
    the~\emph{exposed condition}~\EC{} of~\ref{def:EC}, which we compare with the~\emph{scalar
      atomic condition}~\SAC{} of De~Rosa and Tione. We~introduce the \emph{strong exposed
      condition}~\SEC{} of~\ref{def:SEC}, in which the exposing hyperplanes are required to
    contain~$\imsp(T)$, and prove the inclusions
    $\SAC{} \subseteq \SEC{} \subseteq \EC{} \subseteq \AC{}$ (see~\ref{prop:SAC_in_EC}
    and~\ref{prop:EC_in_AC}) and
    that the last two are equalities in codimension~one (see~\ref{thm:AC=EC-codim1}); whether they
    are equalities in general is~\ref{q:EC=AC}. As
    a~by-product we obtain a~rigidity property of \AC{} integrands in~\ref{prop:AC_imPF_homeo}.
    The~section closes with~\ref{ex:J_fields}, where we~attach to each $J \in \End{\R^{\adim}}$
    satisfying $Jx \bullet x > 0$ for $x \ne 0$ a~field of projections~$P^{J}$ admitting normals
    with the properties of those in~\ref{prop:existence_of_normals} and prove that $P^{J} = P_F$
    for some integrand~$F$ if and only if $J$ is symmetric -- in which case $F$ is a~positive multiple of the area integrand of $[(x,y)\mapsto J^{-1}x\bullet y]$.
\end{miniremark}

\begin{miniremark}
    In this section $F$ is always a $\vdim$-integrand of class~$\cnt{1}$ away from the
    origin. By a~\emph{probability measure} over some space~$X$ we mean a~Radon measure
    $\mu$ over $X$ satisfying $\mu(X) = 1$.
\end{miniremark}

\begin{definition}[\protect{cf.~\cite[Definition~1.1]{DePhilippis2018}}]
    \label{def:AC}
    We say that a~$\vdim$-integrand $F$ satisfies the \emph{atomic condition}, and write
    $F \in \AC{}$, if given any probability measure $\mu$ over $\pgrass{\adim}{\vdim}$ and
    setting
    \begin{displaymath}
        A(\mu) = \textint{}{} \transpose{P}_F \ud \mu \in \End{\R^{\adim}} \,,
    \end{displaymath}
    we have
    \begin{enumerate}
    \item
        \label{i:AC:dim}
        $\dim \ker A(\mu) \le \adim - \vdim$,
    \item
        \label{i:AC:dirac}
        if $\dim \ker A(\mu) = \adim - \vdim$, then $\mu = \Dirac{T}$ for some
        $T \in \pgrass{\adim}{\vdim}$.
    \end{enumerate}
\end{definition}

\begin{remark}
    \label{rem:AC_orig}
    In~\cite[Definition~1.1]{DePhilippis2018} the authors use $\bar{A}(\mu) = \textint{}{} B_F \ud \mu$ in place of
    $A(\mu)$ as above, but this makes no difference because
    \begin{displaymath}
        \bar{A}(\mu)
        = \textint{}{} B_F \ud \mu
        = \textint{}{} \transpose{P}_F(T) F(T) \ud \mu(T)
        = \Lpnorm{\mu}{1}{F} \textint{}{} \transpose{P}_F(T) \ud \nu(T)
        = \Lpnorm{\mu}{1}{F}\,A(\nu) \,,
    \end{displaymath}
    where $\nu$ is the probability measure satisfying
    $\nu(A) = \Lpnorm{(\mu \restrict A)}{1}{F} / \Lpnorm{\mu}{1}{F}$ whenever
    $A \subseteq \pgrass{\adim}{\vdim}$ is Borel.\footnote{As in~\cite[2.4.12]{Federer1969} we use
      the notation $\Lpnorm{\mu}{1}{f} = \int |f| \ud \mu$.}
\end{remark}

\begin{theorem}
    \label{thm:AC_extreme}
    Let $E = \extreme{(\conv \GF{F})}$. There holds
    \begin{displaymath}
        F \in \AC{}
        \quad \iff \quad
        \conv \GF{F} \cap \bigl\{ A : \dim \im A \le \vdim \bigr\}
        = \GF{F} = E \,.
    \end{displaymath}
\end{theorem}

\begin{proof}
    Since any probability measure over $\pgrass{\adim}{\vdim}$ is a (weak) limit
    of convex combinations of atomic measures and $B_F$ is continuous, we see
    that
    \begin{displaymath}
        \bigl\{ A(\nu) : \nu \text{ a probability measure over $\pgrass{\adim}{\vdim}$} \bigr\}
        = \conv \GF{F} \,.
    \end{displaymath}

    Assume $F \in \AC{}$ and let $\mu$ be a probability measure over
    $\pgrass{\adim}{\vdim}$ such that
    $A(\mu) \in \bigl\{ A : \dim \im A \le \vdim \bigr\}$. Then
    $\dim \im A(\mu) \le \vdim$. Since $F \in \AC{}$, we get $\dim \im A(\mu) = \vdim$
    and $A(\mu) = \transpose{P}_F(T)$ for some $T \in \pgrass{\adim}{\vdim}$; hence
    $A(\mu) \in \GF{F}$. This shows that
    $\conv \GF{F} \cap \bigl\{ A : \dim \im A \le \vdim \bigr\} \subseteq
    \GF{F}$. The reverse inclusion holds because $\dim \im P = \vdim$ for any
    $P \in \GF{F}$ and $\GF{F} \subseteq \conv \GF{F}$.

    Using~\cite[Corollary~18.3.1]{Rockafellar1970} we see that $E \subseteq \GF{F}$.
    We~will show the reverse inclusion, i.e., that every point of $\GF{F}$ is an
    extreme point of $\conv \GF{F}$. Assume $X,Y \in \conv \GF{F}$,
    $0 < \lambda < 1$, $T \in \pgrass{\adim}{\vdim}$, and
    $\transpose{P}_F(T) = \lambda X + (1-\lambda)Y$. Since
    $X,Y \in \conv \GF{F}$, there exist probability measures $\nu_X$ and $\nu_Y$
    over $\pgrass{\adim}{\vdim}$ such that $A(\nu_X) = X$ and $A(\nu_Y) = Y$. Set
    $\eta = \lambda \nu_X + (1-\lambda)\nu_Y$. Observe that
    \begin{enumerate}
    \item $\spt \nu_X \cup \spt \nu_Y = \spt \eta$,
    \item $\eta$ is a probability measure over $\pgrass{\adim}{\vdim}$,
    \item and $A(\eta) = \transpose{P}_F(T)$.
    \end{enumerate}
    Since $F \in \AC{}$, we see that $\eta = \Dirac{T}$, which means that
    $\nu_X = \nu_Y = \Dirac{T}$; hence, $X = Y = \transpose{P}_F(T)$ is an
    extreme point of $\conv \GF{F}$.

    Conversely, assume the displayed equalities and let $\mu$ be a probability measure over
    $\pgrass{\adim}{\vdim}$. If $\dim \ker A(\mu) \ge \adim - \vdim$, then
    $\dim \im A(\mu) \le \vdim$, so $A(\mu) \in \GF{F}$ and $A(\mu) = \transpose{P}_F(T)$ for some
    $T \in \pgrass{\adim}{\vdim}$, which shows $\dim \ker A(\mu) = \adim - \vdim$; moreover, since
    $\transpose{P}_F$ is injective (by~\ref{rem:se-cse} we have $\eqproject{\im P_F(T)} = T$) and
    $\transpose{P}_F(T)$ is an extreme point of $\conv \GF{F}$, its representing measure
    is carried by $\{T\}$. Indeed, otherwise a~linear functional non-constant on
    $\transpose{P}_F\lIm\spt\mu\rIm$ splits $\mu$ at an~intermediate value into two non-zero
    parts with distinct barycentres. Their convex combination is $A(\mu)$, contrary to
    extremality. Hence $\mu = \Dirac{T}$.
\end{proof}

\begin{proposition}
    \label{prop:AC_imPF_homeo}
    Assume $F \in \AC$. Then the map
    $\varsigma : \pgrass{\adim}{\vdim} \to \pgrass{\adim}{\vdim}$ given by
    \begin{displaymath}
        \varsigma(S) = \project{(\im \transpose{P}_F(S))}
        \quad \text{for $S \in \pgrass{\adim}{\vdim}$}
    \end{displaymath}
    is injective and hence a~homeomorphism.
\end{proposition}

\begin{proof}
    Recall~\ref{def:image_space}. Assume $S,R,T \in \pgrass{\adim}{\vdim}$ are such that
    $\varsigma(S) = \varsigma(R) = T$.  Then
    \begin{displaymath}
        \transpose{P}_F(S) \in \GF{F} \cap \imsp(T)
        \quad \text{and} \quad
        \transpose{P}_F(R) \in \GF{F} \cap \imsp(T) \,;
    \end{displaymath}
    hence
    \begin{displaymath}
        \conv \bigl\{ \transpose{P}_F(S) ,\, \transpose{P}_F(R) \bigr\}
        \subseteq \conv \GF{F} \cap \imsp(T)
        \subseteq \conv \GF{F} \cap \{ A : \dim \im A \le \vdim\}
        = \GF{F} \,.
    \end{displaymath}
    The midpoint of $\transpose{P}_F(S)$ and $\transpose{P}_F(R)$ is therefore an~extreme
    point of $\conv \GF{F}$; hence $\transpose{P}_F(S) = \transpose{P}_F(R)$. Recall that
    $\im P_F(S) = \im S$ and $\im P_F(R) = \im R$; hence, $S = R$ and $\varsigma$
    is~injective.

    Observe that $\varsigma$ is continuous because $F$ is of class~$\cnt{1}$. Since
    $\pgrass{\adim}{\vdim}$ is compact, $\varsigma$ is a~homeomorphism onto its image
    $\im \varsigma$, which is compact. By the invariance of domain theorem (see,
    e.g.,~\cite[4.2.12]{Engelking1992}), the image is also open. Since
    $\pgrass{\adim}{\vdim}$ is connected, $\im \varsigma = \pgrass{\adim}{\vdim}$.
\end{proof}

\begin{corollary}
    \label{cor:kerPF_homeo}
    If $F \in \AC{}$, then
    \begin{displaymath}
        [S \mapsto \ker P_F(S)] : \pgrass{\adim}{\vdim} \to \grass{\adim}{\codim}
    \end{displaymath}
    is a~homeomorphism.
\end{corollary}

\begin{proof}
    For $S \in \pgrass{\adim}{\vdim}$,
    $\ker P_F(S) = (\im \transpose{P}_F(S))^{\perp}$; apply~\ref{prop:AC_imPF_homeo}
    and the homeomorphism taking a~subspace to its orthogonal complement.
\end{proof}

\begin{miniremark}
    \label{mr:reason_for_SAC}
    De Rosa and Tione devised the following definition while studying the regularity of graphs with
    $p$-integrable mean $F$-curvature; see~\cite[Definition~3.3]{DeRosa2022}.
\end{miniremark}

\begin{definition}[\protect{cf.~\cite[Definition~3.3]{DeRosa2022}}]
    \label{def:SAC}
    A~$\vdim$-integrand~$F$ is said to satisfy the \emph{scalar atomic condition}
    (denoted $F \in \SAC{}$) if
    \begin{displaymath}
        \transpose{P}_F(T) \bullet Q_F(S) > 0
        \quad \text{for $S,T \in \pgrass{\adim}{\vdim}$ with $S \ne T$} \,.
    \end{displaymath}
\end{definition}

\begin{miniremark}
    \label{mr:SAC_not_natural}
    This means that any point $P_F(T) \in \transpose{\GF{F}}$ is exposed in $\conv
    \transpose{\GF{F}}$ by the vector~$\transpose{Q}_F(T)$. The choice of this vector is
    arbitrary; hence, we propose a more general class of integrands.
\end{miniremark}

\begin{definition}
    \label{def:EC}
    We say that a~$\vdim$-integrand~$F$ satisfies the \emph{exposed condition}, and write
    $F \in \EC{}$, if
    \begin{displaymath}
        \conv \GF{F} \cap \bigl\{ A : \dim \im A \le \vdim \bigr\}
        = \GF{F} = \exposed{(\conv \GF{F})} \,.
    \end{displaymath}
\end{definition}

\begin{lemma}
    \label{lem:EC_normals}
    Let $F$ be an integrand satisfying at least one of the following conditions.
    \begin{enumerate}
    \item
        \label{i:ECn:normal}
        There exists $N_F : \pgrass{\adim}{\vdim} \to \End{\R^{\adim}}$ such that
        \begin{gather}
            T \circ N_F(T) = 0
            \quad \text{for $T \in \pgrass{\adim}{\vdim}$}
            \\
            \text{and} \quad
            N_F(T) \bullet P_F(S) > 0
            \quad \text{for $S,T \in \pgrass{\adim}{\vdim}$ with $S \ne T$} \,.
        \end{gather}
    \item
        \label{i:ECn:normal2}
        There exists $M_F : \pgrass{\adim}{\vdim} \to \End{\R^{\adim}}$ such that
        \begin{gather}
            M_F(S) \circ S = 0
            \quad \text{for $S \in \pgrass{\adim}{\vdim}$}
            \\
            \text{and} \quad
            M_F(S) \bullet P_F(T) > 0
            \quad \text{for $S,T \in \pgrass{\adim}{\vdim}$ with $S \ne T$} \,.
        \end{gather}
    \end{enumerate}
    Then $F \in \EC{}$.
\end{lemma}

\begin{proof}
    Assume~\ref{i:ECn:normal} holds. Set
    $H(T) = \End{\R^{\adim}} \cap \{ A : A \bullet N_F(T) \ge 0 \}$ for
    $T \in \pgrass{\adim}{\vdim}$. Then
    \begin{equation}
        \label{eq:ECn:NF_props}
        \transpose{\GF{F}} \subseteq H(T)
        \quad \text{and} \quad
        \transpose{\GF{F}} \cap \Bdry{H(T)} = \bigl\{ P_F(T) \bigr\}
        \quad \text{for $T \in \pgrass{\adim}{\vdim}$} \,.
    \end{equation}
    Let $T \in \pgrass{\adim}{\vdim}$ and set $K = \conv \transpose{\GF{F}} \cap \Bdry{H(T)}$. Note
    that $K$ is a closed face of $\conv \transpose{\GF{F}}$ and
    \begin{displaymath}
        \extreme K \subseteq ( \extreme \conv \transpose{\GF{F}} ) \cap \Bdry{H(T)}
        \subseteq \transpose{\GF{F}} \cap \Bdry{H(T)}
        = \bigl\{ P_F(T) \bigr\} \,;
    \end{displaymath}
    hence, $\Bdry{H(T)} \cap \conv \transpose{\GF{F}} = \{ P_F(T) \}$, and we see that
    $P_F(T) \in \exposed \conv \transpose{\GF{F}}$.

    Let $A \in \conv \transpose{\GF{F}}$ satisfy $\dim\im A \le \vdim$ and choose
    $T \in \pgrass{\adim}{\vdim}$ with $\im A \subseteq \im T$, so $A = T \circ A$.
    Since $T \circ N_F(T) = 0$, we
    get
    \begin{displaymath}
        A \bullet N_F(T) = \trace \bigl( \transpose{A} \circ N_F(T) \bigr)
        = \trace \bigl( \transpose{A} \circ T \circ N_F(T) \bigr) = 0 \,,
    \end{displaymath}
    which shows that $A \in \Bdry{H(T)} \cap \conv \transpose{\GF{F}} = \{ P_F(T) \}$;
    thus, $F \in \EC{}$.

    Assume~\ref{i:ECn:normal2}, fix $T \in \pgrass{\adim}{\vdim}$, and put
    $P = P_F(T)$ and $R = \project{(\im\transpose{P})}$. Since $P = P \circ R$,
    \begin{displaymath}
        M_F(R) \bullet P
        = \trace(R \circ \transpose{P} \circ M_F(R))
        = \trace(\transpose{P} \circ M_F(R) \circ R) = 0 \,.
    \end{displaymath}
    Thus $R=T$ by~\ref{i:ECn:normal2}. Since $P \circ T=T$, we get $P=T$.
    Consequently, $N_F(T)=T^{\perp}$ satisfies~\ref{i:ECn:normal}, because
    $T^{\perp}\bullet S = \tfrac12|S-T|^2>0$ when $S\ne T$.
\end{proof}

\begin{definition}
    \label{def:SEC}
    We say that a~$\vdim$-integrand~$F$ satisfies the \emph{strong exposed condition}, and write
    $F \in \SEC{}$, if condition~\ref{lem:EC_normals}\ref{i:ECn:normal} holds, i.e., there exists
    $N_F : \pgrass{\adim}{\vdim} \to \End{\R^{\adim}}$ such that
    \begin{gather}
        T \circ N_F(T) = 0
        \quad \text{for $T \in \pgrass{\adim}{\vdim}$}
        \\
        \text{and} \quad
        N_F(T) \bullet P_F(S) > 0
        \quad \text{for $S,T \in \pgrass{\adim}{\vdim}$ with $S \ne T$} \,.
    \end{gather}
\end{definition}

\begin{proposition}
    \label{prop:SAC_in_EC}
    The inclusions $\SAC{} \subseteq \SEC{} \subseteq \EC{}$ hold.
\end{proposition}

\begin{proof}
    Let $F \in \SAC{}$. From~\ref{def:aux-objects} and~\ref{rem:sigma_props}\eqref{i:projections}
    deduce that $Q_F(T) = Q_F(T) \circ T^{\perp}$ for $T \in \pgrass{\adim}{\vdim}$; hence,
    $N_F(T) = \transpose{Q}_F(T)$ satisfies $T \circ N_F(T) = 0$ and, by~\ref{def:SAC},
    $N_F(T) \bullet P_F(S) = \transpose{P}_F(S) \bullet Q_F(T) > 0$ for
    $S,T \in \pgrass{\adim}{\vdim}$ with $S \ne T$. Thus $F \in \SEC{}$. The inclusion
    $\SEC{} \subseteq \EC{}$ is~\ref{lem:EC_normals}.
\end{proof}

\begin{proposition}
    \label{prop:existence_of_normals}
    Assume $F \in \AC{}$.
    Let $T \in \pgrass{\adim}{\vdim}$ and set $S = \eqproject{\im \transpose{P}_F(T)}$. Then there
    exist unit vectors $N_F(T) \in \imsp(T^{\perp})$ and $M_F(S) \in \transpose{\imsp(S^{\perp})}$
    such that the hyperplanes $\End{\R^{\adim}} \cap \{ A : A \bullet N_F(T) = 0 \}$ and
    $\End{\R^{\adim}} \cap \{ A : A \bullet M_F(S) = 0 \}$ contain $\imsp(T)$
    and~$\transpose{\imsp(S)}$ respectively, are disjoint from $\rInt \cone \transpose{\GF{F}}$,
    and satisfy
    \begin{displaymath}
        \cone \transpose{\GF{F}} \subseteq
        \bigl\{ A : A \bullet M_F(S) \ge 0 \bigr\}
        \cap \bigl\{ A : A \bullet N_F(T) \ge 0 \bigr\} \,.
    \end{displaymath}
\end{proposition}

\begin{proof}
    Note that $P_F(T) = T \circ P_F(T) \circ S$ and
    recall~\ref{def:image_space} and~\ref{rem:GF_on_hyperplane}. Since elements of
    $\imsp(T) \cup \transpose{\imsp(S)}$ are of rank at most~$\vdim$, we~get
    from~\ref{thm:AC_extreme} and the injectivity in~\ref{prop:AC_imPF_homeo}
    \begin{equation}
        \label{eq:two_normals_for_GF}
        \cone \transpose{\GF{F}} \cap \imsp(T) = \ray \bigl\{ P_F(T) \bigr\}
        = \cone \transpose{\GF{F}} \cap \transpose{\imsp(S)} \,.
    \end{equation}
    By~\ref{thm:AC_extreme} and~\ref{lem:base_of_cone}, the common ray in
    \eqref{eq:two_normals_for_GF} is a~face of $\cone\transpose{\GF{F}}$.
    It is proper, since $P_F$ is injective and the Grassmannian has more than one point.
    Hence it misses the relative interior by~\cite[Corollary~18.1.3]{Rockafellar1970}.
    Employ~\cite[Theorem~11.2]{Rockafellar1970} to extend $\imsp(T)$ and $\transpose{\imsp(S)}$ to
    hyperplanes which are disjoint from $\rInt \cone \transpose{\GF{F}}$, and let $N_F(T)$
    and~$M_F(S)$ be the unit normal vectors of these hyperplanes, oriented so that the displayed
    inclusion holds. From~\ref{rem:basic_props_of_IT} we conclude that
    $N_F(T) \in \imsp(T^{\perp})$ and $M_F(S) \in \transpose{\imsp(S^{\perp})}$.
\end{proof}

\begin{remark}
    \label{rem:who_satisfies_half_AC}
    By~\ref{thm:AC_extreme}, every $F \in \AC{}$ satisfies
    \begin{equation}
        \label{eq:EoN:half_AC}
        \conv\GF{F} \cap \{A:\dim\im A\le\vdim\}=\GF{F} \,.
    \end{equation}
    The separation argument in~\ref{prop:existence_of_normals} also uses the extremality
    supplied by~\ref{thm:AC_extreme}.
    Under the full atomic condition, \ref{prop:existence_of_normals} gives the weak
    inequalities corresponding to~\ref{lem:EC_normals}\ref{i:ECn:normal}\ref{i:ECn:normal2}.
\end{remark}

\begin{remark}
    \label{rem:AC1_not_half_of_AC}
    In~\cite[Definition~3.1]{DeRosa2022} the authors call the condition~\ref{def:AC}\ref{i:AC:dim}
    ``(AC1)''. Note that~\ref{def:AC}\ref{i:AC:dim} can be restated as
    \begin{displaymath}
        \conv \GF{F} \cap \bigl\{ A : \dim \im A \le \vdim-1 \bigr\} = \varnothing \,,
    \end{displaymath}
    which follows from~\ref{rem:who_satisfies_half_AC}\eqref{eq:EoN:half_AC} but still implies
    rectifiability of the \emph{weight measure} of any varifold whose anisotropic first variation
    has Radon total variation and whose density is bounded from below; cf.~\cite[p.~656
    item~(b)]{ArroyoRabasa2019}. Later, De~Rosa and Tione showed that for any $1 < p < \infty$ the
    $p$-norm on $\extpower{2}\R^{4}$ gives rise to an integrand
    satisfying~\ref{def:AC}\ref{i:AC:dim}; see~\cite[Theorem~6.1]{DeRosa2022}. More recently,
    De~Rosa, Firester, and Tsiamis proved that the $q$-norm on $\extpower{\vdim}\R^{\adim}$ gives
    rise to an integrand satisfying~\USAC{}, hence~\AC{}, whenever
    $2 - \tfrac{7}{10 \min \{ \vdim, \adim - \vdim \}} \le q \le 2$;
    see~\cite[Theorem~1.3]{DeRosa2026v1}.
\end{remark}

\begin{remark}
    \label{rem:specific_hyperplane_for_SAC}
    Compare~\ref{def:SAC} and~\ref{def:EC}. According to~\ref{def:SAC}, $F$ satisfies \SAC{} if, for
    each $T \in \pgrass{\adim}{\vdim}$, the very specific hyperplane
    $\End{\R^\adim} \cap \{ A : A \bullet Q_F(T) = 0 \}$ meets $\GF{F}$ only at
    $\transpose{P}_F(T)$. It seems that the choice of~the vector~$Q_F(T)$ is completely
    arbitrary when compared with the normal choices admitted by~\ref{def:SEC}. In
    section~\ref{sec:uniform_AC} we obtain results analogous to those of~\cite{DeRosa2022} for
    integrands satisfying a uniform version of~\SEC{}.
\end{remark}

\begin{proposition}
    \label{prop:EC_in_AC}
    The inclusion $\EC{} \subseteq \AC{}$ holds.
\end{proposition}

\begin{proof}
    Note that $\GF{F}$ is compact as the image of a~compact set~$\pgrass{\adim}{\vdim}$ under
    a~continuous map~$\transpose{P}_F$. Recalling~\ref{thm:Straszewicz}, one obtains
    \begin{displaymath}
        \GF{F} = \exposed \conv \GF{F}
        \subseteq \extreme \conv \GF{F}
        \subseteq \Clos{(\exposed \conv \GF{F})}
        = \GF{F} \,.
        \qedhere
    \end{displaymath}
\end{proof}

\begin{remark}
    \label{rem:EC=AC?}
    In general, not all extreme points of a~convex set are exposed, as the convex hull of
    a~\emph{stadium curve} in~$\R^2$ shows; a~priori, therefore, condition \EC{} seems strictly
    stronger than \AC. Recall, however, that exposed points are always dense in the set of
    extreme points; cf.~\ref{thm:Straszewicz}. If $\adim = \vdim+1$, the two conditions agree,
    and even $\SEC{} = \AC{}$; see~\ref{thm:AC=EC-codim1}.
\end{remark}

\begin{question}
    \label{q:EC=AC}
    Is $\SEC{} = \AC{}$? Is $\EC{} = \AC{}$?
\end{question}

\begin{miniremark}
    \label{mr:rank_one_maps}
    Let $X$ be a vectorspace, $\omega \in X^*$, and $\zeta \in X$. We adopt Allard's convention and
    write $\omega \cdot \zeta$ for the rank-one map such that $(\omega \cdot \zeta)(\xi) =
    \omega(\xi) \zeta$ for $\xi \in X$; cf.~\cite[2.2, 2.3]{Allard1972}.
\end{miniremark}

\begin{theorem}
    \label{thm:AC=EC-codim1}
    If $\adim = \vdim+1$, then $\SEC{} = \EC{} = \AC{}$.
\end{theorem}

\begin{proof}
    The inclusions $\SEC{} \subseteq \EC{} \subseteq \AC{}$ are~\ref{prop:SAC_in_EC}
    and~\ref{prop:EC_in_AC}.

    Assume $F \in \AC{}$. Let $P$ be the set of all even probability measures $\nu$ over the sphere
    $\sphere{\adim-1}$, i.e., those satisfying $\nu(E) = \nu(\scale{-1} \lIm E \rIm)$ for
    $E \subseteq \sphere{\adim-1}$, where $\scale{-1}$ is defined in~\ref{def:scale_translate}.
    We~shall identify measures over $\pgrass{\adim}{\adim-1}$ with elements of $P$ as
    in~\cite[section~5]{DePhilippis2018}.
    
    Set
    \begin{gather}
        \bar{A}(\nu) = \textint{}{} B_F(\perpproject{\lin\{v\}}) \ud \nu(v)
        \quad \text{for $\nu \in P$} \,;
        \\
        \text{thus,} \quad
        \bar{A}( \tfrac 12 (\Dirac{v} + \Dirac{-v}) ) = B_F(\perpproject{\lin\{v\}})
        \quad \text{for $v \in \sphere{\adim-1}$} \,.
    \end{gather}

    Define $G(\lambda v) = |\lambda|F ( \perpproject{\lin\{v\}} )$ whenever $v \in \sphere{\adim-1}$ and
    $\lambda \in \R$.  Write
    \begin{displaymath}
        \bar{A}(\nu) = \textint{\sphere{\adim-1}} {} \bigl(
        G(v) \id{\R^{\adim}} - \uD G(v) \cdot v
        \bigr) \ud \nu(v) \quad \text{for $\nu \in P$}\,,
    \end{displaymath}
    as in the proof of~\cite[Theorem~1.3]{DePhilippis2018}. By that theorem, $F \in \AC$
    implies that $G$ is strictly convex. Fix a vector
    $v_0 \in \sphere{\adim-1}$, set $T = \perpproject{\lin\{v_0\}}$, and define
    \begin{displaymath}
        \varphi_{v_0}(L) = \grad G(v_0) \bullet Lv_0
        \quad \text{for $L \in \End{\R^{\adim}}$} \,,
        \qquad
        N_F(T) = \uD G(v_0) \cdot v_0 \,;
    \end{displaymath}
    see~\ref{mr:rank_one_maps}. Since $\im N_F(T) = \lin\{v_0\} = \ker T$ and
    $\trace(L \circ N_F(T)) = \langle Lv_0 ,\, \uD G(v_0) \rangle$, we get
    \begin{displaymath}
        T \circ N_F(T) = 0
        \quad \text{and} \quad
        \varphi_{v_0}(L) = L \bullet \transpose{N_F(T)}
        \quad \text{for $L \in \End{\R^{\adim}}$} \,.
    \end{displaymath}
    We have
    \begin{displaymath}
        \varphi_{v_0}(\bar{A}(\nu))
        = \textint{\sphere{\adim-1}} {} \bigl(
        G(v_0)G(v) - ( \grad G(v_0) \bullet v ) ( \grad G(v) \bullet v_0 )
        \bigr) \ud \nu(v)
        \quad \text{for $\nu \in P$} \,.
    \end{displaymath}
    Since $|\grad G(v_0)\bullet v|\le G(v)$ and $|\grad G(v)\bullet v_0|\le G(v_0)$,
    with equality only for $v\in\R v_0$ by strict convexity, this is strictly positive unless
    $\spt \nu \subseteq \{v_0, -v_0\}$, in which case it is zero. Let
    $S \in \pgrass{\adim}{\adim-1}$, $S \ne T$, and $v \in \sphere{\adim-1}$ with
    $S = \perpproject{\lin\{v\}}$. Then $v \notin \R v_0$ and, since $B_F = F \transpose{P}_F$,
    \begin{displaymath}
        N_F(T) \bullet P_F(S)
        = \transpose{N_F(T)} \bullet B_F(S) / F(S)
        = \varphi_{v_0} \bigl( \bar{A}( \tfrac 12 (\Dirac{v} + \Dirac{-v}) ) \bigr) / F(S)
        > 0 \,.
    \end{displaymath}
    Thus $F \in \SEC{}$ by~\ref{def:SEC}.
\end{proof}

\begin{remark}
    The idea of the proof was taken from~\cite[section~5]{DePhilippis2018}.
\end{remark}

\begin{remark}
    Assume $F \in \AC{}$, $T \in \pgrass{\adim}{\vdim}$,
    $S = \project{(\im\transpose{P}_F(T))}$, and $C=\cone\transpose{\GF{F}}$.
    If the normals in~\ref{prop:existence_of_normals} can be chosen with
    \begin{equation}
        \label{eq:additional_condition_on_N}
        N_F(T)=M_F(S)=N \in \imsp(T^{\perp})\cap\transpose{\imsp(S^{\perp})} \,,
    \end{equation}
    then $\imsp(T)+\transpose{\imsp(S)}\subseteq H=\{A:A\bullet N=0\}$.
    If also $P_F(R)\bullet N>0$ for $R\ne T$, then $C\cap H=\ray\{P_F(T)\}$ and
    \begin{equation}
        \label{eq:more_restrictive_SAC}
        C\cap\bigl(\imsp(T)+\transpose{\imsp(S)}\bigr)=\ray\{P_F(T)\} \,.
    \end{equation}
    With only weak inequalities, the argument yields containment in the exposed face~$C\cap H$.
    After Euclidean normalisation, the normals of~\ref{def:QEC}
    have both properties for integrands of class~$\cnt{2}$ by~\ref{lem:QEC_compatibility}.
    For $F\in\SAC{}$ one may take a~positive multiple of $\transpose{Q}_F(T)$, since
    $Q_F(T)=S^{\perp}\circ Q_F(T)\circ T^{\perp}$.
\end{remark}

\begin{remark}
    \label{rem:not_polar}
    Assume $F\in\AC{}$ and let $N=N_F$ be as in~\ref{prop:existence_of_normals}. Then
    \begin{equation}
        \label{eq:duality_of_P_and_N}
        P_F(T)\bullet N(S)\ge0
        \quad\text{for $S,T\in\pgrass{\adim}{\vdim}$} \,.
    \end{equation}
    Thus, for $K=\cone\transpose{\GF{F}}$ and $L=\cone\im(-N)$,
    \begin{displaymath}
        K\subseteq\polar{L} \,,\qquad L\subseteq\polar{K} \,.
    \end{displaymath}
    The inequality in~\eqref{eq:duality_of_P_and_N} need not be strict for $S\ne T$,
    and these inclusions do not identify either cone with the polar of the other.
\end{remark}

\begin{remark}
    \label{rem:normals_compatible}
    Assume $F\in\AC{}$ is of class~$\cnt{2}$ away from the origin, and
    let $N_F$ be as in~\ref{prop:existence_of_normals}. Then $T \circ N_F(T) = 0$ and
    $P_F(S) \bullet N_F(T) \ge 0$ for $S,T \in
    \pgrass{\adim}{\vdim}$. In~\ref{lem:QEC_compatibility} we prove that
    $N_F(T) \circ \transpose{P}_F(T) = 0$ for $T \in \pgrass{\adim}{\vdim}$ as well.
\end{remark}

\begin{miniremark}
    \label{mr:J_fields}
    We~attach to each $J \in \End{\R^{\adim}}$ with $Jx \bullet x > 0$ for $x \ne 0$ a~field of
    oblique projections~$P^{J}$ over the Grassmannian. It admits normals $N(T) = T^{\perp} \circ J$
    with the properties of~$N_F(T)$ asserted in~\ref{prop:existence_of_normals}, with $\im P^{J}$ in
    place of $\transpose{\GF{F}} = \im P_F$. Nevertheless, $P^{J} = P_F$ for some integrand~$F$ if
    and only if $J$ is symmetric, and then $F$ is a~positive multiple of the area integrand of $[(x,y)\mapsto J^{-1}x\bullet y]$. Thus the existence of such normals does not characterise the fields of projections
    induced by integrands. The proof reduces the question to the exactness of a~$1$-form on the
    Grassmannian; see~\ref{ex:J_fields}\eqref{eq:Jf:criterion}.
\end{miniremark}

\begin{example}
    \label{ex:J_fields}
    Suppose
    \begin{gather}
        J \in \End{\R^{\adim}} \,,
        \quad
        J x \bullet x > 0 \quad \text{for $x \in \R^{\adim} \without \{0\}$} \,,
        \\
        \im P^{J}(T) = \im T \,,
        \quad
        \ker P^{J}(T) = \transpose{J} \lIm \im T^{\perp} \rIm \,,
        \quad
        P^{J}(T) \in \Proj{\adim}{\vdim} \,,
        \\
        N(T) = T^{\perp} \circ J
        \quad \text{for $T \in \pgrass{\adim}{\vdim}$} \,.
    \end{gather}
    Then $P^{J}$ is well defined and real analytic, $T^{\perp} \circ P^{J}(T) = 0$, and
    $N(T) = T^{\perp} \circ N(T)$ for $T \in \pgrass{\adim}{\vdim}$. Moreover:
    \begin{enumerate}
    \item
        \label{i:Jf:pairing}
        For $S,T \in \pgrass{\adim}{\vdim}$,
        \begin{displaymath}
            P^{J}(S) \bullet N(T) \ge 0 \,,
            \quad \text{with equality if and only if $S = T$} \,;
        \end{displaymath}
        consequently, for $T \in \pgrass{\adim}{\vdim}$,
        \begin{gather}
            \imsp(T) \subseteq \End{\R^{\adim}} \cap \{ A : A \bullet N(T) = 0 \} \,,
            \\
            \rInt \cone \im P^{J} \cap \{ A : A \bullet N(T) = 0 \} = \varnothing \,,
            \quad
            \cone \im P^{J} \subseteq \{ A : A \bullet N(T) \ge 0 \} \,.
        \end{gather}
    \item
        \label{i:Jf:dichotomy}
        The identity $P^{J} = P_F$ holds for some $\vdim$-integrand~$F$ if and only if $\transpose{J} = J$.
    \item
        \label{i:Jf:formula}
        If $\transpose{J} = J$, then $P^{J} = P_F$ if and only if there is $0 < c < \infty$ with
        \begin{displaymath}
            F(T) = c \, \bigl| \extpower{\vdim} J^{-1/2} \xi \bigr|
            \quad \text{whenever $\xi \in \ograss{\adim}{\vdim}$ and $\Sp \xi = \im T$} \,;
        \end{displaymath}
        equivalently,
        $\Phi_F(M) = c \, \HM^{\vdim} \bigl( J^{-1/2} \lIm M \rIm \bigr)$ for
        $(\HM^{\vdim},\vdim)$~rectifiable~$M$, i.e., $F$ is a~positive multiple of the area integrand of
        $[(x,y)\mapsto J^{-1}x\bullet y]$.
    \end{enumerate}
\end{example}

\begin{proof}
    Since $Jx \bullet x > 0$ for $x \ne 0$, the map $J$ is injective and hence invertible. If
    $T \in \pgrass{\adim}{\vdim}$, $z \in \im T^{\perp}$, and $\transpose{J} z \in \im T$, then
    \begin{gather}
        0 = z \bullet \transpose{J} z = J z \bullet z \,,
        \quad \text{whence} \quad
        z = 0 \,;
        \\
        \label{eq:Jex:decomposition}
        \text{thus,} \quad
        \im T \cap \transpose{J} \lIm \im T^{\perp} \rIm = \{0\}
        \quad \text{and} \quad
        \im T + \transpose{J} \lIm \im T^{\perp} \rIm = \R^{\adim} \,,
    \end{gather}
    which shows that $P^J$ is well defined. Set
    \begin{displaymath}
        E(T) = T + \transpose{J} \circ T^{\perp}
        \quad \text{for $T \in \pgrass{\adim}{\vdim}$} \,.
    \end{displaymath}
    Since $E(T)x = x$ for $x \in \im T$ and $E(T)x = \transpose{J}x$ for $x \in \im T^{\perp}$,
    \eqref{eq:Jex:decomposition} shows that $E(T)$ is invertible and
    $P^{J}(T) = E(T) \circ T \circ E(T)^{-1}$; hence, $P^{J}$ is real analytic.

    \emph{Proof of~\ref{i:Jf:pairing}.} Let $S \in \pgrass{\adim}{\vdim}$ and set
    \begin{displaymath}
        H = \tfrac 12 \bigl( J \circ \transpose{P^{J}(S)} + P^{J}(S) \circ \transpose{J} \bigr) \,.
    \end{displaymath}
    Let $x \in \R^{\adim}$, recall~\eqref{eq:Jex:decomposition}, and define $u \in \im S$ and
    $r \in \im S^{\perp}$ by requiring
    \begin{displaymath}
        \transpose{J}x = u + \transpose{J}r \,,
        \quad \text{i.e.,} \quad
        u = P^{J}(S) (\transpose{J} x) \,,
        \quad
        x = (\transpose{J})^{-1} u + r \,;
    \end{displaymath}
    thus, setting $b = (\transpose{J})^{-1} u$,
    \begin{displaymath}
        H x \bullet x = x \bullet (P^{J}(S) \circ \transpose{J}) x
        = \bigl( (\transpose{J})^{-1} u + r \bigr) \bullet u = J^{-1} u \bullet u
        = J b \bullet b  \ge 0 \,,
    \end{displaymath}
    with equality if and only if $b = 0$, i.e., $u = 0$, i.e., $x \in \im S^{\perp}$. Consequently,
    $H$ is symmetric and positive semi-definite,
    \begin{displaymath}
        \ker H = \im S^{\perp} \,,
        \quad \text{and} \quad
        \im H = \im S \,.
    \end{displaymath}
    Let $T \in \pgrass{\adim}{\vdim}$. As $T^{\perp}$ is symmetric and $H x \bullet x \ge 0$ for
    $x \in \R^{\adim}$,
    \begin{displaymath}
        P^{J}(S) \bullet N(T) = \trace ( T^{\perp} \circ H ) \ge 0 \,,
    \end{displaymath}
    with equality if and only if $T^{\perp} \circ H = 0$, i.e., $\im S = \im H \subseteq \im T$,
    i.e., $S = T$. By linearity, $R \bullet N(T) \ge 0$ for $R \in \cone \im P^{J}$. If
    $R \in \imsp(T)$, then $R = T \circ R$ and
    \begin{displaymath}
        R \bullet N(T) = \trace ( \transpose{J} \circ T^{\perp} \circ T \circ R ) = 0 \,.
    \end{displaymath}
    Suppose $R \in \rInt \cone \im P^{J}$, $R \bullet N(T) = 0$, and
    $S \in \pgrass{\adim}{\vdim}$. By~\cite[Theorem~6.4]{Rockafellar1970} there is
    $0 < t < \infty$ with $R + t(R - P^{J}(S)) \in \cone \im P^{J}$; hence,
    \begin{displaymath}
        0 \le \bigl( R + t(R - P^{J}(S)) \bigr) \bullet N(T) = - t \, P^{J}(S) \bullet N(T) \le 0 \,,
    \end{displaymath}
    so $P^{J}(S) \bullet N(T) = 0$ and $S = T$. As $S$ was arbitrary and $\pgrass{\adim}{\vdim}$
    has more than one element, this is impossible.

    For a~field $P : \pgrass{\adim}{\vdim} \to \Proj{\adim}{\vdim}$ with
    $T^{\perp} \circ P(T) = 0$ set
    \begin{displaymath}
        \tau_{P}(T) = P(T) + \transpose{P(T)} - 2T
        \quad \text{for $T \in \pgrass{\adim}{\vdim}$} \,.
    \end{displaymath}
    Since $T \circ P(T) = P(T)$ and $P(T) \circ T = T$, by~\ref{rem:sigma_props}\eqref{eq:tan_grass},
    \begin{displaymath}
        \tau_{P}(T) \in \Tan(\pgrass{\adim}{\vdim},T)
        \quad \text{and} \quad
        T \circ \tau_{P}(T) = P(T) - T \,;
    \end{displaymath}
    in~particular, $P$ is determined by~$\tau_{P}$. Let $F$ be a~$\vdim$-integrand of
    class~$\cnt{1}$, $F_0 = F | \pgrass{\adim}{\vdim}$, $T \in \pgrass{\adim}{\vdim}$, $D = \grad F(T)$, and
    $D_s = D + \transpose{D}$. From~\ref{def:aux-objects},
    \begin{gather}
        B_F(T) = F(T) T + T^{\perp} \circ D_s \circ T \,,
        \quad
        P_F(T) = T + T \circ D_s \circ T^{\perp} / F(T) \,,
        \\
        \tau_{P_F}(T)
        = \bigl( T \circ D_s \circ T^{\perp} + T^{\perp} \circ D_s \circ T \bigr) / F(T)
        = \grad ( 2 \log F_0 )(T) \,,
    \end{gather}
    where the last equality holds because
    $\tfrac 12 ( T \circ D_s \circ T^{\perp} + T^{\perp} \circ D_s \circ T )$ is the orthogonal
    projection of~$D$ onto $\Tan(\pgrass{\adim}{\vdim},T)$. For a~field~$P$ as above of
    class~$\cnt{1}$ define the $1$-form
    \begin{displaymath}
        \omega_P(T)Y = \tau_{P}(T) \bullet Y
        \quad \text{for $T \in \pgrass{\adim}{\vdim}$ and $Y \in \Tan(\pgrass{\adim}{\vdim},T)$} \,.
    \end{displaymath}
    If $\omega_{P} = \ud \varphi$ for some $0$-form $\varphi : \pgrass{\adim}{\vdim} \to \R$ of
    class~$\cnt{1}$, then $F_0 = e^{\varphi/2}$, extended $\vdim$-homogeneously
    by~\ref{rem:integrand_defined_by_values_on_Gnk}, satisfies $\tau_{P_F} = \tau_{P}$; hence,
    $P_F = P$. Since $\pgrass{\adim}{\vdim}$ is connected, we~conclude
    \begin{equation}
        \label{eq:Jf:criterion}
        \text{$P = P_F$ for some $\vdim$-integrand~$F$}
        \quad \iff \quad
        \text{$\omega_{P} = \ud \varphi$ for some~$\varphi$} \,,
    \end{equation}
    and then $F_0 = c\, e^{\varphi/2}$ for some $0 < c < \infty$.

    For $T \in \pgrass{\adim}{\vdim}$, $v \in \im T$, and
    $w \in \im T^{\perp}$ set
    \begin{displaymath}
        Y_{v,w}(x) = (w \bullet x)v + (v \bullet x)w
        \quad \text{for $x \in \R^{\adim}$} \,;
    \end{displaymath}
    by~\ref{rem:sigma_props}\eqref{eq:tan_grass} these maps span $\Tan(\pgrass{\adim}{\vdim},T)$. If $P$ is a~field
    as above and $\delta(T) = P(T) - T$, then
    \begin{displaymath}
        \delta(T) = T \circ \delta(T) \circ T^{\perp} \,,
        \quad
        \tau_{P}(T) = \delta(T) + \transpose{\delta(T)} \,,
        \quad
        \delta(T) v = 0 \,;
    \end{displaymath}
    hence, as $Y_{v,w}$ is symmetric,
    \begin{equation}
        \label{eq:Jf:pairing_formula}
        \tau_{P}(T) \bullet Y_{v,w}
        = 2\, \delta(T) \bullet Y_{v,w}
        = 2\, v \bullet \delta(T) w
        = 2\, v \bullet P(T) w \,.
    \end{equation}
    For $P = P^{J}$, write $w = P^{J}(T) w + \transpose{J} z$ with $z \in \im T^{\perp}$, so that
    $(T^{\perp} \circ \transpose{J}) z = w$. Since $v \bullet w = 0$,
    \begin{equation}
        \label{eq:Jf:tau}
        \tau_{P^{J}}(T) \bullet Y_{v,w} = 2\, v \bullet (w - \transpose{J} z) = - 2\, J v \bullet z
        \,.
    \end{equation}

    Assume $\transpose{J} = J$. Then $J^{-1/2}$ is symmetric and invertible; hence,
    $\extpower{\vdim} J^{-1/2}$ is symmetric by~\cite[1.7.6]{Federer1969} and invertible with
    inverse $\extpower{\vdim} J^{1/2}$. Since
    $\extpower{\vdim} J^{-1} = \extpower{\vdim} J^{-1/2} \circ \extpower{\vdim} J^{-1/2}$, we may
    define
    \begin{multline}
        \psi_J(T) = \log \bigl( \extpower{\vdim} J^{-1} \xi \bullet \xi \bigr)
        = 2 \log \bigl| \extpower{\vdim} J^{-1/2} \xi \bigr|
        \\
        \text{whenever $T \in \pgrass{\adim}{\vdim}$, $\xi \in \ograss{\adim}{\vdim}$, and
          $\Sp \xi = \im T$} \,;
    \end{multline}
    the value does not depend on the choice of~$\xi$, which is unique up to sign. Let
    $T \in \pgrass{\adim}{\vdim}$, $v \in \im T$, $w \in \im T^{\perp}$, let $e_1, \ldots,
    e_{\vdim}$ be an orthonormal basis of~$\im T$, $\xi = e_1 \wedge \cdots \wedge e_{\vdim}$, and
    \begin{gather}
        X(x) = (v \bullet x) w - (w \bullet x) v
        \quad \text{for $x \in \R^{\adim}$} \,,
        \\
        \gamma(t) = \exp(tX) \circ T \circ \exp(-tX) \,,
        \quad
        \xi(t) = \extpower{\vdim} \exp(tX) \, \xi
        \quad \text{for $t \in \R$} \,.
    \end{gather}
    Since $\transpose{X} = -X$, the map $\exp(tX)$ is orthogonal; hence, $\gamma(t) \in
    \pgrass{\adim}{\vdim}$, $\xi(t) \in \ograss{\adim}{\vdim}$, and $\Sp \xi(t) = \im \gamma(t)$ for
    $t \in \R$. Moreover, $\gamma(0) = T$, $\gamma'(0) = X \circ T - T \circ X = Y_{v,w}$, and $X e_i
    = (v \bullet e_i) w$ for $i \in \{1, \ldots, \vdim\}$. Write $w = P^{J}(T) w + J z$ with $z \in
    \im T^{\perp}$. Since $P^{J}(T) w = \textsum{j=1}{\vdim} (P^{J}(T) w \bullet e_j) e_j$,
    \begin{multline}
        \xi'(0)
        = \textsum{i=1}{\vdim} (v \bullet e_i) \,
        e_1 \wedge \cdots \wedge e_{i-1} \wedge w \wedge e_{i+1} \wedge \cdots \wedge e_{\vdim}
        \\
        = (v \bullet P^{J}(T) w) \xi + \textsum{i=1}{\vdim} (v \bullet e_i) \,
        e_1 \wedge \cdots \wedge e_{i-1} \wedge J z \wedge e_{i+1} \wedge \cdots \wedge e_{\vdim} \,.
    \end{multline}
    If $y_1, \ldots, y_{\vdim} \in \R^{\adim}$ and $y_i \in \im T^{\perp}$ for some~$i$, then, as
    $\extpower{\vdim} T$ is symmetric and $\extpower{\vdim} T \xi = \xi$,
    \begin{displaymath}
        \xi \bullet (y_1 \wedge \cdots \wedge y_{\vdim})
        = \xi \bullet (T y_1 \wedge \cdots \wedge T y_{\vdim}) = 0 \,.
    \end{displaymath}
    Applying this with $y_j = J^{-1} e_j$ for $j \ne i$ and $y_i = z$, and recalling that
    $\extpower{\vdim} J^{-1}$ is symmetric, we obtain
    \begin{displaymath}
        \extpower{\vdim} J^{-1} \xi \bullet \xi'(0)
        = (v \bullet P^{J}(T) w) \bigl( \extpower{\vdim} J^{-1} \xi \bullet \xi \bigr) \,;
    \end{displaymath}
    hence, by~\eqref{eq:Jf:pairing_formula},
    \begin{displaymath}
        \uD \psi_J(T) Y_{v,w}
        = \left. \tfrac{\ud}{\ud t} \right|_{t=0} \psi_J(\gamma(t))
        = \frac{2\, \extpower{\vdim} J^{-1} \xi \bullet \xi'(0)}
        {\extpower{\vdim} J^{-1} \xi \bullet \xi}
        = 2\, v \bullet P^{J}(T) w
        = \tau_{P^{J}}(T) \bullet Y_{v,w} \,.
    \end{displaymath}
    As the maps $Y_{v,w}$ span $\Tan(\pgrass{\adim}{\vdim},T)$, we get
    $\omega_{P^{J}} = \ud \psi_J$, and~\eqref{eq:Jf:criterion} gives $P^{J} = P_F$ exactly for
    $F_0 = c\, e^{\psi_J/2}$ with $0 < c < \infty$, i.e., for
    $F(T) = c \, \bigl| \extpower{\vdim} J^{-1/2} \xi \bigr|$ whenever $\xi \in
    \ograss{\adim}{\vdim}$ and $\Sp \xi = \im T$. This is~\ref{i:Jf:formula}; the description
    of~$\Phi_F$ follows from the area formula. This also proves the ``if'' part
    of~\ref{i:Jf:dichotomy}.

    Let $A \in \End{\R^{\adim}}$ be invertible and set
    \begin{gather}
        \Phi_A(T) = \project{ ( A \lIm \im T \rIm ) } \,,
        \quad
        P'(T) = A \circ P( \Phi_{A^{-1}}(T) ) \circ A^{-1} \,,
        \quad
        \mathrm{O}(T) = T
        \quad \text{for $T \in \pgrass{\adim}{\vdim}$} \,,
    \end{gather}
    where $P$ is a~field as above. Then $P'$ is again such a~field. Since
    $( A^{-1} \lIm \im T \rIm )^{\perp} = \transpose{A} \lIm \im T^{\perp} \rIm$,
    \begin{gather}
        \ker (P^{J})'(T)
        = A \lIm \transpose{J} \lIm \transpose{A} \lIm \im T^{\perp} \rIm \rIm \rIm
        = \transpose{ ( A \circ J \circ \transpose{A} ) } \lIm \im T^{\perp} \rIm \,,
        \\
        \text{i.e.,} \quad
        ( P^{J} )' = P^{A \circ J \circ \transpose{A}} \,,
        \quad
        \mathrm{O}' = P^{A \circ \transpose{A}} \,,
    \end{gather}
    and $(A \circ J \circ \transpose{A}) x \bullet x = J (\transpose{A} x) \bullet \transpose{A} x
    > 0$ for $x \ne 0$. Let $P$, $Q$ be two fields as above,
    $\hat{T} = \Phi_{A^{-1}}(T)$, $v \in \im T$, $w \in \im T^{\perp}$, and
    \begin{displaymath}
        \hat{v} = (\hat{T} \circ \transpose{A}) v \,,
        \quad
        \hat{w} = (\hat{T}^{\perp} \circ A^{-1}) w \,.
    \end{displaymath}
    Since $(P - Q)(\hat{T}) = \hat{T} \circ (P - Q)(\hat{T}) \circ \hat{T}^{\perp}$,
    \eqref{eq:Jf:pairing_formula} gives
    \begin{equation}
        \label{eq:Jf:cocycle}
        \bigl( \tau_{P'} - \tau_{Q'} \bigr)(T) \bullet Y_{v,w}
        = 2\, v \bullet (A \circ (P - Q)(\hat{T}) \circ A^{-1}) w
        = 2\, \hat{v} \bullet (P - Q)(\hat{T}) \hat{w}
        = \bigl( \tau_{P} - \tau_{Q} \bigr)(\hat{T}) \bullet Y_{\hat{v},\hat{w}} \,.
    \end{equation}
    Let $X$ and~$\gamma$ be as in the symmetric case, so that $\gamma(0) = T$ and $\gamma'(0) =
    Y_{v,w}$, and set
    \begin{displaymath}
        L(t) = A^{-1} \circ \exp(tX) \circ A \,,
        \quad
        \hat{\gamma}(t) = \Phi_{A^{-1}}(\gamma(t))
        \quad \text{for $t \in \R$} \,.
    \end{displaymath}
    Then $\im \hat{\gamma}(t) = \im (L(t) \circ \hat{T})$; hence, $\hat{\gamma}(t) \circ L(t) \circ
    \hat{T} = L(t) \circ \hat{T}$ for $t \in \R$. Differentiating at $t = 0$ and using $L(0) =
    \id{\R^{\adim}}$ and $\hat{\gamma}(0) = \hat{T}$, we obtain $\hat{\gamma}'(0) \circ \hat{T} =
    \hat{T}^{\perp} \circ L'(0) \circ \hat{T}$. Since $\hat{\gamma}'(0) \in
    \Tan(\pgrass{\adim}{\vdim},\hat{T})$, \ref{rem:sigma_props}\eqref{eq:tan_grass} yields
    \begin{displaymath}
        \hat{\gamma}'(0)
        = \hat{T}^{\perp} \circ L'(0) \circ \hat{T}
        + \transpose{\bigl( \hat{T}^{\perp} \circ L'(0) \circ \hat{T} \bigr)} \,.
    \end{displaymath}
    As $L'(0) = A^{-1} \circ X \circ A$, $A \lIm \im \hat{T} \rIm = \im T$, and $w \bullet y = 0$
    for $y \in \im T$, we have
    \begin{displaymath}
        (\hat{T}^{\perp} \circ L'(0) \circ \hat{T}) x
        = (v \bullet A \hat{T} x) \, \hat{T}^{\perp} A^{-1} w
        = (\hat{v} \bullet x) \hat{w}
        \quad \text{for $x \in \R^{\adim}$} \,;
    \end{displaymath}
    hence, $\uD \Phi_{A^{-1}}(T) Y_{v,w} = \hat{\gamma}'(0) = Y_{\hat{v},\hat{w}}$.
    Taking $Q = \mathrm{O}$, for which $\tau_{\mathrm{O}} = 0$, \eqref{eq:Jf:cocycle} becomes
    \begin{displaymath}
        \omega_{P'}
        = \Phi_{A^{-1}}^{*} \, \omega_{P} + \omega_{P^{A \circ \transpose{A}}} \,.
    \end{displaymath}
    The second summand is exact by the symmetric case, and pull-back along a~diffeomorphism
    preserves exactness; hence, $\omega_{P}$ is exact if and only if $\omega_{P'}$ is.

    Assume $\transpose{J} \ne J$ and, for a~contradiction, that
    $\omega_{P^{J}}$ is exact. Applying the change of coordinates with
    $A = ( \tfrac 12 (J + \transpose{J}) )^{-1/2}$ replaces $J$ by $A \circ J \circ \transpose{A}$,
    whose symmetric part is $\id{\R^{\adim}}$; we may therefore assume
    \begin{displaymath}
        J = \id{\R^{\adim}} + K \,,
        \quad
        \transpose{K} = -K \ne 0 \,.
    \end{displaymath}
    Choose $0 < \kappa < \infty$, orthonormal $a_1, a_2$, and a~$(\vdim-1)$-dimensional linear
    subspace $V' \subseteq ( \lin \{a_1,a_2\} )^{\perp}$ -- possible since
    $\adim - 2 \ge \vdim - 1$ -- such that
    \begin{gather}
        K a_1 = \kappa a_2 \,,
        \quad
        K a_2 = -\kappa a_1 \,,
        \\
        v_{\theta} = \cos \theta \, a_1 + \sin \theta \, a_2 \,,
        \quad
        w_{\theta} = - \sin \theta \, a_1 + \cos \theta \, a_2 \,,
        \quad
        T_{\theta} = \project{ ( \lin \{ v_{\theta} \} + V' ) }
        \quad \text{for $\theta \in \R$} \,.
    \end{gather}
    Then $[ \theta \mapsto T_{\theta} ]$ is a~smooth loop with $T_{\pi} = T_{0}$,
    \begin{displaymath}
        v_{\theta} \in \im T_{\theta} \,,
        \quad
        w_{\theta} \in \im T_{\theta}^{\perp} \,,
        \quad
        \tfrac{\ud}{\ud \theta} T_{\theta} = Y_{v_{\theta},w_{\theta}} \,,
        \quad
        K v_{\theta} = \kappa w_{\theta} \,.
    \end{displaymath}
    Let $z_{\theta} \in \im T_{\theta}^{\perp}$ be as in~\eqref{eq:Jf:tau}. Then $z_{\theta} \ne 0$,
    because $w_{\theta} \notin \im T_{\theta}$, and
    \begin{gather}
        v_{\theta} \bullet z_{\theta} = 0 \,,
        \quad
        w_{\theta} \bullet z_{\theta} = \transpose{J} z_{\theta} \bullet z_{\theta}
        = |z_{\theta}|^{2} \,,
        \\
        \tau_{P^{J}}(T_{\theta}) \bullet Y_{v_{\theta},w_{\theta}}
        = - 2 \, ( v_{\theta} + K v_{\theta} ) \bullet z_{\theta}
        = - 2 \kappa \, |z_{\theta}|^{2} < 0 \,.
    \end{gather}
    Hence, the integral of $\omega_{P^{J}}$ along this loop is negative, which is impossible for
    an~exact form. By~\eqref{eq:Jf:criterion} this proves the ``only if'' part
    of~\ref{i:Jf:dichotomy}.
\end{proof}

\section{Polyconvex integrands}
\label{sec:polyconvex_integrands}

\begin{miniremark}
    In this section we study integrands induced by norms on $\extpower{\vdim} \R^{\adim}$;
    see~\ref{def:polyconvex}. For such an integrand we express~$B_F$ through the derivative
    $\varkappa$ of the exterior power at the identity; see~\ref{def:varkappa}
    and~\ref{lem:BF_via_varkappa}. This formula underlies the computations of
    section~\ref{sec:AC1_failure}. We~also characterise~\SEC{} for polyconvex integrands in terms
    of accretivity of the operators $\varkappa(\transpose{N(T)})$; see~\ref{def:accretive}
    and~\ref{lem:SEC_accretive}. At~the end of the section we~recall Busemann's notion of
    a~\emph{totally convex} norm and discuss its relation to~\SAC{};
    see~\ref{def:totally_convex_norm}, \ref{rem:total_convexity_BS},
    and~\ref{rem:totally_convex_norms_and_AC}. Whether total convexity implies~\AC{} or~\EC{} is
    asked in~\ref{q:total_convexity_SAC}.
\end{miniremark}

\begin{definition}
    \label{def:polyconvex}
    We say that an integrand $F$ is \emph{polyconvex} if there exists a norm $\varphi :
    \extpower{\vdim} \R^{\adim} \to \R$ such that
    \begin{displaymath}
        F(T) = \varphi(\xi)
        \quad \text{whenever $T \in \pgrass{\adim}{\vdim}$, $\xi \in \ograss{\adim}{\vdim}$, and $\Sp \xi = \im T$} \,.
    \end{displaymath}
    In this case we say that \emph{$F$ is associated to $\varphi$}. If $\varphi$ is strictly convex,
    then we say that $F$ is \emph{strictly polyconvex}.
\end{definition}

\begin{remark}
    Integrands of this type are called \emph{extendably convex} by Busemann
    in~\cite[p.~37]{Busemann1960} and also \emph{convex on the Grassmann cone} in his later works.
    We adopt here a~convention inspired by~\cite[Definition~1.3]{DeRosa2025} (see
    also~\cite[section~3]{Lesniak2025}) since the name ``polyconvex'' seems more deeply rooted in
    the calculus of variations.
\end{remark}

\begin{miniremark}
    The following result of De~Rosa, Lei, and Young motivates the study of strictly
    polyconvex integrands. Strict convexity is assumed only where it is explicitly stated.
\end{miniremark}

\begin{theorem}[\protect{see~\cite[Theorem~1.16]{DeRosa2025}}]
    Assume $F \in \AC{}$. Then $F$ is strictly polyconvex.
\end{theorem}

\begin{definition}
    \label{def:varkappa}
    Define the linear map $\varkappa : \End{\R^{\adim}} \to \End{\extpower{\vdim} \R^{\adim}}$ so that
    \begin{multline}
        \varkappa(L)\xi
        = \uD \bigl[
        \End{\R^{\adim}} \ni A \mapsto \langle \xi ,\, \extpower{\vdim}A \rangle
        \bigr](\id{\R^{\adim}})L
        \\
        = \textsum{i=1}{\vdim} u_{1} \wedge \cdots \wedge L u_i \wedge \cdots \wedge u_{\vdim}
        \quad \text{whenever $\xi = u_{1} \wedge \cdots \wedge u_{\vdim} \in \extpower{\vdim} \R^{\adim}$} \,.
    \end{multline}
\end{definition}

\begin{remark}
    \label{rem:varkappa_derivation}
    For $j \in \{0,1,\ldots,\adim\}$ let $\varkappa_{j} : \End{\R^{\adim}} \to
    \End{\extpower{j} \R^{\adim}}$ be defined by the formula of~\ref{def:varkappa} with~$j$ in place
    of~$\vdim$; thus $\varkappa_{\vdim} = \varkappa$, $\varkappa_{1}(L) = L$, and $\varkappa_{0} =
    0$. Whenever $L \in \End{\R^{\adim}}$, $j,l \in \{0,1,\ldots,\adim\}$, $j + l \le \adim$, $\zeta
    \in \extpower{j} \R^{\adim}$, and $\chi \in \extpower{l} \R^{\adim}$, we have
    \begin{displaymath}
        \varkappa_{j+l}(L) ( \zeta \wedge \chi )
        = \bigl( \varkappa_{j}(L) \zeta \bigr) \wedge \chi
        + \zeta \wedge \bigl( \varkappa_{l}(L) \chi \bigr) \,,
    \end{displaymath}
    i.e., $\varkappa(L)$ acts as a~derivation on wedge products. Indeed, both sides are bilinear
    in~$(\zeta,\chi)$, so it suffices to verify the~formula for $\zeta = u_1 \wedge \cdots \wedge
    u_{j}$ and $\chi = u_{j+1} \wedge \cdots \wedge u_{j+l}$, in which case both sides equal
    $\textsum{t=1}{j+l} u_1 \wedge \cdots \wedge L u_{t} \wedge \cdots \wedge u_{j+l}$
    by~\ref{def:varkappa}. Whenever no confusion may arise, we drop the~subscript and write
    $\varkappa$ in place of~$\varkappa_{j}$.
\end{remark}

\begin{lemma}
    \label{lem:BF_via_varkappa}
    Let $F$ be a~polyconvex $\vdim$-integrand associated to a~norm
    $\varphi : \extpower{\vdim} \R^{\adim} \to \R$ of class~$\cnt{1}$ away from the origin. Then
    \begin{displaymath}
        B_F(T) \bullet L = \grad \varphi(\xi) \bullet \varkappa(L)\xi
    \end{displaymath}
    whenever $L \in \End{\R^{\adim}}$, $T \in \pgrass{\adim}{\vdim}$, $\xi \in
    \ograss{\adim}{\vdim}$, and $\Sp \xi = \im T$.
\end{lemma}

\begin{proof}
    Set $L(t) = \id{\R^{\adim}} + tL$ for $t \in \R$. Write $\xi = u_1 \wedge \cdots \wedge
    u_{\vdim}$ with $u_1,\ldots,u_{\vdim}$ an orthonormal basis of~$\im T$, and set
    $\zeta(t) = \extpower{\vdim}L(t)\, \xi = L(t)u_1 \wedge \cdots \wedge L(t)u_{\vdim}$. For small
    $|t|$, the map $L(t)$ is invertible; hence, $\zeta(t)$ is a~non-zero simple $\vdim$-vector and
    $\Sp \zeta(t) = \im (L(t) \circ T)$. Since $\extpower{\vdim}T$ is the orthogonal projection
    onto the line spanned by~$\xi$, the map
    $\extpower{\vdim}(L(t) \circ T) = \extpower{\vdim} L(t) \circ \extpower{\vdim} T$ sends
    $\chi \in \extpower{\vdim} \R^{\adim}$ to $(\chi \bullet \xi)\, \zeta(t)$; thus,
    $\| \extpower{\vdim} (L(t) \circ T) \| = |\zeta(t)|$. Employing~\ref{def:polyconvex} and
    the homogeneity of~$\varphi$,
    \begin{displaymath}
        F \bigl( \project{ \im (L(t) \circ T) } \bigr) \,
        \| \extpower{\vdim} (L(t) \circ T) \|
        = \varphi \bigl( \zeta(t) / |\zeta(t)| \bigr)\, |\zeta(t)|
        = \varphi(\zeta(t)) \,.
    \end{displaymath}
    Recalling~\ref{rem:BF_variational} and~\ref{def:varkappa}, and noting
    $\zeta'(0) = \varkappa(L)\xi$, we conclude
    \begin{displaymath}
        B_F(T) \bullet L
        = \left. \tfrac{\ud}{\ud t} \right|_{t=0} \varphi(\zeta(t))
        = \grad\varphi(\xi) \bullet \varkappa(L)\xi \,.
        \qedhere
    \end{displaymath}
\end{proof}

\begin{definition}[\protect{cf.~\cite[Definition~1.4.1]{Pazy1983} and~\cite[Chap.~VI, 1.1 and Example~2 after~1.1]{Cioranescu1990}}]
    \label{def:accretive}
    Let $X$ be a~Euclidean space, $\nu : X \to \R$ be a~norm, $A \in \End{X}$, $C \subseteq X$ be
    a~cone, and $Z \subseteq X$. We say that $A$ is \emph{$\nu$-accretive} if
    \begin{displaymath}
        \forall x \in X \  \exists v \in \subgrad \nu(x) \quad v \bullet Ax \ge 0 \,,
    \end{displaymath}
    and that $A$ is \emph{strictly $\nu$-accretive on~$C$ off~$Z$} if
    \begin{displaymath}
        v \bullet Ax > 0
        \quad \text{whenever $x \in C \without Z$ and $v \in \subgrad \nu(x)$} \,.
    \end{displaymath}
    If $A$ is strictly $\nu$-accretive on~$X$ off~$\ker A$, we simply say that $A$ is
    \emph{strictly $\nu$-accretive}. In~case $\nu$ is of class~$\cnt{1}$ away from the origin,
    $\subgrad \nu(x) = \{ \grad \nu(x) \}$ for $x \ne 0$.
\end{definition}

\begin{remark}
    \label{rem:accretive_duality_map}
    In~\cite[Definition~1.4.1]{Pazy1983} and~\cite[Chap.~VI, 1.1]{Cioranescu1990}, accretivity is formulated by means of the normalised duality
    mapping $J(x) = \{ v : v \bullet x = \nu(x)^2 = \polar{\nu}(v)^2 \}$ in place
    of~$\subgrad \nu(x)$. By Asplund's theorem~\cite[Chap.~I, 4.4]{Cioranescu1990},
    $J(x) = \subgrad \bigl( \tfrac 12 \nu^2 \bigr)(x) = \nu(x)\, \subgrad \nu(x)$ for $x \in X$;
    hence, for $x \ne 0$ the conditions $v \bullet Ax \ge 0$ for some $v \in J(x)$ and for some
    $v \in \subgrad \nu(x)$ are equivalent, and the notion of~$\nu$-accretivity
    in~\ref{def:accretive} coincides with that of~\cite[Chap.~VI, 1.1(a) and
      Example~2 after~1.1]{Cioranescu1990}. The notion of strict accretivity on a~cone off a~set is ours.
\end{remark}

\begin{lemma}
    \label{lem:SEC_accretive}
    Let $F$ be a~polyconvex $\vdim$-integrand associated to a~norm
    $\varphi : \extpower{\vdim} \R^{\adim} \to \R$ of class~$\cnt{1}$ away from the origin. Then
    $F \in \SEC{}$ if and only if there exists $N : \pgrass{\adim}{\vdim} \to \End{\R^{\adim}}$
    such that for $T \in \pgrass{\adim}{\vdim}$
    \begin{displaymath}
        T \circ N(T) = 0
        \quad \text{and} \quad
        \text{$\varkappa(\transpose{N(T)})$ is strictly $\varphi$-accretive on
          $\ray \ograss{\adim}{\vdim}$ off $\extpower{\vdim} \im T$} \,.
    \end{displaymath}
\end{lemma}

\begin{proof}
    Let $N : \pgrass{\adim}{\vdim} \to \End{\R^{\adim}}$ satisfy $T \circ N(T) = 0$ for
    $T \in \pgrass{\adim}{\vdim}$. Let $S,T \in \pgrass{\adim}{\vdim}$ and
    $\xi \in \ograss{\adim}{\vdim}$ with $\Sp \xi = \im S$. Since
    $P_F(S) = \transpose{B_F(S)} / F(S)$ and $F(S) = \varphi(\xi) > 0$,
    \ref{lem:BF_via_varkappa} yields
    \begin{displaymath}
        \varphi(\xi)\, N(T) \bullet P_F(S)
        = B_F(S) \bullet \transpose{N(T)}
        = \grad \varphi(\xi) \bullet \varkappa(\transpose{N(T)}) \xi \,.
    \end{displaymath}
    Every non-zero $x \in \ray \ograss{\adim}{\vdim}$ is of the form $x = t\xi$ with $0 < t < \infty$ and
    $\xi \in \ograss{\adim}{\vdim}$; moreover, $x \in \extpower{\vdim} \im T$ if and only if
    $\Sp \xi = \im T$, and $\grad \varphi(t\xi) \bullet \varkappa(\transpose{N(T)}) t\xi = t\,
    \grad \varphi(\xi) \bullet \varkappa(\transpose{N(T)}) \xi$. Therefore $N(T) \bullet P_F(S) > 0$
    for all $S \in \pgrass{\adim}{\vdim}$ with $S \ne T$ if and only if
    $\varkappa(\transpose{N(T)})$ is strictly $\varphi$-accretive on $\ray \ograss{\adim}{\vdim}$
    off $\extpower{\vdim} \im T$. Recalling~\ref{def:SEC} finishes the proof.
\end{proof}

\begin{remark}
    \label{rem:accretivity_origin}
    Accretive operators originate in the theory of semigroups of operators and evolution
    equations: $A$ is $\nu$-accretive if and only if $-A$ is \emph{dissipative} in the sense of
    Lumer and Phillips; cf.~\cite[Definition~1.4.1]{Pazy1983} and~\cite[Chap.~VI, Example~2
      after~1.1]{Cioranescu1990}. By~\cite[Theorem~1.4.2]{Pazy1983} applied to~$-A$
    (see also~\cite[Chap.~VI, 1.3]{Cioranescu1990}), $A$ is $\nu$-accretive if and
    only if
    \begin{displaymath}
        \nu(tx + Ax) \ge t \nu(x) \quad \text{for $x \in X$ and $0 < t < \infty$} \,,
    \end{displaymath}
    and, by the Lumer--Phillips theorem~\cite[Theorem~1.4.3]{Pazy1983} -- whose range condition
    is automatic here, since $\lambda \id{X} + A$ is invertible for $\lambda > \|A\|$ -- this holds
    if and only if $\nu(\exp(-tA)x) \le \nu(x)$ for $x \in X$ and $0 \le t < \infty$, i.e., $-A$
    generates a~semigroup of $\nu$-contractions. In~our setting, since $\extpower{\vdim}$ is
    a~group homomorphism and $\varkappa$ is its derivative at the identity,
    $\extpower{\vdim} \exp(tL) = \exp(t \varkappa(L))$ for $L \in \End{\R^{\adim}}$ and $t \in \R$;
    hence, $\varkappa(\transpose{N(T)})$ is $\varphi$-accretive if and only if the maps
    $\extpower{\vdim} \exp(-t \transpose{N(T)})$, $0 \le t < \infty$, are $\varphi$-non-increasing.
    The~condition in~\ref{lem:SEC_accretive} is weaker: only simple $\vdim$-vectors are tested.
    The name, and the systematic use of accretivity in the ellipticity theory of geometric
    functionals, appear to be new; the pairing itself occurs implicitly in the concave barriers of
    Firester and Tsiamis~\cite[section~2.1]{Firester2026}, as we explain in~\ref{rem:concave_barriers}.
\end{remark}

\begin{remark}
    \label{rem:concave_barriers}
    Let $X$ be a~Euclidean space of dimension~$\vdim$, let $\nu : X \to \R$ be a~norm of
    class~$\cnt{1}$ away from the origin, and let $\beta_X : X \to \adjoint{X}$ be the polarity
    of~$X$; cf.~\cite[1.7.1]{Federer1969}. Recalling~\ref{mr:rank_one_maps}, we have
    $A \bullet \bigl( \beta_X(x) \cdot v \bigr) = v \bullet Ax$ for $A \in \End{X}$ and
    $x,v \in X$. Setting
    \begin{displaymath}
        \phi_{\nu}(A)
        = \inf \bigl\{ \grad \nu(x) \bullet Ax : x \in X ,\, \nu(x) = 1 \bigr\}
        = \inf \bigl\{ A \bullet \bigl( \beta_X(x) \cdot \grad \nu(x) \bigr) : \nu(x) = 1 \bigr\}
        \,,
    \end{displaymath}
    one obtains, by~\ref{def:accretive} and compactness of $\{ x : \nu(x) = 1 \}$,
    \begin{gather*}
        \text{$A$ is $\nu$-accretive}
        \quad \text{if and only if} \quad
        \phi_{\nu}(A) \ge 0 \,, \\
        \text{$A$ is strictly $\nu$-accretive and injective}
        \quad \text{if and only if} \quad
        \phi_{\nu}(A) > 0 \,.
    \end{gather*}
    Being an~infimum of linear functionals of~$A$, the function $\phi_{\nu}$ is concave and
    positively $1$-homogeneous; moreover, if $\det A = 0$, $x \in \ker A$, and $\nu(x) = 1$, then
    $\grad \nu(x) \bullet Ax = 0$, whence $\phi_{\nu}(A) \le 0$. Therefore $\phi_{\nu}$ is
    a~\emph{concave barrier} in the sense of~\cite[section~2.1]{Firester2026}, i.e., it is concave,
    positively $1$-homogeneous, and non-positive on the wave cone
    $\End{X} \cap \{ A : \det A = 0 \}$ of the row-wise divergence operator; for barriers of this
    form the last condition is automatic.

    Firester and Tsiamis prove the anisotropic Michael--Simon inequality in every dimension and
    codimension under the following hypothesis; see~\cite[Theorem~4.1]{Firester2026}. Let $F$ be
    a~$\vdim$-integrand, let $\phi$ be a~concave barrier on $\End{\R^{\vdim}}$, let $\rho$ be
    a~Radon probability measure on $\pgrass{\adim}{\vdim}$, and let
    $E_P : \R^{\vdim} \to \R^{\adim}$ for $P \in \pgrass{\adim}{\vdim}$ be a~Borel family of linear
    isometries with $\im E_P = \im P$. Set
    \begin{displaymath}
        \kappa(F,\phi,\rho) = \inf \Bigl\{ \textint{\pgrass{\adim}{\vdim}}{} \phi \bigl(
        \transpose{E}_P \circ B_F(T) \circ E_P \bigr) \ud \rho(P) : T \in \pgrass{\adim}{\vdim}
        \Bigr\} \,.
    \end{displaymath}
    The hypothesis is that $\kappa(F,\phi,\rho) > 0$ for some such $\phi$ and~$\rho$. Each of the
    three barriers they exhibit is of the form~$\phi_{\nu}$:
    \begin{gather*}
        \text{$\nu$ the Euclidean norm, giving $\phi_{\nu}(A) = \inf \{ Av \bullet v : |v| = 1 \}$;
          cf.~\cite[Lemma~2.1]{Firester2026}} \,, \\
        \text{$\nu$ the $\ell^p$~norm, $1 < p < \infty$; cf.~the proof
          of~\cite[Theorem~1.3]{Firester2026}} \,, \\
        \text{$\nu$ the one-homogeneous extension of a~convex even hypersurface integrand;} \\
        \text{cf.~the proof of~\cite[Theorem~1.2]{Firester2026}} \,.
    \end{gather*}
    In the second case the pairing is written by means of the map
    $J_p(x) = (|x_1|^{p-2}x_1, \dotsc, |x_{\vdim}|^{p-2}x_{\vdim})$, which agrees with the normalised duality
    mapping of~\ref{rem:accretive_duality_map} on $\{\nu=1\}$: since
    $J_p = \grad \bigl( \tfrac 1p \nu^p \bigr) = \nu^{p-1} \grad \nu$, one has $J_p(x) = \grad
    \nu(x)$ whenever $\nu(x) = 1$.

    Conversely, if $\phi$ is concave and positively $1$-homogeneous, then
    $\phi(A) = \inf \{ A \bullet M : M \in K \}$ with $K = - \subgrad(-\phi)(0)$, and $\phi$ is of
    the form~$\phi_{\nu}$ precisely when $K$ is the closed convex hull of
    $\bigl\{ \beta_X(x) \cdot \grad \nu(x) : \nu(x) = 1 \bigr\}$. By Euler's identity
    $\grad \nu(x) \bullet x = \nu(x)$, the latter consists of endomorphisms of rank~$1$ and
    trace~$1$. Accretivity functionals thus form the subclass of concave barriers generated by
    rank-one duality pairings.
\end{remark}

\begin{definition}[\protect{cf.~\cite[p.~141]{Busemann1961}}]
    \label{def:totally_convex_norm}
    Assume $\varphi : \extpower{\vdim}  \R^{\adim} \to \R$ is a norm. We say that $\varphi$ is
    \emph{totally convex} if
    \begin{displaymath}
        \subgrad \varphi(\xi) \cap \ray \ograss{\adim}{\vdim} \ne \varnothing
        \quad
        \text{whenever $\xi \in \ray \ograss{\adim}{\vdim}$} \,.
    \end{displaymath}
\end{definition}

\begin{remark}
    \label{rem:total_convexity_BS}
    Total convexity in the sense of~\ref{def:totally_convex_norm} was defined by
    Busemann~\cite[p.~141]{Busemann1961}, who attributes the notion to
    Carath{\'e}odory~\cite[p.~140]{Busemann1961}: at every non-zero simple $\vdim$-vector~$\xi$
    the norm~$\varphi$ admits a~supporting hyperplane $\{ \eta : \zeta \bullet \eta = \varphi(\xi) \}$
    whose normal $\zeta \in \subgrad \varphi(\xi)$ is itself a~simple $\vdim$-vector. Earlier,
    Busemann and Straus~\cite[p.~36]{Busemann1960} defined total convexity of~$\varphi$ by
    requiring that for $T \in \grass{\adim}{\vdim}$ there exist a~projection
    $P \in \Proj{\adim}{\vdim}$ with $\im P = T$ such that $\extpower{\vdim} P$ is
    $\varphi$-non-increasing, i.e., $\varphi(\extpower{\vdim}P \xi) \le \varphi(\xi)$ for
    $\xi \in \extpower{\vdim} \R^{\adim}$. The two definitions are equivalent
    by~\cite[p.~146]{Busemann1961}; this is the case $d = 0$ of the result proven there,
    cf.~\cite[p.~145, item~1)]{Busemann1961}.
\end{remark}

\begin{remark}
    \label{rem:totally_convex_norms_and_AC}
    Let $F$ be a~polyconvex integrand associated to a~norm~$\varphi$ of class~$\cnt{1}$ away
    from the origin. Let $S,T \in \pgrass{\adim}{\vdim}$ and $\xi \in \ograss{\adim}{\vdim}$ satisfy
    $\Sp \xi = \im T$. Since $\varkappa(\id{\R^{\adim}})\xi = \vdim \xi$ and
    $\grad \varphi(\xi) \bullet \xi = \varphi(\xi) = F(T)$, we get from~\ref{lem:BF_via_varkappa}
    \begin{multline}
        F(T)\, \transpose{P}_F(T) \bullet Q_F(S)
        = B_F(T) \bullet Q_F(S)
        \\
        = \grad \varphi(\xi) \bullet \varkappa(Q_F(S))\xi
        = \vdim \varphi(\xi) - \grad \varphi(\xi) \bullet \varkappa(P_F(S))\xi \,;
    \end{multline}
    hence, recalling~\ref{def:SAC}, $F \in \SAC{}$ if and only if
    $\grad \varphi(\xi) \bullet \varkappa(P_F(S)) \xi < \vdim \varphi(\xi)$ whenever $S \ne T$. If
    $\varkappa(P_F(S))/\vdim$ were $\varphi$-non-increasing, then
    \begin{displaymath}
        \grad \varphi(\xi) \bullet \varkappa(P_F(S)/\vdim) \xi
        \le \polar{\varphi}(\grad \varphi(\xi)) \cdot \varphi(\varkappa(P_F(S)/\vdim) \xi)
        \le \varphi(\xi) \,,
    \end{displaymath}
    which is the non-strict version of this inequality. Total convexity, on the other hand,
    concerns the operators $\extpower{\vdim} P$ for projections~$P$
    (see~\ref{rem:total_convexity_BS}), which is completely different from
    $\varkappa(P_F(S))/\vdim$. Contractivity of the latter operators does not force
    $\varphi$ to be Euclidean: if $\vdim=1$, $\varphi(u)=1$, and $T=(\R u)_{\natural}$, then
    \begin{displaymath}
        \varphi(P_F(T)x)=|\grad\varphi(u)\bullet x|\le\varphi(x)
        \quad\text{for $x\in\R^{\adim}$} \,.
    \end{displaymath}
\end{remark}

\begin{question}
    \label{q:total_convexity_SAC}
    Let $\varphi : \extpower{\vdim} \R^{\adim} \to \R$ be a~strictly convex, totally convex
    norm of class~$\cnt{1}$ away from the origin (see~\ref{def:totally_convex_norm}) and let $F$
    be the polyconvex integrand associated to~$\varphi$ (see~\ref{def:polyconvex}). Does $F$
    satisfy \AC{}? Does it satisfy \EC{}?
\end{question}

\section{Quadratic exposed condition and the Caccioppoli-type inequality}
\label{sec:uniform_AC}

\begin{miniremark}
    Here we propose the~\emph{quadratic exposed condition}~\QEC{} of~\ref{def:QEC}, a~uniform
    version of the strong exposed condition~\SEC{} of~\ref{def:SEC}. For integrands of class~$\cnt{2}$ every normal realising~\QEC{} is
    automatically compatible with~$P_F$; see~\ref{lem:QEC_compatibility}. For such integrands
    Lipschitzian normals can be chosen by~\ref{thm:continuous_normals}. The definition of~\QEC{}
    imposes no regularity on the field of normals. We~then introduce yet
    another condition on the integrand, abbreviated \CI{}, enforcing the validity of a~Caccioppoli-type
    inequality; see~\ref{def:CI}. The main result of this section states that \CI{} is equivalent
    to~\QEC{} for integrands of class~$\cnt{2}$; see~\ref{prop:tilt-height},
    \ref{thm:CI_implies_LQ}, \ref{cor:LQ_iff_QEC}, and~\ref{cor:CI_iff_QEC}.
\end{miniremark}

\begin{definition}
    \label{def:QEC}
    We say that a $\vdim$-integrand $F$ satisfies the \emph{quadratic exposed condition} \QEC{} if
    there exist $0 < \Gamma < \infty$ and a~function
    $N : \pgrass{\adim}{\vdim} \to \End{\R^{\adim}} \cap \{ A : \| A \| = 1 \}$ such that for
    $S,T \in \pgrass{\adim}{\vdim}$
    \begin{gather}
        \label{eq:QEC_normal}
        T \circ N(T) = 0
        \\
        \label{eq:QEC_uniform_conv}
        \text{and} \quad
        P_F(S) \bullet N(T)
        \ge \Gamma \| S - T \|^2
        \,.
    \end{gather}
    If the choice of $\Gamma$ and $N$ is clear from the context or irrelevant, we write
    $F \in \QEC{}$; otherwise we employ the notation $F \in \QEC(\Gamma,N)$.
\end{definition}

\begin{remark}
    \label{rem:USAC_in_QEC}
    De~Rosa and Tione introduce in~\cite[Definition~3.3]{DeRosa2022} the \emph{uniform scalar atomic condition}
    (\USAC{}). We say that $F \in \USAC{}$ if there exists $\Gamma > 0$ for which
    \begin{displaymath}
        \transpose{P}_F(T) \bullet Q_F(S) \ge \Gamma \|S-T\|^2
        \quad \text{for $S,T \in \pgrass{\adim}{\vdim}$ with $S \ne T$} \,.
    \end{displaymath}
    As in the proof of~\ref{prop:SAC_in_EC}, $T \circ \transpose{Q}_F(T) = 0$ and
    $P_F(S) \bullet \transpose{Q}_F(T) = \transpose{P}_F(S) \bullet Q_F(T)$. Since
    $Q_F(T) \ne 0$ and $M = \sup \{ \|Q_F(T)\| : T \in \pgrass{\adim}{\vdim} \} < \infty$, setting
    $N(T) = \transpose{Q}_F(T) / \|Q_F(T)\|$ we get $F \in \QEC(\Gamma/M,N)$; thus,
    \begin{displaymath}
        \USAC{} \subseteq \QEC{} \,.
    \end{displaymath}
    In~particular, the family of integrands satisfying $\QEC{}$ is non-empty
    by~\cite[Proposition~3.10]{DeRosa2022}; non-perturbative examples are provided
    by~\cite[Theorem~1.3 and Corollary~4.7]{DeRosa2026v1}.
\end{remark}

\begin{lemma}
    \label{lem:QEC_compatibility}
    Let $F$ be a~$\vdim$-integrand of class~$\cnt{1}$ away from the origin, let
    $T \in \pgrass{\adim}{\vdim}$ be such that $P_F | \pgrass{\adim}{\vdim}$ is differentiable
    at~$T$, and let $N \in \End{\R^{\adim}}$ satisfy
    \begin{displaymath}
        T \circ N = 0
        \quad \text{and} \quad
        P_F(S) \bullet N \ge 0
        \quad \text{for $S \in \pgrass{\adim}{\vdim}$} \,.
    \end{displaymath}
    Then $N \circ \transpose{P}_F(T) = 0$. In~particular, if $F \in \QEC(\Gamma,N)$ is of
    class~$\cnt{2}$ away from the origin, then $N(T) \circ \transpose{P}_F(T) = 0$ for
    $T \in \pgrass{\adim}{\vdim}$.
\end{lemma}

\begin{proof}
    Recall from~\ref{rem:se-cse} that $P_F(S) \circ P_F(S) = P_F(S)$ and
    $\im P_F(S) = \im S$ for $S \in \pgrass{\adim}{\vdim}$; hence,
    $\ker \transpose{P}_F(S) = (\im P_F(S))^{\perp} = \im S^{\perp}$. Abbreviate $P = P_F(T)$ and
    define $f(S) = P_F(S) \bullet N$ for $S \in \pgrass{\adim}{\vdim}$.

    Since $T \circ N = 0$, we have
    $\im N \subseteq \im T^{\perp} = \ker \transpose{P}$; thus, $\transpose{P} \circ N = 0$ and
    $f(T) = \trace(\transpose{P} \circ N) = 0$. As $f \ge 0$, the point~$T$ is a~minimum point of~$f$
    and, $f$ being differentiable at~$T$,
    \begin{equation}
        \label{eq:QECc:critical}
        Y_X \bullet N = 0
        \quad \text{for $X \in \Tan(\pgrass{\adim}{\vdim},T)$} \,,
    \end{equation}
    where $Y_X = \uD (P_F|\pgrass{\adim}{\vdim})(T)X$.

    Let $X \in \Tan(\pgrass{\adim}{\vdim},T)$, abbreviate $Y = Y_X$, and let
    $\gamma$ be a~curve of class~$\cnt{1}$ in~$\pgrass{\adim}{\vdim}$ with $\gamma(0) = T$ and
    $\gamma'(0) = X$. Differentiating $P_F(\gamma(t)) \circ P_F(\gamma(t)) = P_F(\gamma(t))$ at
    $t = 0$ gives
    \begin{equation}
        \label{eq:QECc:DPF}
        Y = P \circ Y + Y \circ P \,.
    \end{equation}
    Differentiating $P_F(\gamma(t)) \circ \gamma(t) = \gamma(t)$ at $t = 0$ gives
    $Y \circ T + P \circ X = X$. Composing with $T^{\perp}$ on the left and with~$T$ on the
    right, and observing $T^{\perp} \circ P = 0$ because $\im P = \im T$, we obtain
    \begin{equation}
        \label{eq:QECc:YX}
        T^{\perp} \circ Y \circ T = T^{\perp} \circ X \circ T \,.
    \end{equation}

    Set $Z = N \circ \transpose{P}$. From~\eqref{eq:QECc:DPF} and
    $\transpose{P} \circ N = 0$ we get
    $(P \circ Y) \bullet N = \trace(\transpose{Y} \circ \transpose{P} \circ N) = 0$; hence,
    \begin{displaymath}
        Y \bullet N
        = (Y \circ P) \bullet N
        = \trace \bigl( \transpose{P} \circ \transpose{Y} \circ N \bigr)
        = \trace \bigl( \transpose{Y} \circ N \circ \transpose{P} \bigr)
        = Y \bullet Z \,.
    \end{displaymath}
    Moreover $\im Z \subseteq \im N \subseteq \im T^{\perp}$ and
    $Z \circ T^{\perp} = N \circ \transpose{P} \circ T^{\perp} = 0$, because
    $\im T^{\perp} = \ker \transpose{P}$; thus, $Z = T^{\perp} \circ Z \circ T$ and
    \begin{displaymath}
        Y \bullet Z
        = \trace \bigl( \transpose{Y} \circ T^{\perp} \circ Z \circ T \bigr)
        = \bigl( T^{\perp} \circ Y \circ T \bigr) \bullet Z
        = \bigl( T^{\perp} \circ X \circ T \bigr) \bullet Z \,,
    \end{displaymath}
    where the last equality is~\eqref{eq:QECc:YX}. Given $M \in \End{\R^{\adim}}$, the map
    $X = \pi(T)M$ belongs to $\Tan(\pgrass{\adim}{\vdim},T)$
    by~\ref{rem:sigma_props}\eqref{i:im_pi} and satisfies
    $T^{\perp} \circ X \circ T = T^{\perp} \circ M \circ T$. Taking $M = Z$
    in~\eqref{eq:QECc:critical} we obtain $Z \bullet Z = 0$, i.e., $Z = 0$.

    If $F$ is of class~$\cnt{2}$ away from the origin, then $P_F$ is of class~$\cnt{1}$ near
    $\pgrass{\adim}{\vdim}$, and the last assertion follows since~\ref{def:QEC} gives
    $T \circ N(T) = 0$ and $P_F(S) \bullet N(T) \ge 0$ for $S,T \in \pgrass{\adim}{\vdim}$.
\end{proof}

\begin{remark}
    \label{rem:basic_QEC}
    Assume $F \in \QEC(\Gamma,N)$ and $N(T) \circ \transpose{P}_F(T) = 0$ for
    $T \in \pgrass{\adim}{\vdim}$; by~\ref{lem:QEC_compatibility} the latter holds whenever $F$ is
    of class~$\cnt{2}$ away from the origin. Then
    $\im \transpose{P}_F(T) \subseteq \ker N(T)$ for $T \in \pgrass{\adim}{\vdim}$. Note that
    $N(T) \circ \transpose{P}_F(S) \ne 0$ whenever $S,T \in \pgrass{\adim}{\vdim}$ and
    $S \ne T$ by~\ref{def:QEC}\eqref{eq:QEC_uniform_conv}; hence, it~follows that the map
    \begin{displaymath}
        \bigl[ S \mapsto \project{(\im \transpose{P}_F(S))} \bigr]
        : \pgrass{\adim}{\vdim} \to \pgrass{\adim}{\vdim}
    \end{displaymath}
    is injective and a~homeomorphism; cf.~\ref{prop:AC_imPF_homeo}. Furthermore, we have
    $\im \transpose{P}_F(T) = \ker N(T)$ because otherwise $\dim \ker N(T) > \vdim$ and
    there would exist $S \in \pgrass{\adim}{\vdim}$ with $S \ne T$ such that
    $\im \transpose{P}_F(S) \subseteq \ker N(T)$, which is impossible because
    $N(T) \circ \transpose{P}_F(S) \ne 0$ for $S \ne T$.
\end{remark}

\begin{remark}
    \label{rem:QEC_second_order}
    Let $S,T\in\pgrass{\adim}{\vdim}$, $P=P_F(T)$, $D=P_F(S)-P$, and
    $N\in\End{\R^{\adim}}$ satisfy $T\circ N=0=N\circ\transpose{P}$.
    By~\ref{rem:se-cse}, $\transpose{P}\circ N=0$ and
    $D=P\circ D+D\circ P+D\circ D$. Pairing with~$N$ gives
    \begin{equation}
        \label{eq:QEC_height}
        P_F(S)\bullet N=(D\circ D)\bullet N \,,
    \end{equation}
    because $(P\circ D)\bullet N=(D\circ P)\bullet N=P\bullet N=0$.
    If $F\in\QEC(\Gamma,N)$ is of class~$\cnt{2}$ away from the origin, then
    \ref{lem:QEC_compatibility} allows~\eqref{eq:QEC_height} with $N(T)$ in place of~$N$.
    For $X\in\Tan(\pgrass{\adim}{\vdim},T)$, choose a~smooth curve~$\gamma$ in the Grassmannian
    with $\gamma(0)=T$ and $\gamma'(0)=X$. Substitute $S=\gamma(t)$ in~\eqref{eq:QEC_height},
    divide by $t^2$, and let $t\to0$ to obtain
    \begin{equation}
        \label{eq:necessary_for_QEC}
        \bigl(\uD P_F(T)X\circ\uD P_F(T)X\bigr)\bullet N(T)\ge\Gamma\|X\|^2 \,.
    \end{equation}
\end{remark}

\begin{miniremark}
    The identity~\ref{rem:QEC_second_order}\eqref{eq:QEC_height} permits transport of a~fixed
    normal to neighbouring planes using only first derivatives of~$P_F$. A~partition of
    unity then gives a~regular field of normals.
\end{miniremark}

\begin{theorem}
    \label{thm:continuous_normals}
    Suppose $l$ is an~integer with $l\ge2$ and $F\in\QEC(\Gamma,N)$ is of class~$\cnt{l}$ away from the origin.
    There exist an~open neighbourhood~$U$ of $\pgrass{\adim}{\vdim}$ and
    $N':U\to\End{\R^{\adim}}$ of class~$\cnt{l-1}$, non-zero on the Grassmannian, such that
    \begin{gather*}
        \tilde N(T)=N'(T)/\|N'(T)\| \quad \text{for $T\in\pgrass{\adim}{\vdim}$} \,,\\
        F\in\QEC(\Gamma/4,\tilde N) \,,\qquad \Lip \tilde N<\infty \,.
    \end{gather*}
\end{theorem}

\begin{proof}
    Fix $R\in\pgrass{\adim}{\vdim}$ and set
    $A(T)=T^{\perp}\circ N(R)\circ\transpose{Q}_F(T)$ on a~neighbourhood of~$\pgrass{\adim}{\vdim}$.
    By~\ref{def:aux-objects}, $A$ is of class~$\cnt{l-1}$; by~\ref{lem:QEC_compatibility}
    and~\ref{rem:se-cse}, for $T\in\pgrass{\adim}{\vdim}$,
    \begin{equation}
        \label{eq:cn:transport}
        A(R)=N(R) \,,\qquad T\circ A(T)=0=A(T)\circ\transpose{P}_F(T) \,.
    \end{equation}
    There is a~neighbourhood~$W$ of~$R$ such that, for $S\in\pgrass{\adim}{\vdim}$ and
    $T\in\pgrass{\adim}{\vdim}\cap W$,
    \begin{equation}
        \label{eq:cn:local}
        \|A(T)\|\le2 \,,\qquad
        P_F(S)\bullet A(T)\ge\tfrac12\Gamma\|S-T\|^2 \,.
    \end{equation}
    Indeed, failure of the second inequality gives $T_j\in\pgrass{\adim}{\vdim}$, $T_j\to R$, and
    $S_j\in\pgrass{\adim}{\vdim}\without\{T_j\}$
    violating it. Passing to a~subsequence, $S_j\to S$. If $S\ne R$, continuity and
    \eqref{eq:cn:transport} contradict~\ref{def:QEC}\eqref{eq:QEC_uniform_conv}; hence $S=R$.
    Set $d_j=\|S_j-T_j\|$ and pass to a~further subsequence with
    $X_j=(S_j-T_j)/d_j\to X$. Since
    $X_j=T_j\circ X_j+X_j\circ T_j+d_jX_j\circ X_j$ and $\transpose{X_j}=X_j$,
    \ref{rem:sigma_props}\eqref{eq:tan_grass} gives $X\in\Tan(\pgrass{\adim}{\vdim},R)$ and $\|X\|=1$.
    Integration of~$\uD P_F$ along the segment from $T_j$ to $S_j$ and its continuity give
    $(P_F(S_j)-P_F(T_j))/d_j\to Y=\uD P_F(R)X$.
    Thus~\eqref{eq:cn:transport} and
    \ref{rem:QEC_second_order}\eqref{eq:QEC_height}\eqref{eq:necessary_for_QEC} imply
    \begin{displaymath}
        \tfrac12\Gamma\ge
        \lim_{j\to\infty}d_j^{-2}P_F(S_j)\bullet A(T_j)
        =(Y\circ Y)\bullet N(R)\ge\Gamma \,,
    \end{displaymath}
    a~contradiction. The norm bound in~\eqref{eq:cn:local} follows by continuity and
    $\|A(R)\|=1$, after reducing~$W$.

    Choose finitely many such neighbourhoods $U_i$ covering~$\pgrass{\adim}{\vdim}$, with transported fields~$A_i$,
    and a~smooth partition of unity~$\{\theta_i\}$ on a~neighbourhood~$U$ of~$\pgrass{\adim}{\vdim}$
    subordinate to them; see~\cite[3.1.13]{Federer1969}. Set $N'=\sum_i\theta_i A_i$.
    By~\eqref{eq:cn:transport}\eqref{eq:cn:local}, for $S,T\in\pgrass{\adim}{\vdim}$,
    \begin{displaymath}
        T\circ N'(T)=0 \,,\qquad \|N'(T)\|\le2 \,,\qquad
        P_F(S)\bullet N'(T)\ge\tfrac12\Gamma\|S-T\|^2 \,.
    \end{displaymath}
    Taking $S\ne T$ gives $N'(T)\ne0$, so normalisation yields
    $F\in\QEC(\Gamma/4,\tilde N)$. The map $N'|\pgrass{\adim}{\vdim}$ is Lipschitzian, since
    $\pgrass{\adim}{\vdim}$ is compact and $N'$ is of class~$\cnt{1}$ on~$U$.
    Moreover, $\inf\{\|N'(Z)\|:Z\in\pgrass{\adim}{\vdim}\}>0$; hence
    \begin{displaymath}
        \|\tilde N(S)-\tilde N(T)\|
        \le \frac{2\|N'(S)-N'(T)\|}{\inf\{\|N'(Z)\|:Z\in\pgrass{\adim}{\vdim}\}}
        \quad \text{for $S,T\in\pgrass{\adim}{\vdim}$} \,,
    \end{displaymath}
    which proves the last assertion.
\end{proof}

\begin{lemma}
    \label{lem:QEC_implies_EC}
    The inclusions $\QEC \subseteq \SEC \subseteq \EC$ hold.
\end{lemma}

\begin{proof}
    If $F \in \QEC(\Gamma,N)$, then $N$ satisfies the conditions of~\ref{def:SEC}, since
    $\Gamma \|S-T\|^2 > 0$ for $S \ne T$. The second inclusion is~\ref{prop:SAC_in_EC}.
\end{proof}

\begin{definition}
    \label{def:VF}
    Whenever $U \subseteq \R^{\adim}$ is open, $F$ is an integrand, and
    $V \in \Var{\vdim}(U)$, we define $V_F \in \Var{\vdim}(U)$ by
    \begin{displaymath}
        V_F(\alpha) = \textint{}{} \alpha(x,T) F(T) \ud V(x,T)
        \quad \text{for $\alpha \in \ccspace{U \times \pgrass{\adim}{\vdim}}$} \,.
    \end{displaymath}
    Recalling~\eqref{eq:intro_fst_var} and $B_F = F \transpose{P}_F$, we have
    \begin{displaymath}
        \delta_FV(g)
        = \textint{}{} B_F(S) \bullet \uD g(x) \ud V(x,S)
        = \textint{}{} \transpose{P}_F(S) \bullet \uD g(x) \ud V_F(x,S)
        \quad \text{for $g \in \VF(U)$} \,.
    \end{displaymath}
\end{definition}

\begin{definition}
    \label{def:CI}
    We say that a~$\vdim$-integrand~$F$ satisfies the \emph{Caccioppoli-type inequality}, and
    write $F \in \CI{}$, if there exist an open set $U \subseteq \R^{\adim}$, a~smooth function
    $\zeta : U \to \R \cap \{ t : 0 \le t \le 1 \}$ with compact support and $\zeta \ne 0$, and
    a~number $0 < \Gamma < \infty$, such that
    \begin{multline}
        \label{eq:CI}
        \textint{}{} \zeta(x)^2 \|S-T\|^2 \ud V_F(x,S)
        \le \Gamma \max \Bigl\{
        \textint{}{} \zeta(x)^2 |T^{\perp}(x-a)| \ud \|\delta_FV\|(x) \,,
        \\
        \textint{}{} |\grad \zeta(x)|^2 |T^{\perp}(x-a)|^2 \ud \|V_F\|(x)
        \Bigr\}
    \end{multline}
    holds for $V \in \Var{\vdim}(U)$ with $\|\delta_FV\|$ Radon, $T \in \pgrass{\adim}{\vdim}$, and
    $a \in \R^{\adim}$. If the choice of~$U$, $\zeta$, and~$\Gamma$ matters, we write
    $F \in \CI(U,\zeta,\Gamma)$.
\end{definition}

\begin{proposition}[\protect{Caccioppoli-type inequality}]
    \label{prop:tilt-height}
    Assume $F$ is a~$\vdim$-integrand of class~$\cnt{2}$ away from the origin,
    $F \in \QEC(\Gamma_0,N)$, $\Lambda = \Lip(P_F|\pgrass{\adim}{\vdim})$,
    $U \subseteq \R^{\adim}$ is open, and $\zeta : U \to \R \cap \{ t : 0 \le t \le 1 \}$ is
    smooth with compact support. Set
    $\Gamma = \max \{ 16 \Lambda^2 \Gamma_0^{-2} ,\, 2 \Gamma_0^{-1} \}$.
    Then~\ref{def:CI}\eqref{eq:CI} holds
    for all $V \in \Var{\vdim}(U)$ with $\|\delta_FV\|$ Radon, $T \in \pgrass{\adim}{\vdim}$, and
    $a \in \R^{\adim}$. In~particular, $F \in \CI{}$.
\end{proposition}

\begin{proof}
    Fix
    $V$, $T$, and~$a$, and abbreviate
    \begin{gather}
        X = \textint{}{} \zeta(x)^2 \|S-T\|^2 \ud V_F(x,S) \,,
        \quad
        Y = \textint{}{} |\grad \zeta(x)|^2 |T^{\perp}(x-a)|^2 \ud \|V_F\|(x) \,,
        \\
        Z = \textint{}{} \zeta(x)^2 |T^{\perp}(x-a)| \ud \|\delta_FV\|(x) \,;
    \end{gather}
    these are finite since $\spt \zeta$ is compact. By~\ref{def:QEC} we have $T \circ N(T) = 0$,
    i.e., $\im N(T) \subseteq \im T^{\perp}$, so that
    $\transpose{N}(T) = \transpose{N}(T) \circ T^{\perp}$; by~\ref{lem:QEC_compatibility} we have
    $N(T) \circ \transpose{P}_F(T) = 0$, so that
    $P_F(S) \circ \transpose{N}(T) = (P_F(S) - P_F(T)) \circ \transpose{N}(T)$ for
    $S \in \pgrass{\adim}{\vdim}$. Define $g \in \VF(U)$ by
    \begin{displaymath}
        g(x) = \zeta(x)^2 \transpose{N}(T) (x-a)
        \quad \text{for $x \in U$} \,.
    \end{displaymath}
    Since $\|N(T)\| = 1$, we get $|g(x)| \le \zeta(x)^2 |T^{\perp}(x-a)|$; hence,
    $|\delta_FV(g)| \le Z$. Moreover,
    \begin{displaymath}
        \transpose{P}_F(S) \bullet \uD g(x)
        = \zeta(x)^2 P_F(S) \bullet N(T)
        + 2 \zeta(x)\, \transpose{P}_F(S) \grad \zeta(x) \bullet \transpose{N}(T)(x-a)
    \end{displaymath}
    and
    \begin{multline}
        \bigl| \transpose{P}_F(S) \grad \zeta(x) \bullet \transpose{N}(T) (x-a) \bigr|
        = \bigl| \grad \zeta(x) \bullet (P_F(S) - P_F(T)) \circ \transpose{N}(T) \circ T^{\perp}
        (x-a) \bigr|
        \\
        \le \Lambda |\grad \zeta(x)| \cdot \| S - T \| \cdot |T^{\perp} (x-a)|
        \quad \text{for $x \in U$ and $S \in \pgrass{\adim}{\vdim}$} \,.
    \end{multline}
    Using~\ref{def:QEC}\eqref{eq:QEC_uniform_conv}, \ref{def:VF}, and the Cauchy--Schwarz inequality, we
    obtain
    \begin{multline}
        \Gamma_0 X
        \le \textint{}{} \zeta(x)^2 P_F(S) \bullet N(T) \ud V_F(x,S)
        \\
        = \delta_FV(g)
        - 2 \textint{}{} \zeta(x)\, \transpose{P}_F(S) \grad \zeta(x) \bullet \transpose{N}(T)(x-a)
        \ud V_F(x,S)
        \le Z + 2 \Lambda \sqrt{XY} \,.
    \end{multline}
    If $2 \Lambda \sqrt{XY} \ge Z$, then $\Gamma_0 X \le 4 \Lambda \sqrt{XY}$ and
    $X \le 16 \Lambda^2 \Gamma_0^{-2} Y$. Otherwise, $\Gamma_0 X < 2Z$ and
    $X \le 2 \Gamma_0^{-1} Z$.
\end{proof}

\begin{theorem}
    \label{thm:CI_implies_LQ}
    Assume $F$ is a~$\vdim$-integrand of class~$\cnt{1}$ away from the origin and
    $F \in \CI(U,\zeta,\Gamma)$. Then there exists $0 < C < \infty$, depending only on $\Gamma$,
    $\zeta$, and~$\adim - \vdim$, such that
    \begin{equation}
        \label{eq:LQ}
        \textint{}{} \|S-T\|^2 F(S) \ud \mu(S)
        \le C \bigl| \bar{A}(\mu) \circ T^{\perp} \bigr|
    \end{equation}
    whenever $\mu$ is a probability measure over $\pgrass{\adim}{\vdim}$
    and~$T \in \pgrass{\adim}{\vdim}$; here, $\bar{A}(\mu) = \int B_F \ud \mu$ as
    in~\ref{rem:AC_orig}.
\end{theorem}

\begin{proof}
    Fix $a \in U$ with $\zeta(a) > 0$ and $0 < r < \infty$ such that
    $\cball{a}{r} \subseteq U \cap \{ x : \zeta(x) > \tfrac 12 \zeta(a) \}$. Let $\mu$ and $T$ be
    given and abbreviate $\bar{A} = \bar{A}(\mu)$. Choose non-negative functions
    $\psi \in \cnt{\infty}(\im T)$ and $\chi \in \cnt{\infty}(\im T^{\perp})$, radial about the
    origin of the respective subspaces, with supports contained in the balls of radius $r/2$ and
    with
    \begin{displaymath}
        \textint{\im T}{} \psi \ud \HM^{\vdim} = 1
        \quad \text{and} \quad
        \textint{\im T^{\perp}}{} \chi \ud \HM^{\adim-\vdim} = 1 \,.
    \end{displaymath}
    For $0 < \varepsilon \le 1$ define $\theta_{\varepsilon} \in \cnt{\infty}(\R^{\adim})$ and
    $V^{\varepsilon} \in \Var{\vdim}(U)$ by
    \begin{gather}
        \theta_{\varepsilon}(x)
        = \varepsilon^{\vdim-\adim} \psi \bigl( T(x-a) \bigr)
        \chi \bigl( \varepsilon^{-1} T^{\perp}(x-a) \bigr)
        \quad \text{for $x \in \R^{\adim}$} \,,
        \\
        V^{\varepsilon}(\alpha)
        = \textint{}{} \textint{}{} \alpha(x,S) \theta_{\varepsilon}(x) \ud \LM^{\adim}(x) \ud \mu(S)
        \quad \text{for $\alpha \in \ccspace{U \times \pgrass{\adim}{\vdim}}$} \,.
    \end{gather}
    By Fubini's theorem, $\int \theta_{\varepsilon} \ud \LM^{\adim} = 1$; moreover,
    $\spt \theta_{\varepsilon} \subseteq \cball{a}{r}$ and $|T^{\perp}(x-a)| \le \varepsilon r/2$
    for $x \in \spt \theta_{\varepsilon}$.

    For $g \in \VF(U)$ we get, by~\ref{def:VF} and integration by parts,
    \begin{multline}
        \bigl| \delta_FV^{\varepsilon}(g) \bigr|
        = \bigl| \textint{}{} \theta_{\varepsilon}(x)\, \bar{A} \bullet \uD g(x) \ud \LM^{\adim}(x) \bigr|
        = \bigl| - \textint{}{} \bigl( \bar{A} \grad \theta_{\varepsilon}(x) \bigr) \bullet g(x) 
        \ud \LM^{\adim}(x) \bigr|
        \\
        \le \sup \bigl\{ \bigl| \bar{A} \grad \theta_{\varepsilon}(x) \bigr| : x \in U \bigr\} \textint{}{} | g |
        \ud \LM^{\adim} \,;
    \end{multline}
    hence, $\|\delta_FV^{\varepsilon}\|$ is Radon and $V^{\varepsilon}$ is admissible
    in~\ref{def:CI}\eqref{eq:CI}.

    Since $\psi$ and $\chi$ are defined on $\im T$ and $\im T^{\perp}$ respectively,
    \begin{displaymath}
        \grad \theta_{\varepsilon}(x)
        = \varepsilon^{\vdim-\adim} \Bigl(
        \chi \bigl( \tfrac{T^{\perp}(x-a)}{\varepsilon} \bigr) \grad \psi \bigl( T(x-a) \bigr)
        + \varepsilon^{-1} \psi \bigl( T(x-a) \bigr)
        \grad \chi \bigl( \tfrac{T^{\perp}(x-a)}{\varepsilon} \bigr)
        \Bigr) \,,
    \end{displaymath}
    where $\grad \psi(y) \in \im T$ and $\grad \chi(z) \in \im T^{\perp}$. Substituting
    $x = a + y + \varepsilon z$ with $y \in \im T$ and $z \in \im T^{\perp}$, and using $\zeta \le
    1$, $\int \psi \ud \HM^{\vdim} = 1$, and $|\bar{A} \circ T^{\perp} v| \le |\bar{A} \circ
    T^{\perp}|\, |v|$ for $v \in \R^{\adim}$, we obtain
    \begin{multline}
        \label{eq:CI:height-fst-var}
        \textint{}{} \zeta^2 |T^{\perp}(x-a)| \ud \|\delta_FV^{\varepsilon}\|
        \\
        \le \varepsilon |\bar{A}| \Bigl( \textint{}{} |z| \chi(z) \ud \HM^{\adim-\vdim}(z) \Bigr)
        \Bigl( \textint{}{} |\grad \psi| \ud \HM^{\vdim} \Bigr)
        + \bigl| \bar{A} \circ T^{\perp} \bigr|
        \textint{}{} |z| |\grad \chi(z)| \ud \HM^{\adim-\vdim}(z) \,.
    \end{multline}
    The integrals involving $\psi$ and $\chi$ are finite and, $\psi$ and $\chi$ being radial,
    independent of~$T$; set $c = \int |z| |\grad \chi(z)| \ud \HM^{\adim-\vdim}(z)$. Similarly,
    \begin{equation}
        \label{eq:CI:height-grad-zeta}
        \textint{}{} |\grad \zeta|^2 |T^{\perp}(x-a)|^2 \ud \|V_F^{\varepsilon}\|
        \le \varepsilon^2 \tfrac{r^2}{4} \bigl( \sup |\grad \zeta| \bigr)^2 \sup F
        \lIm \pgrass{\adim}{\vdim} \rIm \,.
    \end{equation}

    Since $\zeta > \tfrac 12 \zeta(a)$ on $\spt \theta_{\varepsilon}$ and
    $\int \theta_{\varepsilon} \ud \LM^{\adim} = 1$,
    \begin{equation}
        \label{eq:CI:tilt-lower-bound}
        \textint{}{} \zeta^2 \|S-T\|^2 \ud V^{\varepsilon}_F
        \ge \tfrac 14 \zeta(a)^2 \textint{}{} \|S-T\|^2 F(S) \ud \mu(S) \,.
    \end{equation}
    Inserting~\eqref{eq:CI:height-fst-var}, \eqref{eq:CI:height-grad-zeta},
    and~\eqref{eq:CI:tilt-lower-bound} into~\ref{def:CI}\eqref{eq:CI} and letting
    $\varepsilon \downarrow 0$ yields~\eqref{eq:LQ} with $C = 4 \Gamma c\, \zeta(a)^{-2}$.
\end{proof}

\begin{corollary}
    \label{cor:LQ_iff_QEC}
    Let $F$ be a~$\vdim$-integrand of class~$\cnt{1}$ away from the origin.
    Then~\ref{thm:CI_implies_LQ}\eqref{eq:LQ}
    holds for some $0 < C < \infty$ if and only if $F \in \QEC{}$. In~particular,
    $\CI{} \subseteq \QEC{}$ for such integrands.
\end{corollary}

\begin{proof}
    Suppose~\ref{thm:CI_implies_LQ}\eqref{eq:LQ} holds and fix $T \in \pgrass{\adim}{\vdim}$.
    Denote by $\Pi$ the
    set of probability measures over $\pgrass{\adim}{\vdim}$, a~convex set which is compact in the
    weak topology, let $K = \End{\R^{\adim}} \cap \{ A : |A| \le 1 \}$, and set
    \begin{displaymath}
        \Psi(\mu,M)
        = C \Bigl( \textint{}{} B_F(S) \circ T^{\perp} \ud \mu(S) \Bigr) \bullet M
        - \textint{}{} \|S-T\|^2 F(S) \ud \mu(S)
        \quad \text{for $(\mu,M) \in \Pi \times K$} \,.
    \end{displaymath}
    Since the Euclidean norm on $\End{\R^{\adim}}$ is self-dual,
    \ref{thm:CI_implies_LQ}\eqref{eq:LQ} says precisely
    that $\sup \{ \Psi(\mu,M) : M \in K \} \ge 0$ for all $\mu \in \Pi$. The
    function~$\Psi$ is affine and continuous in each variable separately, and $\Pi$, $K$
    are convex and compact; hence, the minimax theorem~\cite[3.4]{Sion1958} yields
    \begin{displaymath}
        \sup_{M \in K} \inf_{\mu \in \Pi} \Psi(\mu,M)
        = \inf_{\mu \in \Pi} \sup_{M \in K} \Psi(\mu,M) \ge 0 \,.
    \end{displaymath}
    The function $[ K \ni M \mapsto \inf_{\mu \in \Pi} \Psi(\mu,M) ]$ is upper semicontinuous
    and $K$ is compact, so there is $M \in K$ with $\Psi(\mu,M) \ge 0$ for all $\mu \in \Pi$; taking
    $\mu = \Dirac{S}$ we get
    \begin{displaymath}
        C \bigl( B_F(S) \circ T^{\perp} \bigr) \bullet M \ge \|S-T\|^2 F(S)
        \quad \text{for $S \in \pgrass{\adim}{\vdim}$} \,.
    \end{displaymath}
    Since $(B_F(S) \circ T^{\perp}) \bullet M = B_F(S) \bullet (M \circ T^{\perp})$,
    $B_F(S) = F(S) \transpose{P}_F(S)$, and $F(S) > 0$, this means
    $P_F(S) \bullet N' \ge \|S-T\|^2$ for $S \in \pgrass{\adim}{\vdim}$, where
    $N' = C\, T^{\perp} \circ \transpose{M}$. Since $T\circ T^{\perp}=0$, we have $T \circ N' = 0$, and
    $N' \ne 0$ because $T$ is not isolated in~$\pgrass{\adim}{\vdim}$. As
    $\|N'\| \le |N'| \le C$, the map $N(T) = N' / \|N'\|$
    satisfies~\ref{def:QEC}\eqref{eq:QEC_normal}\eqref{eq:QEC_uniform_conv} with $\Gamma = C^{-1}$,
    which is independent of~$T$; thus, $F \in \QEC(C^{-1},N)$.

    Conversely, assume $F \in \QEC(\Gamma,N)$ and let $\mu \in \Pi$ and
    $T \in \pgrass{\adim}{\vdim}$. From $T \circ N(T) = 0$ we get
    $\transpose{N}(T) = \transpose{N}(T) \circ T^{\perp}$ and $\dim \im N(T) \le \adim - \vdim$, so
    that $|N(T)| \le (\adim-\vdim)^{1/2} \|N(T)\| = (\adim-\vdim)^{1/2}$. Multiplying~\eqref{eq:QEC_uniform_conv}
    by $F(S)$, integrating against~$\mu$, and using $F(S) P_F(S) = \transpose{B}_F(S)$ we obtain
    \begin{multline}
        \Gamma \textint{}{} \|S-T\|^2 F(S) \ud \mu(S)
        \le \textint{}{} \transpose{B}_F(S) \bullet N(T) \ud \mu(S)
        = \bar{A}(\mu) \bullet \transpose{N}(T)
        \\
        = \bigl( \bar{A}(\mu) \circ T^{\perp} \bigr) \bullet \transpose{N}(T)
        \le (\adim-\vdim)^{1/2} \bigl| \bar{A}(\mu) \circ T^{\perp} \bigr| \,,
    \end{multline}
    which is~\ref{thm:CI_implies_LQ}\eqref{eq:LQ} with $C = (\adim-\vdim)^{1/2} \Gamma^{-1}$. The
    last assertion follows
    from~\ref{thm:CI_implies_LQ}.
\end{proof}

\begin{corollary}
    \label{cor:CI_iff_QEC}
    Let $F$ be a~$\vdim$-integrand of class~$\cnt{2}$ away from the origin. Then
    \begin{displaymath}
        F \in \CI{} \quad \iff \quad F \in \QEC{} \,.
    \end{displaymath}
\end{corollary}

\begin{proof}
    The inclusion $\CI{} \subseteq \QEC{}$ is~\ref{cor:LQ_iff_QEC} and $\QEC{} \subseteq \CI{}$
    is~\ref{prop:tilt-height}.
\end{proof}

\section{Stability of the quadratic exposed condition}
\label{sec:QEC_stability}

\begin{miniremark}
    We~show that~\QEC{} is stable under small perturbations of class~$\cnt{1,1}$,
    provided the field of normals in~\ref{def:QEC} is Lipschitzian; see~\ref{thm:QEC_stability}.
    This is automatic for $\cnt{2}$ integrands by~\ref{thm:continuous_normals}.
    In~\ref{rem:regularity_missing_step} we record where the results of the present paper enter
    the $\varepsilon$-regularity theorem~\cite[Theorem~1.1]{DeRosa2026v1}.
\end{miniremark}

\begin{definition}
    \label{def:C11_norm}
    Give $\End{\R^{\adim}}$ the operator norm.
    For a~map $h:\pgrass{\adim}{\vdim}\to X$ of class~$\cnt{1}$ into a~normed vectorspace $(X,|\cdot|)$,
    let $\uD h(T)$ denote its intrinsic differential on $\Tan(\pgrass{\adim}{\vdim},T)$ and set
    \begin{gather}
        d_h(T)=\uD h(T)\circ\project{\Tan(\pgrass{\adim}{\vdim},T)}
        \in\Hom(\End{\R^{\adim}},X)\quad \text{for $T\in\pgrass{\adim}{\vdim}$} \,,\\
        \boldsymbol\nu_0(h)=\sup\{|h(T)|:T\in\pgrass{\adim}{\vdim}\} \,,
        \quad
        \boldsymbol\nu_1(h)=\boldsymbol\nu_0(h)+\boldsymbol\nu_0(d_h) \,,
        \quad
        \boldsymbol\nu_2(h)=\boldsymbol\nu_1(h)+\Lip d_h \,.
    \end{gather}
    Norms of differentials are the induced operator norms. For a~map defined near~$\pgrass{\adim}{\vdim}$,
    these symbols refer to its restriction to~$\pgrass{\adim}{\vdim}$.
\end{definition}

\begin{remark}
    \label{rem:FG_close}
    Let $F,G$ be $\vdim$-integrands of class~$\cnt{1,1}$ away from the origin, put
    $h=(\log\circ F-\log\circ G)|\pgrass{\adim}{\vdim}$, and let $E(T)$ be its intrinsic
    gradient for the Euclidean scalar product on~$\End{\R^{\adim}}$.
    By~\ref{def:aux-objects} and~\ref{rem:sigma_props}\eqref{i:pi_transpose},
    \begin{equation}
        \label{eq:intrinsic_P_difference}
        P_F(T)-P_G(T)=2\,T\circ E(T)\circ T^{\perp} \,.
    \end{equation}
    The Riesz identification takes $d_h(T)$ to~$E(T)$. Since this identification is linear
    and $[ \End{\R^{\adim}} \ni T \mapsto T^{\perp} ]$ is $1$-Lipschitzian,
    \eqref{eq:intrinsic_P_difference} gives a~constant $c=c(\adim,\vdim)\ge1$ such that the numbers
    \begin{gather}
        \Delta=\sup\{\|P_F(T)-P_G(T)\|:T\in\pgrass{\adim}{\vdim}\} \,,\qquad
        \Delta'=\Lip\bigl((P_F-P_G)|\pgrass{\adim}{\vdim}\bigr)
        \\
        \text{satisfy}\quad \Delta+\Delta'\le c\boldsymbol\nu_2(h) \,.
    \end{gather}
\end{remark}

\begin{theorem}
    \label{thm:QEC_stability}
    Assume $F\in\QEC(\Gamma,N_F)$ is of class~$\cnt{1,1}$ away from the origin and $N_F$ is
    Lipschitzian. There exists $\varepsilon=\varepsilon(\adim,\vdim,\Gamma,F)>0$ such that
    every $\vdim$-integrand~$G$ of class~$\cnt{1,1}$ away from the origin satisfying
    \begin{displaymath}
        \boldsymbol\nu_2\bigl((\log\circ F-\log\circ G)|\pgrass{\adim}{\vdim}\bigr)\le\varepsilon
    \end{displaymath}
    admits a~Lipschitzian field~$\tilde N_G$ with $G\in\QEC(\Gamma/2,\tilde N_G)$.
\end{theorem}

\begin{proof}
    Set $\Lambda=\Lip(P_F|\pgrass{\adim}{\vdim})$.
    By Rademacher's theorem in charts, $P_F|\pgrass{\adim}{\vdim}$ is differentiable on a~dense set.
    Thus~\ref{lem:QEC_compatibility} and continuity give
    \begin{equation}
        \label{eq:stab:compatibility}
        N_F(T)\circ\transpose{P}_F(T)=0\quad \text{for $T\in\pgrass{\adim}{\vdim}$} \,.
    \end{equation}
    Let $c$, $\Delta$, and~$\Delta'$ be as in~\ref{rem:FG_close} and choose
    \begin{displaymath}
        \varepsilon=c^{-1}\min\left\{\tfrac14,\,
        \frac{\Gamma}{4\adim\bigl(2\Lambda+1+(\Lambda+1)^2\bigr)}\right\} \,.
    \end{displaymath}
    Define $N_G(T)=N_F(T)\circ\transpose{Q}_G(T)$ for $T\in\pgrass{\adim}{\vdim}$.
    By~\ref{rem:se-cse}, \eqref{eq:stab:compatibility}, and
    $\ker\transpose{P}_G(T)=\im T^{\perp}$, for $T\in\pgrass{\adim}{\vdim}$,
    \begin{gather}
        \label{eq:stab:transport}
        T\circ N_G(T)=0=N_G(T)\circ\transpose{P}_G(T)
        =\transpose{P}_G(T)\circ N_G(T) \,,\\
        N_G(T)-N_F(T)=N_F(T)\circ
        \bigl(\transpose{P}_F(T)-\transpose{P}_G(T)\bigr) \,,\qquad
        \tfrac34\le1-\Delta\le\|N_G(T)\|\le1+\Delta\le\tfrac54 \,.
        \notag
    \end{gather}
    For $S,T\in\pgrass{\adim}{\vdim}$ set $d=\|S-T\|$ and $D_\theta=P_\theta(S)-P_\theta(T)$ for
    $\theta\in\{F,G\}$.  By~\ref{rem:QEC_second_order}\eqref{eq:QEC_height} and the definitions
    of~$\Lambda$ and~$\Delta'$,
    \begin{equation}
        \label{eq:stab:quadratic}
        P_\theta(S)\bullet N_\theta(T)=(D_\theta\circ D_\theta)\bullet N_\theta(T) \,,
        \qquad
        \|D_F\|\le\Lambda d \,,\quad \|D_G-D_F\|\le\Delta'd \,.
    \end{equation}
    By~\eqref{eq:stab:quadratic},
    \begin{multline}
        P_G(S)\bullet N_G(T)-P_F(S)\bullet N_F(T)
        \\
        = (D_G\circ D_G-D_F\circ D_F)\bullet N_F(T)
        + (D_G\circ D_G)\bullet\bigl(N_G(T)-N_F(T)\bigr) \,. 
    \end{multline}
    Since $D_G\circ D_G-D_F\circ D_F=(D_G-D_F)\circ D_G+D_F\circ(D_G-D_F)$,
    $\|D_G\|\le(\Lambda+\Delta')d$, $\|N_F(T)\|=1$, $\|N_G(T)-N_F(T)\|\le\Delta$
    by~\eqref{eq:stab:transport}, and $|A\bullet B|\le\adim\|A\|\|B\|$, we obtain
    \begin{displaymath}
        |P_G(S)\bullet N_G(T)-P_F(S)\bullet N_F(T)|
        \le\adim\bigl(\Delta'(2\Lambda+\Delta')+\Delta(\Lambda+\Delta')^2\bigr)d^2
        \le\tfrac14\Gamma d^2 \,.
    \end{displaymath}
    Hence $P_G(S)\bullet N_G(T)\ge\tfrac34\Gamma d^2$. By~\eqref{eq:stab:transport},
    $\tilde N_G(T)=N_G(T)/\|N_G(T)\|$ defines a~Lipschitzian field, $T\circ\tilde N_G(T)=0$, and
    \begin{displaymath}
        P_G(S)\bullet\tilde N_G(T)\ge\tfrac35\Gamma\|S-T\|^2
        \ge\tfrac12\Gamma\|S-T\|^2.
        \qedhere
    \end{displaymath}
\end{proof}

\begin{remark}
    The same argument proves stability when $N_F$ is merely continuous; the resulting
    field~$\tilde N_G$ is then continuous. For the scalar condition, compare~\cite[Proposition~3.10]{DeRosa2022}.
\end{remark}

\begin{remark}
    \label{rem:continuity_of_N}
    By~\ref{thm:continuous_normals}, every integrand in~$\QEC{}$ of class~$\cnt{2}$ away
    from the origin admits Lipschitzian normals, with a~possibly smaller constant.
    Thus~\ref{thm:QEC_stability} applies without a~separate normal-regularity assumption
    in this class. The local argument of~\ref{thm:continuous_normals} uses continuity
    of~$\uD P_F$ and does not establish this conclusion for integrands of class~$\cnt{1,1}$.
\end{remark}

\begin{remark}
    \label{rem:regularity_missing_step}
    The proof of~\cite[Proposition~2.1]{DeRosa2026v1} rests on two results of the present
    paper. It~uses~\ref{thm:QEC_stability} to produce $\alpha > 0$ for which $F - \alpha$ is
    positive on the Grassmannian and still satisfies~\QEC{} after extension
    by~\ref{rem:integrand_defined_by_values_on_Gnk}, in the role played
    by~\cite[Proposition~3.10]{DeRosa2022} for~\USAC{}; and it uses the inclusion $\QEC{} \subseteq \AC{}$
    of~\ref{lem:QEC_implies_EC} and~\ref{prop:EC_in_AC}. Moreover, \cite[Proposition~2.2]{DeRosa2026v1} is
    deduced from~\ref{prop:tilt-height}. What~\cite[Proposition~2.1]{DeRosa2026v1} supplies is the conversion
    of~\QEC{} into ellipticity, namely the analogue of~\cite[Proposition~3.13]{DeRosa2022}; the passage from
    an~atomic-type condition to ellipticity of this kind is also available
    in~\cite[Theorem~8.8]{DeRosa2020}. The~definition of~\QEC{} used in~\cite[Definition~1]{DeRosa2026v1} is equivalent to
    that of the first version of this paper with~$N$ Lipschitzian, after replacing $N(T)$
    by~$\transpose{N(T)}$. Indeed, $B_F = F \, \transpose{P}_F$ by~\ref{def:aux-objects}; hence,
    $B_F(S) \bullet \transpose{N(T)} = F(S) \, P_F(S) \bullet N(T)$ for
    $S, T \in \pgrass{\adim}{\vdim}$. Moreover, $\|\transpose{N(T)}\| = 1$ and
    $1 \le |N(T)| \le (\adim - \vdim)^{1/2}$, because $\im N(T) \subseteq \im T^{\perp}$
    by~\ref{def:QEC}\eqref{eq:QEC_normal}; thus either norm may be used to normalise~$N$. The first
    version additionally requires
    $N(T) \circ \transpose{P}_F(T) = 0 = \transpose{P}_F(T) \circ N(T)$; for~$F$ of
    class~$\cnt{2}$ both identities follow from~\ref{def:QEC}\eqref{eq:QEC_normal}, the first
    by~\ref{lem:QEC_compatibility} and the second since
    $\im N(T) \subseteq \im T^{\perp} = \ker \transpose{P}_F(T)$.
\end{remark}

\section{Failure of the atomic condition for strictly polyconvex integrands}
\label{sec:AC1_failure}

\begin{miniremark}
    \label{mr:AC1_failure_overview}
    By~\cite[Theorem~1.16]{DeRosa2025}, every \AC{} integrand of class~$\cnt{1}$ is strictly
    polyconvex. We show that the converse fails. For $\vdim = 2$ and $\adim = 5$, we construct an
    inner-product norm~$\varphi$ on~$\extpower{2}\R^{5}$ whose associated integrand
    violates~\ref{def:AC}\ref{i:AC:dim}, i.e., the condition (AC1)
    of~\cite[Definition~3.1]{DeRosa2022}; see~\ref{thm:AC1_fails}. A~suspension extends the example
    to all $\vdim \ge 2$ and $\adim - \vdim \ge 3$; see~\ref{thm:AC1_fails_all_codim}. The failure
    occurs already in~\ref{def:AC}\ref{i:AC:dim} and not only in the rigidity
    part~\ref{def:AC}\ref{i:AC:dirac}. These integrands do not depend on the point, so
    \cite[Theorem~1.2(ii)]{DePhilippis2018} applies and produces a~varifold with locally bounded
    anisotropic first variation and positive lower density which is not $\vdim$-rectifiable. That
    varifold is supported on the linear subspace $\im A(\mu)$, which for the measure~$\mu$
    of~\ref{thm:AC1_fails} is a~line; cf.~\cite[section~4, Step~2]{DePhilippis2018}.
    In~particular, the
    codimension-one characterisation~\cite[Theorem~1.3]{DePhilippis2018} does not extend to
    $2 \le \vdim \le \adim - 3$.
\end{miniremark}

\begin{miniremark}
    \label{rem:polarity}
    Denote by $\beta : \R^{\adim} \to \aforms{1}{\R^{\adim}}$ the standard polarity. We~equip
    $\extpower{j} \R^{\adim}$ and $\aforms{j} \R^{\adim}$ with the standard Euclidean structures as
    in~\cite[1.7.5]{Federer1969}. Note that, by~\cite[1.5.1]{Federer1969}, we have
    \begin{displaymath}
        ( \zeta \restrict \beta(u) ) \bullet w = \zeta \bullet (u \wedge w)
        \quad \text{for $\zeta \in \extpower{2}\R^{\adim}$ and $u,w \in \R^{\adim}$} \,.
    \end{displaymath}
\end{miniremark}

\begin{remark}
    \label{rem:BF_form_in_2d}
    Recall~\ref{mr:rank_one_maps} and~\ref{rem:polarity}. Suppose
    \begin{gather}
        \text{$F$ is a~polyconvex $2$-integrand associated to}
        \\
        \text{a~norm $\varphi : \extpower{2}\R^{\adim} \to \R$ of class~$\cnt{1}$ away from~$0$} \,,
        \\
        T \in \pgrass{\adim}{2} \,, \quad \text{$u, v$ is an orthonormal basis of~$\im T$} \,, \quad
        \xi = u \wedge v \,, \quad \eta = \grad \varphi(\xi) \,.
    \end{gather}
    Then
    \begin{displaymath}
        B_F(T) = \beta(v) \cdot (\eta \restrict \beta(u))
        - \beta(u) \cdot (\eta \restrict \beta(v)) \,.
    \end{displaymath}
    Indeed, if $a, b \in \R^{\adim}$ and $L = \beta(b) \cdot a$, then
    $\varkappa(L)\xi = (b \bullet u)\, a \wedge v + (b \bullet v)\, u \wedge a$ and
    \begin{multline}
        \eta \bullet \varkappa(L)\xi
        = (b \bullet u)\, \eta \bullet (a \wedge v)
        + (b \bullet v)\, \eta \bullet (u \wedge a)
        \\
        = - (b \bullet u) \bigl( \eta \restrict \beta(v) \bigr) \bullet a
        + (b \bullet v) \bigl( \eta \restrict \beta(u) \bigr) \bullet a
        \\
        = \bigl( \beta(v) \cdot (\eta \restrict \beta(u))
        - \beta(u) \cdot (\eta \restrict \beta(v)) \bigr) \bullet L \,;
        \label{eq:BF2d:identity}
    \end{multline}
    the maps $\beta(b) \cdot a$ span~$\End{\R^{\adim}}$, so~\eqref{eq:BF2d:identity} holds for
    $L \in \End{\R^{\adim}}$, and the claim follows from~\ref{lem:BF_via_varkappa}. Note
    that~\eqref{eq:BF2d:identity} holds for any $\eta \in \extpower{2}\R^{\adim}$.
\end{remark}

\begin{remark}
    \label{rem:rotation_contraction}
    If $R \in \orthgroup{\adim}$, then $\extpower{2}R$ is a~linear isometry of
    $\extpower{2}\R^{\adim}$ by~\cite[1.7.5]{Federer1969} and
    \begin{displaymath}
        ( \extpower{2}R\, \zeta ) \restrict \beta(Ru)
        = R \bigl( \zeta \restrict \beta(u) \bigr)
        \quad \text{for $u \in \R^{\adim}$ and $\zeta \in \extpower{2}\R^{\adim}$} \,.
    \end{displaymath}
    Indeed, for $y \in \R^{\adim}$,
    \begin{multline}
        ( \extpower{2}R\, \zeta ) \restrict \beta(Ru) \bullet y
        = \extpower{2}R\, \zeta \bullet (Ru \wedge y)
        \\
        = \zeta \bullet (u \wedge R^{-1}y)
        = \bigl( \zeta \restrict \beta(u) \bigr) \bullet R^{-1} y
        = R \bigl( \zeta \restrict \beta(u) \bigr) \bullet y \,.
    \end{multline}
\end{remark}

\begin{miniremark}
    The next lemma produces inner-product norms with prescribed values and gradients at finitely
    many orthonormal points. We~use $\Dirac{0}$ as in~\ref{mr:notation_citations}.
\end{miniremark}

\begin{lemma}
    \label{lem:shear_norm}
    Suppose
    \begin{gather}
        \text{$X$ is a~Euclidean space} \,,
        \quad
        p \in \natp \,,
        \quad
        \text{$\zeta_1,\ldots,\zeta_p \in X$ are orthonormal} \,,
        \quad
        \eta_1,\ldots,\eta_p \in X \,,
        \\
        \eta_i \bullet \zeta_j = \Dirac{0}\{i-j\}
        \quad \text{for $i,j \in \{1,\ldots,p\}$} \,,
        \\
        E = \textsum{i=1}{p} \beta(\zeta_i) \cdot (\eta_i - \zeta_i) \in \End{X} \,,
        \quad
        \varphi(x) = | x + \transpose{E} x |
        \quad \text{for $x \in X$} \,.
    \end{gather}
    Then the following two statements hold.
    \begin{enumerate}
    \item
        \label{i:shear:norm}
        The function~$\varphi$ is the norm induced by the inner product
        $[ (x,y) \mapsto x \bullet My ]$, where $M = (\id{X} + E) \circ (\id{X} + \transpose{E})$
        is symmetric and positive definite; in~particular, $\varphi$ is real analytic away from the
        origin and
        \begin{displaymath}
            \varphi(x + y) < \varphi(x) + \varphi(y)
            \quad \text{unless $x,y$ are non-negative multiples of a~common vector} \,.
        \end{displaymath}
    \item
        \label{i:shear:values}
        The identities $\varphi(\zeta_j) = 1$ and $\grad \varphi(\zeta_j) = \eta_j$ hold for
        $j \in \{1,\ldots,p\}$.
    \end{enumerate}
\end{lemma}

\begin{proof}
    Since $(\eta_i - \zeta_i) \bullet \zeta_j = 0$ for $i,j \in \{1,\ldots,p\}$,
    \begin{displaymath}
        \transpose{E} \zeta_j = \textsum{i=1}{p} \bigl( (\eta_i - \zeta_i) \bullet \zeta_j \bigr)
        \zeta_i = 0
        \quad \text{and} \quad
        \im \transpose{E} \subseteq \lin \{ \zeta_1, \ldots, \zeta_p \} \,;
    \end{displaymath}
    hence, $\transpose{E} \circ \transpose{E} = 0$. Set $S = \id{X} + \transpose{E}$. Then
    \begin{displaymath}
        S \circ (\id{X} - \transpose{E}) = \id{X} \,,
        \quad
        \varphi(x) = |Sx| \,,
        \quad
        \varphi(x)^2 = x \bullet (\transpose{S} \circ S) x \,,
        \quad
        \transpose{S} \circ S = M \,;
    \end{displaymath}
    thus, $M$ is symmetric and positive definite, $\varphi$ is the norm induced by~$M$, and,
    $S$ being injective, strict subadditivity of~$\varphi$ follows from that of~$|\cdot|$. This
    proves~\ref{i:shear:norm}. Since $S \zeta_j = \zeta_j$, we get $\varphi(\zeta_j) = 1$ and
    \begin{displaymath}
        \grad \varphi(\zeta_j)
        = \varphi(\zeta_j)^{-1} M \zeta_j
        = (\id{X} + E) \bigl( (\id{X} + \transpose{E}) \zeta_j \bigr)
        = (\id{X} + E) \zeta_j
        = \zeta_j + (\eta_j - \zeta_j)
        = \eta_j \,. \qedhere
    \end{displaymath}
\end{proof}

\begin{miniremark}
    \label{mr:cyclic_data}
    We fix $\adim = 5$ and $\vdim = 2$. Let $e_1,\ldots,e_5$ be the standard basis of~$\R^{5}$;
    \emph{indices are understood modulo~$5$}, i.e., $e_{j+5} = e_j$. Let $\sigma \in \orthgroup{5}$
    satisfy $\sigma e_j = e_{j+1}$, so that $\sigma^5 = \id{\R^5}$. For $i \in \{1,\ldots,5\}$ define
    \begin{gather}
        \label{eq:cyclic_atoms}
        \xi_i = e_i \wedge e_{i+1} \in \ograss{5}{2} \,,
        \qquad
        T_i = \project{( \Sp \xi_i )} \in \pgrass{5}{2} \,,
        \\
        \label{eq:cyclic_covectors}
        \eta_1
        = e_1 \wedge e_2 + 2\, e_1 \wedge e_3 + 2\, e_1 \wedge e_4
        - 2\, e_2 \wedge e_4 - 2\, e_2 \wedge e_5 \,,
        \qquad
        \eta_{i+1} = \extpower{2}\sigma \, \eta_{i} \,,
        \\
        d = e_1 + e_2 + e_3 + e_4 + e_5 \,,
        \qquad
        \mu = \tfrac{1}{5} \textsum{i=1}{5} \Dirac{T_i} \,.
    \end{gather}
    The $2$-vectors $\xi_1,\ldots,\xi_5$ are orthonormal. The coefficients of $e_1 \wedge e_2$,
    $e_2 \wedge e_3$, $e_3 \wedge e_4$, $e_4 \wedge e_5$, $e_1 \wedge e_5$ in~$\eta_1$ are $1, 0,
    0, 0, 0$; hence,
    \begin{equation}
        \label{eq:cyclic_gram}
        \eta_1 \bullet \xi_j = \Dirac{0}\{1-j\}
        \quad \text{for $j \in \{1,\ldots,5\}$} \,.
    \end{equation}
    Since $\extpower{2}\sigma$ is an isometry with $\extpower{2}\sigma\, \xi_i = \xi_{i+1}$ and
    $\extpower{2}\sigma\, \eta_i = \eta_{i+1}$,
    \begin{equation}
        \label{eq:cyclic_gram_full}
        \eta_i \bullet \xi_j = \Dirac{0}\{i-j\}
        \quad \text{for $i,j \in \{1,\ldots,5\}$} \,.
    \end{equation}
    Let $\varphi : \extpower{2}\R^{5} \to \R$ be the norm given by~\ref{lem:shear_norm} with
    \begin{displaymath}
        X = \extpower{2}\R^{5} \,,
        \quad
        p = 5 \,,
        \quad
        \zeta_i = \xi_i \,,
        \quad
        \text{and $\eta_i$ as above} \,,
    \end{displaymath}
    and let $F$ be any $2$-integrand associated to~$\varphi$; such~$F$ exists
    by~\ref{rem:integrand_defined_by_values_on_Gnk}.
\end{miniremark}

\begin{theorem}
    \label{thm:AC1_fails}
    Let $\varphi$, $F$, $\mu$, and~$d$ be as in~\ref{mr:cyclic_data}. Then $F$ is a~strictly
    polyconvex integrand associated to the inner-product norm~$\varphi$,
    \begin{displaymath}
        A(\mu) = \textint{}{} \transpose{P}_F \ud \mu
        = \tfrac{2}{5} \, \beta(d) \cdot d \,,
        \quad \text{and} \quad
        \dim \ker A(\mu) = 4 > 3 = \adim - \vdim \,.
    \end{displaymath}
    In~particular, $F$ violates~\ref{def:AC}\ref{i:AC:dim}, $F \notin \AC$, and $F$ fails the
    condition (AC1) of~\cite[Definition~3.1]{DeRosa2022}.
\end{theorem}

\begin{proof}
    Strict polyconvexity follows from~\ref{lem:shear_norm}\ref{i:shear:norm}
    and~\ref{def:polyconvex}. By~\ref{lem:shear_norm}\ref{i:shear:values}
    and~\ref{def:aux-objects},
    \begin{equation}
        \label{eq:cyclic_values}
        F(T_i) = \varphi(\xi_i) = 1 \,,
        \quad
        \grad \varphi(\xi_i) = \eta_i \,,
        \quad
        \transpose{P}_F(T_i) = B_F(T_i)
        \quad \text{for $i \in \{1,\ldots,5\}$} \,.
    \end{equation}
    By~\ref{rem:BF_form_in_2d} with $(u,v) = (e_i, e_{i+1})$,
    \begin{equation}
        \label{eq:cyclic_BF}
        B_F(T_i)
        = \beta(e_{i+1}) \cdot (\eta_{i} \restrict \beta(e_{i}))
        - \beta(e_{i}) \cdot (\eta_{i} \restrict \beta(e_{i+1}))
        \,.
    \end{equation}
    From~\ref{mr:cyclic_data}\eqref{eq:cyclic_covectors} we compute
    \begin{displaymath}
        \eta_1 \restrict \beta(e_1) = e_2 + 2 e_3 + 2 e_4
        \quad \text{and} \quad
        \eta_1 \restrict \beta(e_2) = - e_1 - 2 e_4 - 2 e_5 \,;
    \end{displaymath}
    hence, by~\ref{rem:rotation_contraction} with $R = \sigma^{i-1}$,
    \begin{displaymath}
        \eta_i \restrict \beta(e_i) = e_{i+1} + 2 e_{i+2} + 2 e_{i+3}
        \quad \text{and} \quad
        \eta_i \restrict \beta(e_{i+1}) = - e_{i} - 2 e_{i+3} - 2 e_{i+4} \,,
    \end{displaymath}
    and~\eqref{eq:cyclic_BF} becomes
    \begin{equation}
        \label{eq:cyclic_BF_explicit}
        B_F(T_i)
        = \beta(e_{i}) \cdot e_{i}
        + \beta(e_{i+1}) \cdot e_{i+1}
        + 2 \beta(e_{i+1}) \cdot e_{i+2}
        + 2 \beta(e_{i+1}) \cdot e_{i+3}
        + 2 \beta(e_{i}) \cdot e_{i+3}
        + 2 \beta(e_{i}) \cdot e_{i+4} \,.
    \end{equation}
    Summing over~$i$ and substituting $b = i+1$ in the terms involving $\beta(e_{i+1})$,
    \begin{multline}
        \label{eq:cyclic_sum}
        \textsum{i=1}{5} B_F(T_i)
        = 2 \textsum{b=1}{5} \beta(e_b) \cdot e_b
        + 2 \textsum{b=1}{5} \beta(e_b) \cdot
        \bigl( e_{b+1} + e_{b+2} + e_{b+3} + e_{b+4} \bigr)
        \\
        = 2 \textsum{a=1}{5} \textsum{b=1}{5} \beta(e_{b}) \cdot e_{a}
        = 2\, \beta(d) \cdot d \,.
    \end{multline}
    Hence, by~\eqref{eq:cyclic_values},
    \begin{displaymath}
        A(\mu) = \tfrac{1}{5} \textsum{i=1}{5} B_F(T_i) = \tfrac{2}{5}\, \beta(d) \cdot d \,,
        \quad
        \ker A(\mu) = \R^{5} \cap \{ x : d \bullet x = 0 \} \,,
        \quad
        \dim \ker A(\mu) = 4 \,.
    \end{displaymath}
    Finally, \ref{def:AC} coincides with~\cite[Definition~1.1]{DePhilippis2018}
    by~\ref{rem:AC_orig}, and~\ref{def:AC}\ref{i:AC:dim} is the condition (AC1)
    of~\cite[Definition~3.1]{DeRosa2022}.
\end{proof}

\begin{miniremark}
    The construction is robust: only the directions of $\grad \psi(\xi_i)$ matter, and the
    admissible inner-product norms form a~$15$-dimensional family.
\end{miniremark}

\begin{corollary}
    \label{cor:AC1_family}
    Suppose $\xi_i$, $\eta_i$, $\mu$, and~$d$ are as in~\ref{mr:cyclic_data},
    \begin{gather}
        \text{$\psi : \extpower{2}\R^{5} \to \R$ is a~norm of class~$\cnt{1}$ away from~$0$} \,,
        \\
        \label{eq:gradient_rays}
        \grad \psi(\xi_i) = \lambda_i \eta_i
        \quad \text{for some $\lambda_i > 0$ and each $i \in \{1,\ldots,5\}$} \,.
    \end{gather}
    Then every $2$-integrand~$G$ associated to~$\psi$ satisfies
    $A(\mu) = \tfrac25\, \beta(d) \cdot d$ and violates~\ref{def:AC}\ref{i:AC:dim}.
    Moreover, setting
    \begin{gather}
        \varphi_M(x) = (x \bullet Mx)^{1/2}
        \quad \text{for $x \in \extpower{2}\R^{5}$ and positive definite self-adjoint
        $M \in \End{\extpower{2}\R^{5}}$} \,,
        \\
        K = \End{\extpower{2}\R^{5}} \cap \bigl\{ M :
        M = \transpose{M} ,\ M \xi_i = \eta_i \text{ for } i \in \{1,\ldots,5\}
        \bigr\} \,,
        \\
        U = K \cap \{ M : \text{$M$ is positive definite} \} \,,
    \end{gather}
    the inner-product norms satisfying~\eqref{eq:gradient_rays} with
    $\lambda_1 = \cdots = \lambda_5 = 1$ are exactly $\varphi_M$ with $M \in U$;
    moreover, $K$ is an affine space, $\dim K = 15$, and $U$ is
    non-empty and relatively open in~$K$.
\end{corollary}

\begin{proof}
    Since $\psi$ is positively $1$-homogeneous, by~\ref{mr:cyclic_data}\eqref{eq:cyclic_gram_full}
    \begin{displaymath}
        G(T_i) = \psi(\xi_i) = \grad\psi(\xi_i) \bullet \xi_i = \lambda_i \, \eta_i \bullet \xi_i
        = \lambda_i \,.
    \end{displaymath}
    By~\ref{rem:BF_form_in_2d},
    \begin{displaymath}
        B_G(T_i) = \lambda_i \bigl( \beta(e_{i+1}) \cdot (\eta_i \restrict \beta(e_i))
        - \beta(e_{i}) \cdot (\eta_i \restrict \beta(e_{i+1})) \bigr) \,;
    \end{displaymath}
    hence, $\transpose{P}_G(T_i) = B_G(T_i)/G(T_i)$ equals the right-hand side
    of~\ref{thm:AC1_fails}\eqref{eq:cyclic_BF_explicit}, and the first claim follows
    from~\ref{thm:AC1_fails}\eqref{eq:cyclic_sum}.

    Next, $\grad\varphi_M(\xi_i) = \varphi_M(\xi_i)^{-1} M \xi_i$ and $\varphi_M(\xi_i)^{2} = \xi_i
    \bullet M\xi_i$; as $\eta_i \bullet \xi_i = 1$, \eqref{eq:gradient_rays} with $\lambda_i = 1$
    holds if and only if $M\xi_i = \eta_i$. Let $W = \lin \{ \xi_1, \ldots, \xi_5 \}$. Then
    \begin{multline}
        \End{\extpower{2}\R^{5}} \cap \bigl\{ M : M = \transpose{M} ,\ M \xi_i = 0
        \text{ for } i \in \{1,\ldots,5\} \bigr\}
        \\
        \cong \bigl\{ M \in \End{W^{\perp}} : M = \transpose{M} \bigr\} \,,
        \quad
        \dim W^{\perp} = 5 \,;
    \end{multline}
    thus, $K$, which is non-empty by~\ref{lem:shear_norm}, is an affine space of
    dimension $\binom{6}{2} = 15$. Positive definiteness is an open condition, and the~$M$
    of~\ref{lem:shear_norm} belongs to~$U$.
\end{proof}

\begin{miniremark}
    Wedging with a~fixed simple $(\vdim-2)$-vector in the orthogonal complement transports the
    example to higher dimensions.
\end{miniremark}

\begin{theorem}
    \label{thm:AC1_fails_all_codim}
    Suppose $\vdim, \adim \in \natp$, $\vdim \ge 2$, and $\adim - \vdim \ge 3$. Then there exists
    a~strictly polyconvex $\vdim$-integrand on~$\End{\R^{\adim}}$, associated to an inner-product
    norm on~$\extpower{\vdim}\R^{\adim}$, which violates~\ref{def:AC}\ref{i:AC:dim}.
\end{theorem}

\begin{proof}
    Let $m = \vdim - 2$, so that $\adim \ge 5 + m$, let $e_1,\ldots,e_{\adim}$ be the standard
    basis of~$\R^{\adim}$, and identify~$\R^5$ with $\lin \{ e_1,\ldots,e_5 \}$. We keep the
    notation of~\ref{mr:cyclic_data} and~\ref{thm:AC1_fails}; the convention $e_{j+5} = e_j$ applies
    only to the indices $i, i+1, \ldots, i+4$ with $i \in \{1,\ldots,5\}$, while
    $e_6, \ldots, e_{\adim}$ denote the remaining members of the standard basis. Set
    \begin{gather}
        \chi = e_6 \wedge \cdots \wedge e_{5+m} \in \extpower{m}\R^{\adim} \,,
        \quad
        \xi_i' = \xi_i \wedge \chi \,,
        \quad
        \eta_i' = \eta_i \wedge \chi \,,
        \quad
        T_i' = \project{( \Sp \xi_i' )} \,,
        \\
        P = \project{\lin \{ e_1,\ldots,e_5 \}} \,,
        \quad
        \mu' = \tfrac15 \textsum{i=1}{5} \Dirac{T_i'} \,,
    \end{gather}
    with $\chi = 1$ in case $m = 0$. By~\cite[1.7.5]{Federer1969},
    \begin{gather}
        \label{eq:suspended_gram}
        \xi_i' \in \ograss{\adim}{\vdim} \,,
        \quad
        \xi_i' \bullet \xi_j' = \xi_i \bullet \xi_j = \Dirac{0}\{i-j\} \,,
        \quad
        \eta_i' \bullet \xi_j' = \eta_i \bullet \xi_j = \Dirac{0}\{i-j\} \,,
        \\
        (\zeta \wedge \chi) \bullet (\zeta' \wedge \chi) = \zeta \bullet \zeta'
        \quad \text{for $\zeta, \zeta' \in \extpower{2} \lin \{ e_1,\ldots,e_5 \}$} \,.
    \end{gather}
    Let $\varphi'$ be the norm given by~\ref{lem:shear_norm} with
    \begin{displaymath}
        X = \extpower{\vdim}\R^{\adim} \,,
        \quad
        p = 5 \,,
        \quad
        \zeta_i = \xi_i' \,,
        \quad
        \eta_i' \text{ in place of } \eta_i \,;
    \end{displaymath}
    let $F'$ be any
    $\vdim$-integrand associated to~$\varphi'$; it exists
    by~\ref{rem:integrand_defined_by_values_on_Gnk} and is strictly polyconvex
    by~\ref{lem:shear_norm}\ref{i:shear:norm} and~\ref{def:polyconvex}. As
    in~\ref{thm:AC1_fails}\eqref{eq:cyclic_values},
    \begin{displaymath}
        F'(T_i') = 1 \,,
        \quad
        \grad \varphi'(\xi_i') = \eta_i' \,,
        \quad
        \transpose{P}_{F'}(T_i') = B_{F'}(T_i') \,.
    \end{displaymath}
    Fix $i \in \{1,\ldots,5\}$ and $L \in \End{\R^{\adim}}$. By~\ref{lem:BF_via_varkappa}
    and~\ref{rem:varkappa_derivation},
    \begin{equation}
        \label{eq:suspension_split}
        B_{F'}(T_i') \bullet L
        = (\eta_i \wedge \chi) \bullet
        \bigl( (\varkappa(L)\xi_i) \wedge \chi \bigr)
        + (\eta_i \wedge \chi) \bullet
        \bigl( \xi_i \wedge \varkappa(L)\chi \bigr) \,.
    \end{equation}
    By~\cite[1.7.5]{Federer1969}, a~basis $\vdim$-vector is orthogonal to $\eta_i \wedge \chi$
    unless it equals $\pm (e_a \wedge e_b) \wedge \chi$ with $a,b \in \{1,\ldots,5\}$.

    Since
    $\varkappa(L)\xi_i - \varkappa(P \circ L)\xi_i = ((P^{\perp}\circ L)e_i) \wedge e_{i+1} + e_i \wedge
    ((P^{\perp}\circ L)e_{i+1})$ and $(P^{\perp}\circ L)e_i, (P^{\perp}\circ L)e_{i+1} \in \lin \{ e_6, \ldots,
    e_{\adim} \}$, the difference contributes nothing to the first summand
    of~\eqref{eq:suspension_split}. Regard $B_F(T_i)$, given
    by~\ref{thm:AC1_fails}\eqref{eq:cyclic_BF_explicit}, as an element of~$\End{\R^{\adim}}$;
    then $P \circ B_F(T_i) = B_F(T_i)$. Hence, by~\eqref{eq:suspended_gram},
    \ref{rem:BF_form_in_2d}\eqref{eq:BF2d:identity}, and~\ref{thm:AC1_fails}\eqref{eq:cyclic_BF},
    \begin{displaymath}
        (\eta_i \wedge \chi) \bullet \bigl( (\varkappa(L)\xi_i) \wedge \chi \bigr)
        = \eta_i \bullet \varkappa(P \circ L)\xi_i
        = B_F(T_i) \bullet (P \circ L)
        = B_F(T_i) \bullet L \,.
    \end{displaymath}

    We have
    $\varkappa(L)\chi = \textsum{j=6}{5+m} e_6 \wedge \cdots \wedge L e_j \wedge \cdots
    \wedge e_{5+m}$. When $L e_j$ is expanded in the standard basis, only the component
    $(e_j \bullet L e_j)\, e_j$ survives the pairing with $\eta_i \wedge \chi$: other
    components produce a~repeated factor, three factors from $\{ e_1, \ldots, e_5 \}$, or a~factor
    $e_a$ with $a > 5 + m$. Consequently, by~\eqref{eq:suspended_gram},
    \begin{displaymath}
        (\eta_i \wedge \chi) \bullet \bigl( \xi_i \wedge \varkappa(L)\chi \bigr)
        = \textsum{j=6}{5+m} (e_j \bullet L e_j) \,
        (\eta_i \wedge \chi) \bullet (\xi_i \wedge \chi)
        = \Bigl( \textsum{j=6}{5+m} \beta(e_j) \cdot e_j \Bigr) \bullet L \,.
    \end{displaymath}

    By~\ref{thm:AC1_fails}\eqref{eq:cyclic_sum},
    \begin{gather}
        \label{eq:suspended_BF}
        B_{F'}(T_i')
        = B_F(T_i) + \textsum{j=6}{5+m} \beta(e_j) \cdot e_j \,,
        \\
        \label{eq:suspended_sum}
        A(\mu')
        = \tfrac15 \textsum{i=1}{5} B_{F'}(T_i')
        = \tfrac25 \, \beta(d) \cdot d + \textsum{j=6}{5+m} \beta(e_j) \cdot e_j \,.
    \end{gather}
    The endomorphism $A(\mu')$ is symmetric and $\im A(\mu') = \lin \{ d, e_6, \ldots, e_{5+m} \}$;
    hence,
    \begin{displaymath}
        \dim \ker A(\mu') = \adim - (m+1) = \adim - \vdim + 1 > \adim - \vdim \,,
    \end{displaymath}
    and~\ref{def:AC}\ref{i:AC:dim} fails for~$F'$.
\end{proof}

\begin{remark}
    \label{rem:AC1_fails_discussion}
    We collect a~few comments on~\ref{thm:AC1_fails} and~\ref{thm:AC1_fails_all_codim}.
    \begin{enumerate}
    \item
        Together with~\cite[Theorem~1.16]{DeRosa2025}, \ref{thm:AC1_fails_all_codim} shows that
        for $2 \le \vdim \le \adim - 3$ the atomic condition is \emph{strictly} stronger than
        strict polyconvexity, even for integrands associated to inner-product norms;
        compare~\cite[Theorem~1.3]{DePhilippis2018}. For $\vdim \in \{1, \adim-1\}$ every
        $\vdim$-vector is simple and the associated norm is determined by~$F$. The case
        $\adim - \vdim = 2$ remains open.
    \item
        There is no conflict with~\cite[Proposition~3.8 and Corollary~3.12]{DeRosa2022}: these
        provide a~$\cnt{1}$~neighbourhood of the \emph{area} integrand on which (AC1) holds,
        whereas the integrand of~\ref{thm:AC1_fails} is far from the area integrand. Indeed,
        $\varphi = |\cdot|$ would force $\grad \varphi(\xi_i) = \xi_i$, while
        $|\grad\varphi(\xi_i) - \xi_i| = |\eta_i - \xi_i| = 4$.
    \item
        By~\cite[Theorem~6.1]{DeRosa2022}, on~$\pgrass{4}{2}$ the integrands associated to the
        $\ell^p$~norms in Pl{\"u}cker coordinates, $1 < p < \infty$, satisfy (AC1); other
        pairs $(\adim,\vdim)$ are asked about in~\cite[Remark~6.4]{DeRosa2022}. By~\cite[Theorem~1.3
          and Remark~2]{DeRosa2026v1} the $\ell^q$~integrands satisfy \USAC{}, hence~\AC{}, on
        every $\pgrass{\adim}{\vdim}$ whenever
        $2 - \tfrac{7}{10 \min \{ \vdim, \adim - \vdim \}} \le q \le 2$, and fail~\SAC{} for
        $q > 4$ and $2 \le \vdim \le \adim - 2$. Theorem~\ref{thm:AC1_fails_all_codim} shows that for
        $\adim - \vdim \ge 3$ no such result can follow from convexity properties of the norm
        alone: the coordinate structure of the $\ell^q$~norms is essential.
    \item
        Numerical experiments conducted with Claude suggest that the failure set has non-empty
        interior in the $\cnt{2}$~topology on integrands: at the configuration
        of~\ref{mr:cyclic_data} the natural obstruction map is, numerically, a~submersion onto
        the normal space of the variety of endomorphisms of rank at most~$1$. On~$\pgrass{4}{2}$,
        in contrast, analogous configurations degenerate: writing
        $\sigma_1(A) \ge \sigma_2(A) \ge \cdots \ge \sigma_{\adim}(A) \ge 0$ for the singular
        values of $A \in \End{\R^{\adim}}$, i.e., the eigenvalues of
        $(\transpose{A} \circ A)^{1/2}$ arranged in non-increasing order, one has
        $\dim \im A \le 1$ if and only if $\sigma_2(A) = 0$, so in case $\vdim = 2$ the
        scale-invariant quotient $\sigma_2(A(\mu))/\sigma_1(A(\mu))$ measures the proximity
        of~$\mu$ to a~violation of~\ref{def:AC}\ref{i:AC:dim}; minimising it over measures and
        inner-product norms drives the norm to a~degenerate limit. These observations are not
        proved here.
    \end{enumerate}
\end{remark}

\subsection*{Tool and computational resource disclosure}
\addcontentsline{toc}{section}{\numberline{}Tool and computational resource disclosure}

The first version of this paper was written entirely by the authors. By May 2026 it contained the
material of sections~\ref{sec:basic_notation}--\ref{sec:polyconvex_integrands} and early forms of
sections~\ref{sec:uniform_AC} and~\ref{sec:QEC_stability}, including the quadratic exposed
condition, its stability, and the inclusion $\QEC{} \subseteq \CI{}$. Since then, Claude Opus and
Sonnet models have been used extensively for checking grammar and finding typos, browsing the
literature in search of references, proof-reading and correcting parts written by the authors,
verifying conjectures, performing numerical experiments, drafting proofs, and improving the
introduction. The ideas and conjectures behind the results are due to the authors. In some cases
an~AI model worked out the formal proof of a~statement conjectured by the authors; this applies to
the inclusion $\CI{} \subseteq \QEC{}$ in~\ref{thm:CI_implies_LQ}--\ref{cor:CI_iff_QEC} and to the
dichotomy in~\ref{ex:J_fields}\ref{i:Jf:dichotomy}. Numerical experiments suggested that strict
polyconvexity might not imply~\AC{}. The counterexample of~\ref{thm:AC1_fails} follows an~idea of
the authors. Claude carried out the construction and wrote the first draft of
section~\ref{sec:AC1_failure}. The authors verified the construction and its proof, then revised
the section. Most of the text went through several iterations of rewriting and
editing, carried out alternately by AI and by the authors. ChatGPT with the Astra-6 model was also
used for proof-reading and error correction. The authors have checked all the content carefully and
take full responsibility for~it.

\subsection*{Acknowledgements}
\addcontentsline{toc}{section}{\numberline{}Acknowledgements}

The~research of Sławomir Kolasiński and Jacek Jakimiuk was financed by the
\href{https://ncn.gov.pl/}{National Science Centre Poland} grant number 2022/46/E/ST1/00328. Part of
this work was done while S.K. was visiting Bernd Kirchheim and Riccardo Tione at the University of
Leipzig, benefiting from their kind hospitality.

\bigskip

{\small
  \addcontentsline{toc}{section}{\numberline{}References}
  \bibliography{extract}
  \bibliographystyle{halpha}
}

\bigskip

{\small \noindent
  Jacek Jakimiuk
  \\
  Instytut Matematyki,
  Uniwersytet Warszawski
  \\
  ul. Banacha 2, 02-097 Warszawa, Poland
  \\
  \texttt{j.jakimiuk@mimuw.edu.pl}
}

\bigskip

{\small \noindent
  Sławomir Kolasiński
  \\
  Instytut Matematyki,
  Uniwersytet Warszawski
  \\
  ul. Banacha 2, 02-097 Warszawa, Poland
  \\
  \texttt{s.kolasinski@mimuw.edu.pl}
}

\bigskip

{\small \noindent
  Maciej Leśniak
  \\
  \texttt{maciej.lesniak0@gmail.com}
}

\end{document}